\documentclass[]{amsart}
\usepackage[bottom]{footmisc}
\usepackage{xr}
\usepackage[dvipsnames,table,xcdraw]{xcolor}

\calclayout

\usepackage[colorlinks=true]{hyperref}
	\hypersetup{urlcolor=RubineRed,linkcolor=RoyalBlue,citecolor=ForestGreen}

\usepackage{cooltooltips}
\usepackage{graphicx}

\usepackage{forest,adjustbox}
\usepackage{forest}
\usetikzlibrary{shapes.geometric,decorations.markings}
\usetikzlibrary{positioning,arrows.meta,decorations.pathreplacing}

\forestset{%
  my tree/.style={
    for tree={
      circle,
      draw,
      inner sep=1pt,
      l sep'=5mm,
      l'=5mm,
      s sep'=2mm,
    },
  }
}
\usepackage{caption} 
			
	\newtheorem{theorem}{Theorem}[subsection]
	\newtheorem{lemma}[theorem]{Lemma}
	\newtheorem{proposition}[theorem]{Proposition}
	\newtheorem{corollary}[theorem]{Corollary}
	\newtheorem*{theorem*}{Theorem}
	\newtheorem*{lemma*}{Lemma}
	\newtheorem*{proposition*}{Proposition}
	\newtheorem*{corollary*}{Corollary}
		\theoremstyle{definition}
		
		\newtheorem{definition}[theorem]{Definition}
		\newtheorem{example}[theorem]{Example}

		\newtheorem{remark}[theorem]{Remark}

\numberwithin{equation}{section}

\usepackage[bottom]{footmisc}
\usepackage{tikz}
\usepackage{tikz-cd}
\usepackage{mathrsfs}
\usepackage{supertabular}
\usepackage[shortlabels]{enumitem}
\usetikzlibrary{lindenmayersystems}
\DeclareMathOperator{\Gal}{Gal}
\DeclareMathOperator{\GL}{GL} 
\DeclareMathOperator{\SL}{SL}

\DeclareMathOperator{\Spec}{Spec} 
\DeclareMathOperator{\Tr}{Tr} 
\DeclareMathOperator{\Hom}{Hom} 
\DeclareMathOperator{\End}{End} 

\newcommand{\Gm}{\mathds{G}}
\newcommand{\B}{\mathbf{B}}

\newcommand{\p}{\mathfrak{p}}

\newcommand{\cF}{\mathcal{F}}

\newcommand{\ra}{\rightarrow}

\newcommand{\hra}{\hookrightarrow}

\renewcommand{\P}{\mathbf{P}}

\newcommand{\cU}{\mathcal{U}}

\newcommand{\Ad}{\text{Ad}}

\renewcommand{\O}{\mathcal{O}}

\DeclareFontEncoding{OT2}{}{} 

\renewcommand{\H}{\mathbf{H}}

\DeclareMathOperator{\Fr}{Fr}
\DeclareMathOperator{\Ver}{Ver}
\DeclareMathOperator{\Norm}{Norm}

\DeclareMathOperator{\SU}{\mathbf{SU}}

\DeclareMathOperator{\Hbf}{\mathbf{H}}

\DeclareMathOperator{\Zbf}{\mathbf{Z}}

\DeclareMathOperator{\diag}{diag}

\DeclareMathOperator{\Res}{\text{Res}}

\DeclareMathOperator{\Sh}{Sh}
\DeclareMathOperator{\Stab}{Stab}

\DeclareMathOperator{\der}{der}
\DeclareMathOperator{\ad}{ad}
\DeclareMathOperator{\Art}{Art}

\usepackage{dsfont}
\usepackage{upgreek}
\usepackage{stmaryrd}
\usepackage{color}
\usepackage{amsthm,amssymb,amsmath}
\usepackage{lastpage}
\usepackage{centernot}
\usepackage{graphicx}
\usetikzlibrary{cd}
\usetikzlibrary{trees}
\usetikzlibrary{fit}

\tikzset{%
  highlight/.style={rectangle,rounded corners,fill=blue!15,draw,fill opacity=0.5,thick,inner sep=0pt}
}
\newcommand{\tikzmark}[2]{\tikz[overlay,remember picture,baseline=(#1.base)] \node (#1) {#2};}
\newcommand{\Highlight}[1][submatrix]{%
    \tikz[overlay,remember picture]{
    \node[highlight,fit=(left.north west) (right.south east)] (#1) {};}
}
\usepackage{nicematrix}
\usepackage{url}
\usetikzlibrary{hobby,backgrounds,calc,trees}
\usetikzlibrary{matrix,arrows,decorations.pathmorphing}

\def\1{\mathds{1}}
\def\Nm{\mathbf{N}}
\def\cC{\mathcal{C}}

\def\cO{\mathcal{O}}
\def\cH{\mathcal{H}}
\def\cX{\mathcal{X}}
\def\cY{\mathcal{Y}}

\def\cZ{\mathcal{Z}}
\def\cK{\mathcal{K}}

\def\ff{\mathfrak{f}}
\def\fc{\mathfrak{c}}

\def\A{\mathds{A}}
\def\F{\mathds{F}}
\def\Z{\mathds{Z}}
\def\N{\mathds{N}}
\def\Q{\mathds{Q}}
\def\R{\mathds{R}}
\def\C{\mathds{C}}

\def\T{\mathbf{T}}

\def\U{\mathbf{U}}
\def\G{\mathbf{G}}
\def\H{\mathbf{H}}

\def\Hom{\text{Hom}}

\def\Spec{\text{Spec}}
\def\Gal{\text{Gal}}

\DeclareMathOperator\Pic{Pic}

\newcommand{\authnote}[2][]{\noindent {\if!#1!  {\bf TODO} \else {\small \bf #1} \fi: #2}}

\newcounter{tasknumber}[subsection]
\newcommand{\task}[2][]{%
  \addtocounter{tasknumber}{1}%
  \begin{center}%
  \framebox[1.1\width]{\begin{minipage}{0.9\textwidth}%
  \textbf{Task \arabic{tasknumber}} \textit{\if!#1(unassigned)!\else (#1)\fi}: {#2}%
  \end{minipage}}%
  \end{center}%
}

\newcounter{assumptionnumber}
\newcommand{\assumption}[2][]{%
  \addtocounter{assumptionnumber}{1}%
  \begin{center}%
  \framebox[1.1\width]{\begin{minipage}{0.9\textwidth}%
  \textbf{Assumption \arabic{assumptionnumber}} \textit{\if!#1!\else (#1)\fi}: {#2}%
  \end{minipage}}%
  \end{center}%
}

\usepackage[scr=boondoxo,scrscaled=1.05]{mathalfa}
\usepackage[intoc,refpage]{nomencl}
\def\@@@nomenclature[#1]#2#3{%
\def\@tempa{#2}\def\@tempb{#3}%
\protected@write\@glossaryfile{}%
{\string\glossaryentry{#1\nom@verb\@tempa @[{\nom@verb\@tempa}]%

nompageref{\begingroup\nom@verb\@tempb\protect\nomeqref{\theequation}}}%
{\thepage}}%
 \endgroup
 \@esphack}

\usepackage[intoc,refpage]{nomencl}
\def\@@@nomenclature[#1]#2#3{%
\def\@tempa{#2}\def\@tempb{#3}%
\protected@write\@glossaryfile{}%
{\string\glossaryentry{#1\nom@verb\@tempa @[{\nom@verb\@tempa}]%

nompageref{\begingroup\nom@verb\@tempb\protect\nomeqref{\theequation}}}%
{\thepage}}%
 \endgroup
 \@esphack}

\usepackage{ifthen}
\renewcommand{\nomgroup}[1]{%
\ifthenelse{\equal{#1}{A}}{\item[\textbf{Chapter II \S 1 \& \S 2}]}{%
\ifthenelse{\equal{#1}{B}}{\item[\textbf{Chapter II \S 3}]}{%
\ifthenelse{\equal{#1}{C}}{\item[\textbf{Chapter III}]}{
\ifthenelse{\equal{#1}{D}}{\item[\textbf{Chapter IV}]}{
\ifthenelse{\equal{#1}{E}}{\item[\textbf{Chapter V}]}{
\ifthenelse{\equal{#1}{F}}{\item[\textbf{Chapter VI}]}{
\ifthenelse{\equal{#1}{G}}{\item[\textbf{Chapter VII}]}{
\ifthenelse{\equal{#1}{H}}{\item[\textbf{Chapter V}]}{
}}}}}}}}}
\makenomenclature

\usepackage{apptools}
\AtAppendix{\counterwithin{theorem}{section}}

\usepackage{comment}
\usepackage{aligned-overset}
  
\begin{document}

\title{\textsc{General Horizontal Norm Compatible Systems}}


\author{Reda Boumasmoud}
\address{Institut de Mathematiques de Jussieu-Paris Rive Gauche. 75005 Paris, France}
\curraddr{}
\email{reda.boumasmoud@imj-prg.fr \& reda.boumasmoud@gmail.com}
\thanks{The author was supported by the Swiss National Science Foundation grant \#PP00P2-144658 and \#P2ELP2-191672.}

\makeatletter 
\@namedef{subjclassname@2020}{%
  \textup{2020} Mathematics Subject Classification}
\makeatother


\date{}

\begin{abstract}
The ultimate goal of the paper is to construct a novel norm-compatible family of cycles appearing in the context of Gross--Gan--Prasad cycles \cite{gan-gross-prasad} arising from Shimura varieties attached to $U(n-1,1)\hra U(n,1) \times U(n-1,1)$ for arbitrary $n$. 
\end{abstract}


 \maketitle


\tableofcontents

\section{Introduction}
Arithmetic geometry could be defined as the study of motives over number fields. It is organized around a handful of conjectures. 
One of them is due to Beilinson, Bloch and Kato and it relates the special values of the $L$-functions attached to algebraic varieties or, more generally, motives to the arithmetic invariants of these algebraic objects, such as their Bloch--Kato Selmer groups and the regulators arising from them. 
This conjecture can be interpreted as a remarkable attempt to unify a few conjectures that have been wandering in number theory's paysage for a few decades, such as the Birch--Swinnerton--Dyer conjecture and the main conjecture of Iwasawa theory. 

Euler systems are currently one of the most efficient and fruitful\footnote{There is also the method of Ribet(--Mazur--Wiles), generalized to the $\GL_2$-setting by Skinner and Urban and the method of Wiles ($"R=T"$) that yields results for $L$-functions adjoint motives.} tool 
to approach cases of Bloch--Kato type conjectures. They consist of "geometric" families of classes in the continuous Galois cohomology groups containing the relevant Bloch--Kato Selmer groups, which are related via "explicit reciprocity laws" to the special values one wishes to study: If $E$ is a number field and $T$ a $p$-representation of $\Gal(\overline{E}/E)$, 
an Euler system\footnote{I would like to avoid here the difficult task of defining what an "Euler/Kolyvagin" system is.} over $T$ is a collection of (Galois) cohomology classes in $\H^1(L,T):=\H^1(\Gal(\overline{L}/L),T)$, indexed by finite abelian fields extension $E \subset L \subset E[\infty]$, for some fixed infinite abelian extension $E[\infty]$.  
These classes are constrained by a set of compatibility conditions for norm application $\H^1(L,T) \to \H^1(E,T)$ involving appropriate Euler factors between them. 
In general, the existence of a non-trivial Euler System adds just enough "rigidity" on the relevant Bloch--Selmer groups to capture its structure.

There are three major types norm-compatibility relations: horizontal (or tame), vertical and congruence relations. 
Horizontal norm relations aim to model the local L-factor of some L-function
(e.g., the L-function of a modular form or an elliptic curve) using cohomological data.  
Vertical norm relations
refer to compatibility of cohomology classes in $\Z_p$-extensions of $E$. 
They are key ingredient in the setting of Iwasawa theory; they yield upper bounds for the size of Selmer groups of $T$ (by a theorem of Rubin which generalizes the work of Kolyvagin) and so plays an fundamental part for the proven case of the Iwasawa main conjecture. 

Assume that $T$ is a group of $p$-adic cohomology étale $\H^{r-1}(X_{\overline{E}},\Q_p(m))$, where $X$ is an algebraic variety over $E$ and $m\ge r/2$. One approach to obtain Euler systems is by means of the Hochschild--Serre spectral sequence
$$E_2^{pq} \to \H^{p}(E,\H^{q}(X_{\overline{E}},\Q_p(m)))\Rightarrow \H^{p+q}(X,\Q_p(m))$$
which links the étale cohomolgy of $X_{\overline{E}}$ with Jannsen' continuous cohomology group $\H^{r}(X,\Q_p(m))$. 
First, we construct classes in the cohomology $\H^{r-1}(X,\Q_p(m))$\footnote{For example, if $m=r/2$, classes of algebraic cycles over $X$ that are defined over abelian extensions of $E$ and of correct dimension}. If these classes are in the kernel $\cF^1$ of the boundary homomorphism 
$\H^{r}(X,\Q_p(m)) \to \H^{r}(\overline{X},\Q_p(m))$, then one can take their images by the differential $\cF^1 \to  \H^{1}(E,\H^{r-1}(X_{\overline{E}},T))$ and the verification of the compatibility conditions on these classes is reduced to similar compatibilities for their corresponding algebraic cycles.

Whereas it is very hard to construct and control algebraic cycles of a general variety, the situation is much more pleasant for Shimura varieties. In this case, we have in hands an abundant source of malleable cycles, these are the special cycles obtained from Hecke translates of connected component of Shimura subvarieties of smaller dimensions.

A fundamental historical example is Kolyvagin's Heegner points Euler system, which consist of the family of classes in the cohomology of elliptic curves arising from Heegner points. This construction arises from the embeddings of the tori associated to an imaginary quadratic field $\cK$ into forms of the group $\GL_2$ and generates a collection of points on modular (or Shimura) curves, which are defined over ring class fields of $\cK$ and which can be mapped to elliptic curves via modular parametrizations. 
Kolyvagin 's pioneer method exploited the horizontal relations of this Euler system to give, combined with the Gross--Zagier formula\footnote{and some analytic results due to Murty--Murty and Bump--Friedberg--Hoffstein
\cite{bump-friedberg-hoffstein,murty-murty}} \cite{kolyvagin:euler_systems,gross-zagier}, the strongest evidence towards the Birch and Swinerton-Dyer conjecture.

Classical Euler systems have been fairly thoroughly exploited already in late 1980's and in 1990's.
Unfortunately, the list of constructed ones was rather short \cite{kolyvagin:euler_systems,kato04}. 
But in recent years, there has been a surge of new candidates proposed by various people (for instance; Cornut \cite{Cornut2018}, Jetchev \cite{jetchev:unitary}, Boumasmoud--Brooks--Jetchev \cite{BBJ18,BBJ16}, Loeffler--Zerbes--Lei \cite{LLZ14,LLZ17}, Loeffler--Zerbes--Kings \cite{KLZ15,KLZ17} and, Loeffler--Zerbes--Skinner \cite{LSZ17}). Judging by the success of their
predecessors, one expects that these new Euler Systems should soon yield significant new developments in arithmetic geometry.

Establishing norm compatibility relations for the early Euler systems constructed out of Heegner points or Siegel units was a pleasant exercise in CM theory and modular functions. 
Yet, the situation changed dramatically for the new applicants to the point that most of the difficulty for obtaining constructions coming from subvarieties of higher-dimensional Shimura varieties seems to be now concentrated in this new bottleneck. 
The sought for relations should involve the Euler factor that was first defined by Langlands, namely, the Hecke polynomial attached to the minuscule cocharacter appearing in the definition of the ambient Shimura variety. 
Gradually, this has become an underground subfield of its own, with a handful of experts with strong connections and shared preoccupations with the larger community of people working in $p$-adic L-functions, starving for new methods that would yield easier - or at least more conceptual - proofs of these relations. 
Such methods would be especially welcome in the lack of general strategy to establish the desired distribution relations for a given collection of special cycles on a Shimura variety. 
This paper attempts to contribute to this part of the theory by proposing a general construction (using $\mathbb{U}$-operators) which emerges from the proof of the main theorem as explained in \S \ref{generalrecipe}.

The ultimate goal of the paper is to construct a novel norm-compatible family of cycles appearing in the context of Gross--Gan--Prasad cycles \cite{gan-gross-prasad} arising from Shimura varieties attached to $U(n-1,1)\hra U(n,1) \times U(n-1,1)$ for arbitrary $n$. These cycles were first considered by Jetchev \cite{jetchev:unitary} in the simplest non-classical case $n = 2$ and further developed by Boumasmoud, Brooks and Jetchev in \cite{BBJ16} and \cite{BBJ18}. 

\subsection{From classical Heegner points towards a general construction}\label{classical}
In order to outline the strategy of the proof, let us first briefly review the classical case of heegner points on a modular curve.

Let $N$ be an integer and $E/\Q$ be an imaginary quadratic field with ring of integers $\cO_{E}$. Assume that all primes of $N$ split in $E$ and let $\mathcal{N}$ be an ideal of $\cO_{E}$ of norm $N$. If $m$ is prime to $N$, the isogeny $\C/\cO_{m} \to \C/(\mathcal{N} \cap \cO_{m})^{-1}$ corresponds to a Heegner point $x_{m}$ in $X_0(N)(E[m])$, where $E[m]$ denotes the ring class field of conductor $m$ and $\cO_m = \Z + m \cO_E$ is the corresponding order of $E$. Set $\text{CM}_E:=\{x_m\colon m \text{ prime to }N\}$. The points in $\text{CM}_E$ are related by the following norm-compatibilities \cite[Proposition 3.10]{darmon:ratpoints}:
\begin{proposition}[\textbf{\textit{Distribution relations}}]
Let $m$ be an integer and $\ell$ a prime which is unramified over $E$. We also suppose that $m\ell$ is prime to $N$. In this case, we have
\begin{enumerate}
\item[(i)]\textbf{Tame relations:} Let $\lambda$ be a prime of $E$ that lies over $\ell$. If $\ell \nmid m$, then
\begin{equation}\label{tameheegner}\Tr_{E[m\ell]/E[m]} x_{m\ell}=\left(T_\ell - \Fr_\lambda- \left( \frac{\mathfrak{d}_E}{\ell}\right)\Fr_\lambda^{-1}\right) x_m,\end{equation}
where, $T_\ell$ denotes the Hecke operator corresponding to $[\GL_2(\Z_\ell) \text{diag}(\ell,1) \GL_2(\Z_\ell)]$ and $\Fr_\lambda\in \Gal(E[\infty]/E)$ denotes the geomteric Frobenius.
\item[(ii)]\textbf{Vertical relations:} if $\ell \mid m$, then
\begin{equation}\label{verticalheegner} 
\Tr_{E[m\ell]/E[m]} x_{m\ell}=T_\ell x_m-x_{m/\ell}.
\end{equation}
\end{enumerate}
\end{proposition}
Now, consider the Hecke polynomial
 $$H_{\ell}(X)=X^2-T_\ell X +\ell S_\ell,$$
 where, $S_\ell$ denotes the Hecke operator corresponding to the double coset $[\GL_2(\Z_\ell) \text{diag}(\ell,\ell) \GL_2(\Z_\ell)]$.
There exists an operator $\cU_\ell\in \text{End}_\Z\Z[\text{CM}_E]$ (see Figure \ref{U-opGL2} and \cite[Introduction]{UoperatorsII2021}), a variant of the combinatorial "successor" operator\footnote{It may be compared with Cornut--Vatsal operator  $T_P^u$ in \cite[6.3]{cornut-vatsal:durham}.}, verifying the properties below:
\begin{enumerate}[noitemsep,partopsep=0pt,topsep=0pt,parsep=0pt]
\item[i.] \textbf{The Hecke side:}
\begin{equation}\label{heckesideheegner} H_\ell(\cU_\ell)=0 \text{ in } \text{End}_{\Z}\Z[\text{CM}_E].\end{equation}
\item[ii.] \textbf{The Galois side:} Let $\ell \nmid m$. For $s\ge 1$, we have
\begin{equation}\label{galoissideheegner} \Tr_{E[m\ell^{s+1}]/E[m\ell^{s}]}x_{m\ell^{s+1}}=\cU_\ell x_{m\ell^s},\end{equation}
\item[iii.] \textbf{The Congruence side:} If $\ell \nmid m$, then \begin{equation}\label{congruencesideheegner}(\cU_\ell-\ell\Fr_\lambda) x_m \equiv 0\mod \left(\ell-\left( \frac{\mathfrak{d}_E}{\ell}\right)\right) \text{ in }\Z[ \cO_{E_\ell}^\times\backslash \text{CM}_E],\end{equation}
 where $\mathfrak{d}_E$ denotes the different ideal of $E$.
 \end{enumerate}
 \begin{figure}
\centering
\begin{tikzpicture}[
  grow cyclic,
  level distance=2.5cm,	
  level/.style={
    level distance/.expanded={\ifnum#1>3 \tikzleveldistance/2\else\tikzleveldistance\fi},
    nodes/.expanded={\ifodd#1 fill=red!75   \else fill=blue!75\fi}
  },
  level 1/.style={sibling angle=60},
  level 2/.style={dashed,sibling angle=60},
  level 3/.style={solid,sibling angle=60},
  level 4/.style={solid, sibling angle=20},
  nodes={circle,draw,inner sep=+1pt, minimum size=5pt},
  ]
\path[rotate=10]
 node (L_0)[draw,top color=red!90,bottom color=blue!10,minimum size=10pt]  [above right]{}
  child  {
    node [draw,top color=red!10,bottom color=blue!90,minimum size=10pt][above right] (a){} 
    child { 
      node  [solid,draw,top color=red!90,bottom color=blue!10,minimum size=5pt] (b) {}
      child  {
        node  [draw,top color=red!10,bottom color=blue!90,minimum size=5pt] (c) {} 
          child {node [solid,draw,top color=red!90,bottom color=blue!10,minimum size=5pt] (d){}}   
          child {node [solid,draw,top color=red!90,bottom color=blue!10,minimum size=5pt] (e){}}   
          child {node [solid,draw,top color=red!90,bottom color=blue!10,minimum size=5pt] (f){}}
          child {node [solid,draw,top color=red!90,bottom color=blue!10,minimum size=5pt] (g){}}  
          child {node [solid,draw,top color=red!90,bottom color=blue!10,minimum size=5pt] (h){}} 
      }
    }
  };
  \node [draw=none,fill=none, below =1pt] at (L_0.south east) {$ x_m $};
    \node [draw=none,fill=none, below right=1pt] at (a.south east) {$x_{m\ell}$};
     \node [draw=none,fill=none, below right=1pt] at (b.south east) {$x_{m\ell^{s-1}}$};
     \node (us) [draw=none,fill=none, below =1pt] at (b.south) {};
     \node (uw) [draw=none,fill=none, left =1pt] at (b.west) {};
     \node (ua) [draw=none,fill=none, above =1pt] at (b.north) {};
     \node [draw=none,fill=none, below = 1pt] at (c.south east) {$x_{m \ell^{s}}$};
     \node (x) [draw=none,fill=none, below = 1pt,minimum size=10pt] at (d.west) {$x_{m\ell^{s+1}}$};
       \node (xw) [draw=none,fill=none, left = 1pt,minimum size=10pt] at (d.west) {};
         \node (xe) [draw=none,fill=none, right = 1pt,minimum size=10pt] at (d.west) {};
        \node (yw) [draw=none,fill=none, left= 4pt] at (e.south west) {};
         \node (ye) [draw=none,fill=none, right= 4pt] at (e.south west) {};
           \node (zw) [draw=none,fill=none,  left= 4pt,minimum size=10pt] at (f.west) {};
           \node (ze) [draw=none,fill=none,  right= 4pt,minimum size=10pt] at (f.west) {};
              \node (sw)[draw=none,fill=none,  left= 4pt,minimum size=10pt] at (g.west) {};
              \node (se) [draw=none,fill=none,  right= 4pt,minimum size=10pt] at (g.west) {};
                 \node (tw) [draw=none,fill=none, below left= 4pt,minimum size=10pt] at (h.north) {};
                 \node (te) [draw=none,fill=none,  right= 4pt,minimum size=10pt] at (h.west) {};
                  \node (ta) [draw=none,fill=none,  above= 4pt,minimum size=10pt] at (h.west) {};
                  \node [draw=none,fill=none,   right= 6pt,minimum size=10pt] at (g.north west) {\textcolor{blue}{$\mathcal{U}_\ell(x_{m\ell^{s-1}})$=Galois orbit of $x_{m\ell^{s+1}}$}};
                  \node [draw=none,fill=none,   above right= 30pt,minimum size=10pt] at (a.north east) {\textcolor{red}{The Hecke support $T_\ell x_{m\ell^{s}}$}};
       \begin{pgfonlayer}{background}
     \draw[red,fill=red,opacity=0.3](x.south) to[closed,curve through={(xw.west) .. (yw.west) ..(zw.south) .. (sw.south) .. (tw.south) .. (us.east) .. (uw.west) ..(ua.north) .. (ta.north west) ..(te.north) .. (se.north) .. (ze.north) .. (ye.east) ..(xe.east) }](x.south);

   \draw[blue,fill=blue,opacity=0.3](x.south) to[closed,curve through={(xw.west) .. (yw.west) ..(zw.south) .. (sw.south) .. (tw.west) .. (ta.north west) ..(te.north) .. (se.north) .. (ze.north) .. (ye.east) ..(xe.east) }](x.south);
   \end{pgfonlayer}
 \end{tikzpicture}
\caption{(local) Hecke vs Galois actions\label{U-opGL2}.}
\end{figure}
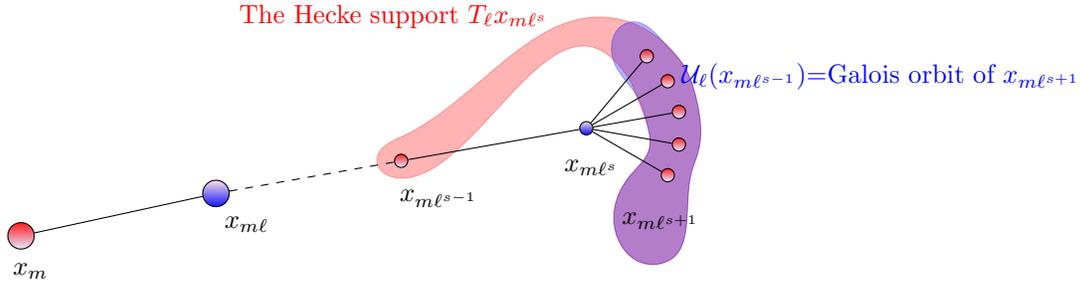
It is an easy exercise to verify (see \cite[Introduction]{UoperatorsII2021}) that the {\em "Hecke side (\ref{heckesideheegner}) and Galois side (\ref{galoissideheegner})"} implies the vertical distribution relations (\ref{verticalheegner}), while the {\em "Hecke side (\ref{heckesideheegner}) and Congruence side (\ref{congruencesideheegner})"} implies the tame distribution relation (\ref{tameheegner}).

\subsection{A general recipe for horizontal/vertical norm relations families}\label{generalrecipe}
{A strategy to construct norm compatible families (for general embeddings of Shimura data) emerges from the the reformulation of \S \ref{classical} and it is confirmed by the method of proof of Theorem \ref{Horizontal}. It goes as follows:
\begin{enumerate}[(A)]
\item\label{B} \textbf{The Hecke side:} In \cite{UoperatorsII2021}, we define and study the arithmetic of $\mathbb{U}$-operators in terms of Iwahori--Hecke algebras. This is done in utmost generality, for arbitrary connected reductive groups over p-adic fields. 
In the unramified case, we prove that the $\mathbb{U}$-operators attached to a cocharacter (not necessarily minuscule) is a right root of the corresponding Hecke polynomial \cite[Theorem \ref{uoprootheckeseed}]{Seedrelations2020}. 
\item \textbf{The congruence side:}  As the method of proof involves only the $\mathbb{U}$-operator and its powers through Lemma \ref{divisibilitylemma}; the Hecke polynomial enters the picture only through \ref{B}.
\item \textbf{The Galois side:} This is dealt with in \cite{Vunitarynormrelations2020}, it generalizes the equation (\ref{galoissideheegner}) for the general unitary case and gives evidence how ${U}$-operators can provide a strategy to establish vertical norm relations involving other Shimura varieties and also vertical norm relations.
\end{enumerate}
A much more detailed description of this recipe for constructing (horizontal) normal compatible systems of cycles of Shimura varieties will be the subject of another upcoming paper.
}

\subsection{Main results of the paper}
\subsubsection{The Gan--Gross--Prasad setting} Gan, Gross and Prasad formulated some conjectures\footnote{These conjectures may be thought of as Gross--Zagier-type formulas the cycles being generalizations of classical Heegner points.} relating special values of derivatives of automorphic $L$-functions to heights of certain special cycles on Shimura varieties constructed from embeddings of reductive groups, e.g. \cite[Conjecture~27.1]{gan-gross-prasad}. In this paper, we consider the case of special cycles on higher-dimensional Shimura varieties, where the embedding
$\Res_{E/\Q} \G_{m, E} \hra \GL_{2, \Q}$ defining Heegner points is replaced by an embedding of unitary groups $\U(n-1,1)\hra \U(n,1) \times \U(n-1,1)$ associated to a CM-extension $E/F$.

(\S \ref{groups}) Let $E$ be a CM field, that is, an imaginary quadratic extension of a totally real number field $F$. Set $[E:\Q]=2[F:\Q]=2d$. Let $\tau$ be the non-trivial element of $\Gal(E/F)$. Fix an integer $n> 1$. Let $W$ be a Hermitian $E$-space of dimension $n$ and of signature $(n-1,1)$ at one fixed distinguished embedding $\iota\colon E\hra \C$ and, of signature $(n,0)$ at the other archimedean places. Let $D$ be a positive definite Hermitian $E$-line. Consider the $n+1$-dimensional Hermitian $E$-space $V = W \oplus D$, it has signature $(n,1)$ at the distinguished archimedean place and, signature $(n,0)$ at the other ones. 

We consider the $F$-algebraic reductive groups of unitary isometries $\U(V)$ and $\U(W)$. Set $\G_V:=\Res_{F/\Q}\U(V)$ and $\G_W:=\Res_{F/\Q}\U(W)$. We identify $\G_W$ with the subgroup of $\G_V$. Let $\G = \G_V \times \G_W$ and 
$\mathbf{H} = \Delta(\G_W) \subset \G$, where $\Delta$ denotes the diagonal embedding $\Delta\colon \G_W \hookrightarrow \G$.

(\S \ref{Hermsymdom}) Let $\cX_V$ be the Hermitian symmetric domain consisting of negative definite lines in $V \otimes_{F, \iota} \R$ \ and similarly let $\cX_W$ be the set of negative definite lines in $W \otimes_{F, \iota} \R$. Setting $\cX = \cX_V \times \cX_W$, the diagonal embedding $W \hookrightarrow V \oplus W$ induces an embedding of Hermitian symmetric domains $\cX_W$ into $\cX$; set $\cY$ for the image of $\cX_W$.

(\S \ref{unitaryshimuravar}) The two pairs $(\G,\cX)$ and $({\Hbf},\cY)$ are Shimura data. For small enough compact open subgroup $K_{\G} \subset \G(\A_f)$ (resp. $K_{\Hbf}\subset \G(\A_f)$), the Shimura variety $\Sh_{K_{\G}}(\G,\cX)$ (resp. $\Sh_{K_{\Hbf}}({\Hbf},\cY)$) is a complex quasi-projective smooth variety whose $\C$-points are given by $$\G(\Q)\backslash(\cX \times (\G(\A_f)\slash K_{\G}))\quad(\text{resp. }{\Hbf}(\Q)\backslash(\cY \times ({\Hbf}(\A_f)\slash K_{\Hbf}))),$$ where $\G(\Q)$ (resp. ${\Hbf}(\Q)$) acts diagonally on $\cX \times (\G(\A_f)\slash K_{\G})$ (resp. $\cY \times ({\Hbf}(\A_f)\slash K_{\Hbf}))$. In fact, these varieties are defined over the reflex field $E=E(\G,\cX)=E({\Hbf},\cY)$ (See \S \ref{reflexfield} for the calculation of the reflex field).

(\S \ref{defcycles}) For every $g\in \G(\A_f)$, we will denote by $\mathfrak{z}_g$ the $n$-codimensional ${\Hbf}$-special cycle $[\cY\times gK] \subset \Sh_K(\G,\cX)(\C)$, as defined in Definition \ref{specialcycles}. Set,
$$\cZ_{\G,K}({\Hbf}):=\{\mathfrak{z}_g: g \in \G(\A_f).\}$$
The natural map $\G(\A_f)\to \cZ_{\G,K}({\Hbf})$ given by $g\mapsto \mathfrak{z}_g$, induces the bijection
$$\cZ_{\G,K}({\Hbf}) \simeq {\Hbf}(\Q){Z}_{\G}(\Q)\backslash\G(\A_f)/K.$$
(\S \ref{fieldofdefinition1}) The ${\Hbf}$-special cycles $\cZ_{\G,K}({\Hbf})$ are all defined over the {transfer class field} $E(\infty)$ (\S \ref{ringtransferclassfield}).

(\S \ref{galoisH}) The Galois group $\Gal(E(\infty)/E)$ acts on the set of special cycles {\em through} the left action of $H(\A_f)$. More precisely, for every $\sigma \in \Gal(E(\infty)/E)$, we let $h_\sigma \in {\Hbf}(\A_f)$ be any element verifying $Art_E^1(\det(h_\sigma) \cdot \T^1(\Q)=\sigma|_{E(\infty)}$. For every $g\in \G(\A_f)$, we have
$$\sigma(\mathfrak{z}_g)=\mathfrak{z}_{h_\sigma g}.$$ 
(\S \ref{compactsfiniteset}) For $\star\in \{W,V\}$, we fix any compact open subgroups $K_\star\subset \U_\star(\A_{F,f})$. There exists a finite set $S$ of places of $F$ such that $K_\star$ is of the form $K_{\star,S}\times K_\star^S$ where $K_{\star,S}$ is some compact open subgroup of $ \U_\star(\A_{F,f}^S)$ and $K_\star^S$ is the product of the hyperspecial compact open subgroups $K_{\star,v}:=\underline{\U}_\star(\O_{F_v})\subset \underline\U_\star(F_{v})$ for all $v \not \in S$. In particular, $K_{W,v}=K_{V,v}\cap \underline\U_W(F_{v})$. Set $K_v:=K_{V,v}\times K_{W,v}$.

\subsubsection{Main theorems on distribution} Set $$\mathcal{P}_{sp}:=\left\{\p \in \Spec(\cO_F)\colon \p \text{ is split in }E/F, \p\not\in S, \p \nmid \fc_1, \p \cO_E \nmid I_0\right\},$$
where $S$ is defined in \S \ref{basecycle}, $\fc_1$ in Remark \ref{containementfieldofdefinitions} and $I_0$ in Lemma \ref{I0nekovar}. Denote by $\mathcal{N}_{sp}$ the set of square free products of primes in $\mathcal{P}_{sp}$. For every place $v$ in $\mathcal{P}_{sp}$ corresponding to the prime ideal  $\p_v \in \mathcal{N}_{sp}$, let $w$ be the place of $E$ defined by the embedding $\iota_v \colon \overline{F} \hra \overline{F}_v$ fixed in \S \ref{NotationShimura}. We denote by ${\mathfrak{P}_w}$ the prime ideal of $\cO_E$ above $\p_v$ corresponding to the place $w$ and set $\Fr_w$ for the corresponding {\em geometric} Frobenius\footnote{Which induces $x \mapsto x^{-q_v}$ on the residue fields of $E_{v}^{un}$ (resp. $E_{w}^{un}$ and $E_{\overline{w}}^{un}$).}. Let $\textbf{Frob}_{w} \in \T^1(\A_f)$ be any element such that $\Art_E^1(\textbf{Frob}_{w} )|_{E(\infty)^{un,w}} = \Fr_w$, where $E(\infty)^{un,w}$ is the maximal unramified at $w$ extension in $E(\infty)$.

\textbf{Local congruence relations.} As we have pointed out in the Heegner case, the horizontal norm compatibility relation below are derived from local congruence relations (see Lemma \ref{divisibilitylemma} and Theorem \ref{Hordist}) generalizing the congruence equality (\ref{congruencesideheegner}):
\begin{theorem}\label{divisibilityheckepol}
With the above notation, we have 
$$H_{w}(\textbf{Frob}_{w})([1]_v)\equiv 0 \mod q_v^{n-1} \left(q_v- 1\right) \quad\text{in } \Z[q_v^{-1}][{\H}^{\der}(F_v)\backslash {\G}(F_{v})/K_v],$$
where $H_{w}$ is the Hecke polynomial attached to $\Sh_K(\G,\cX)$ at the place ${w}$ of the reflex field $E=E(\G,\cX)$.
\end{theorem}
\textbf{Tame relations.} As a corollary of Theorem \ref{divisibilityheckepol}, we obtain local horizontal relations in Corollary \ref{relhor2}, from which we derive the tame relations.

In \S \ref{basecycle}, we fix a cycle $\xi_1:=\mathfrak{z}_{g_0}$, for any $g_0$. There exists a field $\mathcal{K}$ (\S \ref{fieldkappa}) over which the base cycle $\mathfrak{z}_{g_0}$ is defined. 
Set $\mathscr{R}:=\Z[q_v^{-1},n_{E,F}^{-1}]$, where $n_{E,F}:=[E^\times \cap \A_{F,f}^\times \mathfrak{O}_{1}^\times :F^\times]$.
\begin{theorem}[Horizontal relations]\label{Horizontal}There exists a collection of cycles $\xi_{\mathfrak{f}}\subset \mathscr{R}[\mathcal{Z}_{\G,K}(\H)]$ (for all $\ff\in\mathcal{N}_{sp}$) each defined over $\cK(\mathfrak{f})$ (constructed in \S \ref{constructioncycles}) such that for every place $v \in \mathcal{P}_{sc}$, with $\p_v \nmid \ff$, we have
$$H_{{w}}(\Fr_{{w}}) \cdot \xi_{\mathfrak{f}}=\Tr_{\cK(\p_v \mathfrak{f})/\cK(\mathfrak{f})}\xi_{\p_v \mathfrak{f}},$$
where, $H_{{w}}$ is the Hecke polynomial attached to $\Sh_K(\G,\cX)$ at the place ${w}$ of the reflex field $E=E(\G,\cX)$ defined by $\iota_v$.
\end{theorem}
\begin{remark}
Actually, the cycles $\xi_{\mathfrak{f}}$ for $\mathfrak{f} \in \mathcal{N}_{sp}^r$ can be constructed in $\Z[\mathcal{Z}_{\G,K}(\H)]$ if $r\ge 1$. For this it suffices to use $P_{\mu^\varpi} \in \cH(\G(F)\sslash K,\Z)][X]$ the minimal polynomial of $\cU_{\mu^\varpi}$ instead of $H_{{w}}$ and exactly the same proof holds.
\end{remark}
\begin{remark}\label{finalremark}
The methods used in the sections \ref{localhorizontal} to prove Theorem \ref{Horizontal} for split places gives also, mutatis mutandis,  horizontal norm-compatible systems for inert places of $F$. 
\end{remark}

\subsection{Arithmetic applications via split Kolyvagin systems}
Kolyvagin's original argument (see~\cite{kolyvagin:euler_systems,kolyvagin:mordellweil,kolyvagin:structure_of_selmer,kolyvagin:structureofsha,gross:kolyvagin} as well as \cite{howard:heeg}) uses the tame norm relations at inert auxiliary primes. 
In all these variants of the same fundamental argument, one uses the fact that the global cohomology $H^1(\cK, \overline{T})$ (resp., the local cohomology groups $H^1(\cK_\lambda, \overline{T})$) 
decompose into two eigenspaces $H^1(\cK, \overline{T})^{\pm}$ (resp., $H^1(\cK_\lambda, \overline{T})^{\pm}$) for the action of a complex conjugation on the residual representation $\overline{T}$.  

Instead of applying global duality for the entire cohomology $H^1(\cK, \overline{T})$, one does that for each of the eigenspaces $H^1(\cK, \overline{T})^{\pm}$ 
and crucially uses the fact that if $\lambda$ is a Kolyvagin prime then $H^1_{ur}(\cK_\lambda, \overline{T})^{\pm}$ (being isomorphic to $(\overline{T} / (\Fr_\lambda - 1)\overline{T})^{\pm}$) are both one-dimensional. In the case when 
the residual representation $\overline{T}$ of $G_{\cK}$ need not extend to a representation of $G_{F}$, one can no longer 
apply the duality for the $\pm$-parts of the corresponding Selmer groups. If one attempts to use inert special primes $\lambda$, the local condition at $\lambda$ will no longer be one, but higher-dimensional and the same argument with global duality will no longer work. 

Very recently, Jetchev, Nekov\'a\v{r} and Skinner \cite{jetchev-nekovar-skinner} have managed to solve this problem by using split instead of inert 
auxiliary primes. In this case, if $w$ and $\overline{w}$ are the two places of $\cK$ above a split place $v$ of $F$, one applies the duality simultaneously for $w$ and $\overline{w}$ where the local term in the duality becomes 
$$
H^1(\cK_v, \overline{T}) / H^1_{ur}(\cK_v, \overline{T}) \oplus H^1(\cK_{\overline{v}}, \overline{T}) / H^1_{ur}(\cK_{\overline{v}}, \overline{T}). 
$$
By using a suitable application of the \v{C}ebotarev density theorem,  Jetchev, Nekov\'a\v{r} and Skinner still manage to run the argument, this time avoiding completely the action of complex conjugation on the residual Galois representation and the corresponding Selmer groups. The Kolyvagin systems obtained in this manner are referred to as split Kolyvagin systems.   

This application of the methods of Jetchev, Nekov\'a\v{r} and Skinner to the Gan--Gross--Prasad setting above, relies in a key way on the tame norm relations at split primes of Theorem \ref{Horizontal}. This allows the construction of a split Kolyvagin system for cohomological Galois representations appearing in the middle-degree cohomology for Shimura varieties associated to certain product unitary isometry groups $\U(V) \times \U(W)$ that appear naturally in the context of the Gan--Gross--Prasad conjectures.
\subsection{Acknowledgements}
This paper consists of chapters VI and VII of the author’s EPFL 2019 thesis. 
Therefore, I am thankful to my adviser D. Jetchev.  
I am indebted and very grateful to C. Cornut and J. Nekovar for their meticulous and careful reading.
Finally, I thank Y. Dadoun for valuable discussions and comments.
\section{Unitary Shimura varieties}\label{SectionShimura}
\subsection*{Notations}\label{NotationShimura}We fix an algebraic closure $\overline{\Q}$ of $\Q$ and an embedding $i_{\overline{\Q}}\colon \overline{\Q} \hra \C\nomenclature[F]{$i_{\overline{\Q}}$}{Fixed embedding $ \overline{\Q} \hra \C$}$. Let $F\subset \overline{\Q}$ be a totally real number field. Let $E\nomenclature[F]{$E$}{CM field in $\overline{\Q}$, with maximal totally real subfield $F$}$ be an imaginary quadratic extension of $F$, $E$ is usually said a CM field. Set $[E:\Q]=2[F:\Q]=2d\nomenclature[F]{$d$}{$[F:\Q]$}$ and let $$\Sigma_F\nomenclature[F]{$\Sigma_F$}{$\Hom_\Q(F\,, \overline{\Q})$}:=\Hom_\Q(F\,, \overline{\Q})=\Hom_\Q(F\,, \R)=\{{\,\iota}_i\nomenclature[F]{$\iota_i$}{The elements of $\Sigma_F$}\colon F\to \R, 1\le i\le d\}$$ be the set of real embedding corresponding to the archimedean places of $F$. Let $\tau\colon x \mapsto x^{\tau}\nomenclature[F]{$\tau$}{Non trivial element of $\Gal(E/F)$}$ be the non-trivial element of $\Gal(E/F)$.

For each place $v$ of $F$ (possibly archimedean), we fix an algebraic closure $\overline{F}_v\nomenclature[F]{$\overline{F}_v$}{Fixed algebraic closure for each place $v$ of $F$}$ of $F_v$ and consider the set $\Sigma_{E,v}:= \Hom_F(E\,, \overline{F}_v)$. The Galois group $\Gal(\overline{F}_v/F_v)$ acts on $\Sigma_{E,v}$ by post-composition: an element $\sigma \in \Gal(\overline{F}_v/F_v)$ sends any homomorphism $\varphi\in \Sigma_{E,v}$ to $\sigma \circ \varphi \in \Sigma_{E,v}$. This yields a bijection 
$$\Sigma_{E,v}/ \Gal(\overline{F}_v/F_v) \simeq \{w\text{ a place of }E\colon w \mid v\}$$
between $\Gal(\overline{F}_v/F_v)$-orbits of $F$-embeddings of $E$ in $\overline{F}_v$ and the places of $E$ above $v$. 
Set
$${\Sigma}_E:= \Hom_\Q(E\,, \overline{\Q})=\Hom_\Q(E\,, \C)=\bigsqcup_{\iota \in \Sigma_F}\Sigma_{E,\iota}.$$
For each $1\le i \le d$, choose an embedding $\widetilde{\iota}_i\in \Sigma_E\nomenclature[F]{$\widetilde{\iota}_i$}{A fixed extension in $\Sigma_E$ of $\iota_i\in \Sigma_F$}$ that extends ${\,\iota}_i\in \Sigma_F$. The set $\widetilde{\Sigma}_E=\{\widetilde{\iota}_i\colon 1\le i \le d\}\nomenclature[F]{$\widetilde{\Sigma}_E$}{A fixed CM type}$ is a CM type for $E$, because ${\Sigma}_E:=\widetilde{\Sigma}_E\sqcup \widetilde{\Sigma}_E^\tau$. We denote by $c \colon x \mapsto \overline{x}\nomenclature[F]{$c$}{Complex conjugation of $\Gal(\C/\R)$}$ the complex conjugation of $\Gal(\C/\R)$, we then have $c\circ  \widetilde{\iota}_i = \widetilde{\iota}_i\circ \tau$. 

We fix, for each finite place $v$ of $F$, an embedding $\begin{tikzcd}\iota_v\colon \overline{F} \arrow[r, hook] &  \overline{F}_v.\end{tikzcd}\nomenclature[F]{$\iota_v$}{Fixed embedding $\overline{F}\hra \overline{F}_v$}$We may then view elements of $\Sigma_E$ (in particular $\Sigma_{E,\iota}$) as $v$-adic embeddings of $E$: 
$$\begin{tikzcd} {\Sigma}_E \arrow{r}{}&\Sigma_{E,v}& \widetilde{\iota}\arrow[r, mapsto] &\widetilde{\iota}_v:= \iota_v \circ  \widetilde{\iota}.\end{tikzcd}$$
We fix a place of $E$ above each finite place of $F$ as follows: if $v$ splits in $E$, we fix the place $w_v$ to be the one defined by $\iota_v$ and by abuse of notation, denote the other place by $\overline{w}_v$. If $v$ is inert/ramified in $E$, we abuse notation and denote also by $v$ the unique place of $E$ above it.
\subsection{The groups}\label{groups}
Fix a positive integer $n>1\nomenclature[F]{$n$}{Fixed integer}$ and let $(V,\psi)\nomenclature[F]{$(V,\psi)$}{$n+1$-dimensional $E$-hermitian space with signature $(n,1)$ at $\iota_1$ and $(n+1,0)$ elsewhere}$ be a non-degenerate hermitian $E$-space of dimension $n+1$. Suppose that \footnote{\label{tensorhermitianspace} 
For every $F$-algebra $R$, set $E_R:=E\otimes_F R$ and $V_R:=V\otimes_{F}R= V\otimes_{E}E_R$. We define the action of $\tau$ on $E_R=E\otimes_F R$ by letting it act on left component. Then, extend $\psi$ to a Hermitian form $\psi_{R}\colon V_R\times V_R \to E_R$ as follows  $$\psi_{R}(v\otimes x, v'\otimes y)=\psi_R(v,v')xy^\tau,\quad \forall v,v'\in V \text{ and }\forall x,y \in E_R.$$
For example: For each $1\le i \le d$, the fixed embedding ${\,\iota}_i\colon F \to F_v$ induces a natural $F$-algebra structure on $\R$. Now letting $R=\R$, one gets $(V,\psi)\otimes_{F\,,{\,\iota}_i}\R:=(V_R,\psi_R)$, the hermitian space $(V\otimes_{F\,,{\,\iota}_i} \R,\psi_{i})$ over $E\otimes_{F\,,{\,\iota}_i}\R$.}
$$\text{sign}((V,\psi)\otimes_{F\,,{\,\iota}_i} \R)=\begin{cases} (n,1) & \text{ if }i=1,\\
(n+1,0) &\text{ if }i\neq1.\end{cases}$$
We consider the $F$-algebraic reductive group of unitary isometries $\U(V,\psi)$, this is a connected reductive group whose $R$-points, for any $F$-algebra $R$, are given by
$$\U(V,\psi)(R)=\{g \in \GL(V \otimes_F R)\colon \psi(gx,gy)=\psi(x,y), \forall x,y \in  V \otimes_F R\}.$$
We will be mainly interested in the cases where $R=F_v$, the completion of $F$ at finite places $v$, or $R=\A_{F,f}$, the ring of finite adeles of $F$. Since our hermitian form $\psi$ is fixed, we shall refer to the group of unitary isometries $\U(V,\psi)$ as $\U(V)$. 

Let $v\in V$ be an anisotropic vector, every other $v'\in V$ such that $\psi(v',v')=\psi(v,v)$ is conjugate under $\mathbf{SU}(V)(F)=\U(V)(F)\cap \SL(V)(E)$ to $v$, (see \cite[1.5]{Shimura08}). Fix once and for all an anisotropic vector $w_{n+1}\in V\nomenclature[F]{$w_{n+1}$}{Fixed anisotropic vector in $V$ with value $1$}$. 
Without loss of generality, assume that $\psi(w_{n+1},w_{n+1})=1$ \footnote{If $\psi(v,v)\neq 0$, one can choose any non-zero vector $v' \in E\cdot v$ and consider the modified hermitian $E$-space $(V,\psi(v,v)^{-1}\psi)$. But, we may have changed the signature. A better argument: By density of $V$ in $V_{\iota_1}$, there is a vector $v \in V$ with $\psi(v,v)$ positive at $\iota_1$, hence everywhere. We choose this $v$ and consider the hermitian $E$-space $(V,\psi(v,v)^{-1}\psi)$. Although the hermitian $E$-spaces are different, the associated unitary groups are isomorphic.}. 
Set $D=E\cdot w_{n+1}$ and $W=D^\perp$. Hence, the signature of $(D,\psi|_{D})$ is $(1,0)$ at all archimedean places of $F$. Consequently, the induced hermitian subspace $(W,\psi|_W)$ has signature $(n-1,1)$ at the distinguished archimedean place ${\,\iota}_1$ and $(n,0)$ at the other archimedean places ${\,\iota}_i, \forall 2\le i \le n$. Similarly, we associate to $(W,\psi|_W)$ the $F$-algebraic group of unitary isometries $\U(W)$. 

Set $\G_V:=\Res_{F/\Q}\U(V)\nomenclature[F]{$\G_\star$}{$\Res_{F/\Q}\U(\star)$ for $\star\in \{V,W\}$}$ and $\G_W:=\Res_{F/\Q}\U(W)$. Thus, for $\star\in \{V,W\}$ and for any $\R$-algebra $R$, we have
\begin{align*}\G_{\star,\R}(R)=\U(\star)(F \otimes_\Q R)&= \U(\star)(F \otimes_\Q \R \otimes_\R R)\\
&= \U(\star)(\bigoplus_{\iota  \in \Sigma_F} \R_\iota \otimes_\R R)\\
&=\prod_{\iota\in \Sigma_F}\U(\star)(R_{\iota_i}),\end{align*}
where $\R_\iota$ is just $\R$ indexed by elements of $\Sigma_F$ and $R_{\iota}$ is $R$ endowed with the $F$-algebra structure given by the embedding $\iota \colon F \hra \R$. 
This yields $\G_{\star,\R}= \prod_{\iota \in \Sigma_F}\G_{\star,{\,\iota}}$ where, $$\G_{\star,{\,\iota}}=\U((\star,\psi)\otimes_{F,{\,\iota}}\R)\simeq\begin{cases}\U(\dim_E\star-1,1)_\R & \text{ if } \iota=\iota_1,\\
\U(\dim_E\star,0)_\R &\text{ if }\iota \neq \iota_1.
\end{cases}$$
Likewise
$$\G_{\star,\C}= \prod_{\widetilde{\iota} \in \widetilde{\Sigma}_E}\G_{\star, \widetilde{\iota}} \quad \text{ where } \quad\G_{\star,\widetilde{\iota}}=\GL(\star\otimes_{E,\widetilde{\iota}}\C)\simeq \GL(\dim_E\star)_\C.$$
By left-exactness of the Weil restriction we have an embedding of $\Q$-algebraic groups $\G_W\hra \G_V$ that identifies $\G_W$ with the subgroup of $\G_V$ given by:
$$ \{g \in \G_V(R)\subset \GL(V\otimes_\Q R)\colon g \cdot x=x, \quad\forall x \in  D \otimes_\Q R\},$$
for any $\Q$-algebras $R$. Let $\G = \G_V \times \G_W\nomenclature[F]{$\G$}{$\G_V \times \G_W$}$ and $\H=\Delta(\G_W)\subset \G\nomenclature[F]{$\H$}{$\Delta(\G_W)$}$, where $\Delta\nomenclature[F]{$\Delta$}{The diagonal embedding $\Delta\colon \G_W \hookrightarrow \G$}$ denotes the diagonal embedding $\Delta\colon \G_W \hookrightarrow \G$.
\subsection{The Deligne torus and variant}
\label{comparisoniotass}
\begin{itemize}[topsep= 0pt]
\item The torus $\Res_{\C/\R} \mathds{G}_{m,\C}$ is called the Deligne torus and is usually denoted by $\mathds{S}$. we will also consider its norm one subgroup $\U(1):=  \U_{\C/\R}(1)$. We get by base change to $\C$ a canonical isomorphism of $\C$-tori
 $$\mathds{S}_\C\simeq  \mathds{G}_{m,\C}\times \mathds{G}_{m,\C},$$
where the factors are ordered in the way that $\mathds{S}(\R)=\C^\times \to \mathds{S}(\C)=\C^\times \times {\C}^\times$ is the map $z\mapsto (z, \overline{z})$. 

Define the cocharacter $\mu=\mu_{\mathds{S}} \colon \mathds{G}_{m,\C} \to \mathds{S}_\C$, given on $\C$-points by: $ \C^\times \to  \mathds{S}_\C(\C)\simeq \C^\times \times \C^\times$, $ z \mapsto (z,1)$, one may also consider $\mu^c=\overline{\mu}_{\mathds{S}}$ ($c$ for complex conjugation) given on $\C$-points by $z \mapsto (1,{z})$. 
\item For each $\iota \in \Sigma_F$, consider the extension $ E_{\iota}:=E\otimes_{F,\iota}\R/\R$. We will use the notation $\iota\mathds{S}:=\Res_{E_{\iota}/\R}\mathds{G}_{m,E_{\iota}}$ and $ \widetilde{\iota}\U(1):=  \U_{E_{\iota}/\R}(1)$. Likewise, a base change to $E_{\iota}$ yields an isomorphism
 $$\iota\mathds{S}_{E_\iota}\simeq \mathds{G}_{m,E_\iota}\times \mathds{G}_{m,E_\iota},$$
 where the factors are again ordered in the way that $E_\iota^\times \to \iota\mathds{S}(E_\iota)=E_\iota^\times \times E_\iota^\times$ is the map $z\mapsto (z, z^{\tau})$. Define similarly the cocharacter $\mu_{\iota}=\mu_{\iota\mathds{S}} \colon \mathds{G}_{m,E_\iota} \to \iota\mathds{S}_{E_\iota}$, given on $E_\iota$-points by: $ E_\iota^\times \to  \iota\mathds{S}_{E_\iota}(E_\iota)\simeq E_\iota^\times \times E_\iota^\times$, $ z \mapsto (z,1)$, one may also consider $\mu^\tau=\overline{\mu}_{\iota\mathds{S}}$ given on $\C$-points by $z \mapsto (1,{z})$.
\item The distinguished complex embedding $\widetilde{\iota}_1\colon E \to \C$, induces an isomorphism of fields 
 $$\begin{tikzcd} \jmath_1 \colon E_{\iota_1}=E\otimes_{F,\iota_1} \R \arrow{r}{\simeq}  &   \widetilde{\iota}_1(E)\otimes_{F,\iota_1} \R=\C  \end{tikzcd},$$
and so yields an isomorphisms of $\R$-groups \begin{tikzcd}\jmath_1 \colon \mathds{S} \arrow{r}{\simeq}&  {\iota}_1\mathds{S}\end{tikzcd},\begin{tikzcd}\U(1) \arrow{r}{\simeq}&{}\widetilde{\iota}\U(1)\end{tikzcd}.\newline Moreover, the base change of $\jmath_1$ to $E_{\iota_1}$ is compatible with the base change to $\C$, in the sense that
$$\jmath_1 \circ \mu= \mu_{\iota_1} \text{ and } \jmath_1 \circ\mu^c =\mu_{\iota_1}^\tau,$$
and the following diagram
\begin{equation}\label{squareiota}\begin{tikzcd}[column sep=large]
\mathds{S}\arrow{d}{\simeq}[swap]{\jmath_1} \arrow[two heads]{r}{z \mapsto \frac{z}{\overline{z}}} &\U(1)\arrow{d}{\jmath_1}[swap]{\simeq}\\
 {\iota}_1 \mathds{S} \arrow[two heads]{r}{s \mapsto \frac{s}{s^\tau}} &\widetilde{\iota}\U(1)
\end{tikzcd}\end{equation}
commutes.
\end{itemize}
We refer the reader to \cite[\S II.1.3]{boumasmoud19} for more details on tori of the form $\Res_{A/B} \mathds{G}_{m,A}$.
\subsection{The Hermitian symmetric domains}\label{Hermsymdom}
Let $\mathfrak{B}_V=(w_1,\cdots,w_{n+1})\nomenclature[F]{$\mathfrak{B}_\star$}{$E$-basis of $\star\in \{V,W\}$}$ be an orthogonal $E$-basis of $(V,\psi)$, here $w_{n+1}$ is the $E$-generator fixed in \S \ref{groups} of the $E$-line $D=W^\perp$. Thus, $\mathfrak{B}_W:=(w_1,\cdots,w_{n})$ is also an orthogonal $E$-basis of $(W,\psi)$. Recall that $E_{{\iota}_1}=E\otimes_{F,{\,\iota}_1}\R (\simeq \widetilde{\iota}_1(E) \otimes_{F,{\,\iota}_1}\R=\C)$ and let $\mathfrak{B}_{V,1}=(w_{1,1},\cdots,w_{n+1,1})\nomenclature[F]{$\mathfrak{B}_{V,1}$}{Induced $E_{\iota_1}$-base of $V_{\iota_1}$}$ be the orthogonal $E_{{\iota}_1}$-basis of $(V_{{\iota}_1}:=V\otimes_{F,{\,\iota}_1}\R,\psi_1)$ obtained from $\mathfrak{B}_V$ by base change along the distinguished ${\,\iota}_1\colon F\hookrightarrow \R$. Since the signature of $(V_{\iota_1},\psi_1)$ is $(n,1)$, we may and will assume that $c_1=\psi_1(w_{1,1},w_{1,1})<0$ and $c_i=\psi_1(w_{i,1},w_{i,1})>0$ for all $1< i \le n+1$. Consider for simplicity the orthogonal basis $\mathfrak{B}_{V,1}^\prime=\left(
w_{1,1}^\prime,w_{2,1}^\prime\cdots,w_{n+1,1}^\prime\right)=\left(\frac{1}{\sqrt{-c_1}}w_{1,1},\frac{1}{\sqrt{c_2}}w_{2,1}\cdots,\frac{1}{\sqrt{c_{n+1}}}w_{n+1,1}\right)$, thus the Hermitian form $\psi_1$ with respect to $\mathfrak{B}_{V,1}^\prime$ is given by 
$$\psi_1(x,y)=-x_1{y}_1^\tau+\cdots+x_n{y}_n^\tau+x_{n+1}{y}_{n+1}^\tau,$$
where $x=\sum_{i=1}^{n+1} x_i w_{i,1}^\prime$ and $y=\sum_{i=1}^{n+1} y_i w_{i,1}^\prime$ are two elements of $V_{{\iota}_1}$, with $x_i,y_i \in E_{{\iota}_1}$. This form corresponds to the matrix $J_{\mathfrak{B}_V^\prime}=\diag(-1,1,\cdots,1)$. Likewise, let $J_{\mathfrak{B}_W^\prime}=\diag(-1,1,\cdots,1)$ be the hermitian matrix corresponding to the basis $\mathfrak{B}_W^\prime=\left(w_{1,1}^\prime,w_{2,1}^\prime, \cdots,w_{n,1}^\prime\right)$. 

Let us view the $E_{{\iota}_1}$-vectors space $V_{{\iota}_1}\simeq E_{{\iota}_1}^{n+1} (\simeq \C^{n+1})$ as being a union $$V_1^-\cup V_1^0 \cup V_1^+,$$ of negative (resp. null, positive) vectors $x\in V_{{\iota}_1}$, depending on the sign of $\psi_1(x,x)\in \R$. For instance
$$V_1^-=\left\{x=\sum_{i=1}^{n+1} x_i w_{i,1}^\prime\in V_{{\iota}_1} \colon \psi_1(x,x)=-|x_{1}|^2+\sum_{i=2}^{{n+1}} |x_i|^2<0\right\}.$$
For every non-zero $x\in V_{{\iota}_1}$, all non-zero vectors of the "complex"\footnote{By a slight abuse of language, we use the adjective complex here, since this is really a complex line up to base change along $\widetilde{\iota}_1$.} line $E_{{\iota}_1} x$ have the same sign as $x$. Therefore, we can attach to each $E_{{\iota}_1}$-line in $V_{{\iota}_1}$ a sign in $\{-,0,+\}$. In particular, we may think of $V_1^-$ as the union of all negative lines in $ V_{{\iota}_1}$. Let $\cX_V\nomenclature[F]{$\cX_V,\cX_W$}{Hermitian symmetric domain of negative lines in $V_1$ and $W_1$}$ denote the set of negative lines in $V_1^-$. 

Intersecting the negative lines with the hyperplane defined by 
$$\left\{x=\sum_{i=1}^{n+1} x_i w_{i,1}^\prime\in V_{{\iota}_1} \colon x_{1}=1\right\},$$
one gets an identification of $\cX_V$ with the complex open ball of dimension $n$:
$$\mathbb{B}_n:=\left\{x=\sum_{i=1}^{n+1} x_i w_{i,1}^\prime\in V_{{\iota}_1} \colon x_1=1 \text{ and }\sum_{i=2}^{{n+1}} |x_i|^2<1\right\}.$$
For instance, the "complex" line $\ell_{V}=E_{\iota_1}w_{1,1}\nomenclature[F]{$\ell_{V}$}{Negative $E_{\iota_1}$-line spanned by $w_{1,1}$}$ is an element of $\cX_V$. The line $\ell_{V}$ corresponds to the centre $(1,0,\cdots,0)$ (for the basis $\mathfrak{B}_V^\prime$) of the ball $\mathbb{B}_n$.

The group $\G_{V,{\,\iota}_1}(\R)\simeq \U(n,1)(\R)$ acts transitively on the set $\cX_V$ and the stabilizer of $\ell_{V}$ is isomorphic to $\U(n)\times \U(1)$ (See \cite[Lemma 3.1.3]{Goldman99}). 

The negative line $\ell_{V}$ determines a homomorphism of $\R$-algebraic groups \begin{tikzcd}{h}_{\mathfrak{B}_V,1}\colon \mathds{S} \arrow{r}& \G_{V,{\,\iota}_1},\end{tikzcd}as follows: the basis $\mathfrak{B}_V$ gives rise to the maximal $F$-subtorus of $\U(V)$
$$\T(\mathfrak{B}_{V}):=\U(Ew_{1})\times \cdots \times\U(Ew_{n+1})\subset \U_{V},\nomenclature[F]{$\T(\mathfrak{B}_{V})$}{Maximal $F$-subtorus of $\U(V)$ defined by $\mathfrak{B}_{V}$}$$
similarly, the basis $\mathfrak{B}_{V,1}$ defines the maximal $\R$-subtorus of $\G_{V,{\,\iota}_1}$:
$$\T(\mathfrak{B}_{V,1}):=\U(\ell_V) \times \U(E_{{\iota}_1} w_{1,2})\times \cdots\times \U(E_{{\iota}_1} w_{1,n+1})=\T(\mathfrak{B}_V)_\R\subset \G_{V,{\,\iota}_1}.$$
The homomorphism of $E_{{\iota}_1}$-spaces $E_{{\iota}_1}\to \ell_{V}$ given by $z\mapsto z w_{1,1}$ induces an isomorphism of $\R$-groups $u_{V,1}\colon \widetilde{\iota}_1\U(1)_\R \simeq \U(E_{{\iota}_1} w_{1,1})$. Extend $u_{V,1}$ to the morphism of $\R$-groups  $\widetilde{\iota}_1\U(1)_\R \to \T(\mathfrak{B}_{V,1})\subset \G_{V,{\,\iota}_1}$ (by letting it act trivially on $\ell_V^\perp \subset V_{{\iota}_1}$), given with respect to the basis $\mathfrak{B}_{V,1}$, as follows:
$$u_{V,1}\colon\widetilde{\iota}_1\U(1)_\R\to \G_{V,{\,\iota}_1}, \quad z\mapsto \diag(z,1,\cdots,1).$$ 
Now, define the homomorphism $\widetilde{h}_{\mathfrak{B}_V,1}$ to be the composition of the following homomorphism of $\R$-algebraic groups:
 $$\begin{tikzcd}{\iota}_1 \mathds{S} \arrow{r}{s \mapsto \frac{s}{s^\tau}}&\widetilde{\iota}\U(1)_\R\arrow{r}{u_{V,1}}&\G_{V,{\,\iota}_1}.
\end{tikzcd}$$ 
Accordingly, the desired\begin{tikzcd}[column sep= small]{h}_{\mathfrak{B}_V,1}\colon \mathds{S} \arrow{r}& \G_{V,{\,\iota}_1},\end{tikzcd} is ${h}_{\mathfrak{B}_V,1}= \widetilde{h}_{\mathfrak{B}_V,1} \circ \jmath_1$, applying the square (\ref{squareiota}) we also see that $${h}_{\mathfrak{B}_V,1}\colon z \longmapsto u_{V,1}\left(\frac{\jmath_1(z)}{(\jmath_1({z}))^\tau}\right) = u_{V,1}\circ \jmath_1\left(\frac{z}{\overline{z}}\right).$$

The centralizer of $\widetilde{h}_{\mathfrak{B}_V,1}$ is the compact subgroup
$$\U(\ell_V) \times \U(\ell_V^\perp) \simeq \U(1)_\R \times \U(n,0)_\R.$$
Similarly, any negative line $\ell\in \cX_V$ defines a homomorphism of $\R$-algebraic groups $\widetilde{h}_{V,\ell}\colon \widetilde{\iota}_1\mathds{S}  \to \G_{V,{\,\iota}_1}$ and the transitive action of $\G_{V,{\,\iota}_1}(\R)$ on $\cX_V$, there exists a unitary isometry $g\in \G_{V,{\,\iota}_1}(\R)$ such that $g \cdot \ell_V = \ell$. Hence, the construction of ${h}_{V,1}$ above, yields the equality
$$\widetilde{h}_{V,\ell}= g \widetilde{h}_{V,1} g^{-1}\text{ and }{h}_{V,\ell}= g {h}_{V,1} g^{-1}.$$
Therefore, one may and will identify the set of negatives lines $\cX_V$ with the $\G_{V,{\,\iota}_1}(\R)$-conjugacy class of the homomorphism ${h}_{V,1}$.

The discussion above applies also to $W$. Let $\cX_W$ be the set of negative definite $\C$-lines in $W_{\iota_1}:=W \otimes _{F, {\,\iota}_1} \R (\simeq \C^{n})$. The negative line $\ell_W=E_{\iota_1} w_{1,1} (=\ell_V)$, defines a homomorphism of $\R$-algebraic groups $\widetilde{h}_{W,1}\colon\widetilde{\iota}_1\mathds{S} \to \G_{W,{\,\iota}_1}$, given on $\R$-points by\footnote{With respect to the basis $\mathfrak{B}_{W,1}=(w_{1,1},\cdots,w_{n,1})$.}
$$z\in E_{\iota_1}^\times=\widetilde{\iota}_1\mathds{S} (\R)\mapsto \diag(z/{z}^\tau,1,\cdots,1)\in \G_{W,{\,\iota}_1}(\R).$$
Consider the induced homomorphism\begin{tikzcd}[column sep= small]{h}_{W,1}=\widetilde{h}_{W,1} \circ \jmath_1\colon \mathds{S} \arrow{r}& \G_{W,{\,\iota}_1}.\end{tikzcd} 
Likewise, we identify as above $\cX_W$ with the $\G_{W,{\,\iota}_1}(\R)$-conjugacy class of the $\R$-algebraic homomorphism ${h}_{W,1}$.

For every $\star\in \{V,W\}$, the transitive action of $\G_{\star,{\,\iota}_1}(\R)$ on $\cX_{\star}$ naturally induces a transitive action of $\G_{\star}(\R)=\prod_{i=1}^d\G_{\star,{\,\iota}_i}(\R)$ on $\cX_{\star}$ with isotropy group the maximal connected compact subgroup:
$$\Stab_{\G_{\star}(\R)}({h}_{\star})= (\U(\ell_\star^\perp)\times \U(\ell_\star))\times \prod_{i=2}^d\G_{\star,{\,\iota}_i}(\R)\simeq \U(1)_\R \U(\dim_E \star-1)_\R\times \U(\dim_E \star,0)_\R^{d-1},$$
where, ${h}_{\star} \colon\mathds{S}\to \G_{\star,\R}$ is the $\R$-algebraic homomorphism induced from ${h}_{\star,1}\colon\mathds{S}\to \G_{\star,{\,\iota}_1}$ and given on $\R$-points by
$${h}_{\star} \colon z\mapsto ({h}_{\star,1} (z),\text{Id}_2,\cdots,\text{Id}_{d}).$$
We shall, therefore, identify $\cX_\star$ with the $\G_{\star}(\R)$-conjugacy class of the homomorphism ${h}_{\star} $. Let $\cX = \cX_V \times \cX_W\nomenclature[F]{$\cX$}{$\cX_V \times \cX_W$}$, i.e. the $\G(\R)$-conjugacy class of the homomorphism ${h}_{V}\times {h}_{W}\colon \mathds{S}\to \G_\R$. The diagonal embedding $\Delta\colon W \hookrightarrow V \oplus W$ induces an embedding of $\cX_W$ into $\cX$. Write $\cY\nomenclature[F]{$\cY$}{$\Delta(\cX_W)$ with $\Delta\colon \cX_W\hra \cX$ diagonal}$ for $\Delta(\cX_W)$; the $\H(\R)$-conjugacy class of the homomorphism $\Delta({h}_{W})\colon \mathds{S}\to \H_\R$. It can be easily checked that $$\cY=\{h \in \cX| h\colon \mathds{S} \to \G_\R \text{ factors through } \Delta\colon\H_\R\hookrightarrow \G_\R\}.$$
\subsection{The reflex field}\label{reflexfield}
We define $$\T:=\Res_{E/\Q}\mathds{G}_{m,E},\quad\Zbf:=\Res_{F/\Q}\mathds{G}_{m,F},\nomenclature[F]{$\T, \Zbf$}{$\Res_{E/\Q}\mathds{G}_{m,E}, \Res_{F/\Q}\mathds{G}_{m,F}$}$$ 
and $$\T^1:=\Res_{F/\Q} \U_{E/F}(1)= \ker (\Norm\colon \T \twoheadrightarrow \Zbf)\nomenclature[F]{$\T^1$}{$ \ker (\Norm\colon \T \twoheadrightarrow \Zbf)$}.$$ For $\star \in \{{V},W\}$, let $\det:\G_\star\twoheadrightarrow \T^1$ be the determinant map. We have $\G_\star^{\der}=\ker(\det)=\Res_{F/\Q}\SU(\star)$.

Recall the maximal $F$-subtorus of $\U_{{\star}}$,
$$\T(\mathfrak{B}_{\star})=\U(Ew_1)\times\cdots \times\U(Ew_{{\dim \star}}) \simeq (\U_{E/F}(1))^{{\dim \star}}.$$
It induces the maximal $\Q$-subtorus of $\G_{{\star}}=\Res_{F/\Q}\U_{{\star}}$:
$$\T_{\mathfrak{B}_{\star}}:=\Res_{F/\Q}(\T(\mathfrak{B}_{\star}))=\Res_{F/\Q}\U(Ew_1)\times\cdots \times \Res_{F/\Q}\U(Ew_{{\dim \star}}) \simeq\left( \T^1\right)^{{\dim \star}}.$$
Using the natural diagonal embedding $\T^1\to \T_{\mathfrak{B}_\star}$, we may view $\T^1$ as the center of the group $\G_\star$. Define ${\mu}_{\mathfrak{B}_{\star}}=({h}_{{\star}})_\C\circ \mu \in X_*(\T_{\mathfrak{B}_{\star}})_\C\nomenclature[F]{${\mu}_{\mathfrak{B}_{\star}},({h}_{\star})_\C$}{}$ \footnote{We can recover ${h}_{\mathfrak{B}_{\star}}$ from $\mu_{\star}$ via ${h}_{\mathfrak{B}_{\star}}(z)=\mu_{\star}(z)\cdot \mu_{\star}(\overline{z})$.}, where $\mu$ is the cocharacter of the Deligne torus introduced in \S \ref{comparisoniotass} and $({h}_{\star})_\C$ is, using the identification\footnote{Determined by the choice of the CM type $\widetilde{\Sigma}_1$.} $\G_{\star,\C}\simeq \prod_{\widetilde{\iota} \in \widetilde{\Sigma}_E}\GL(\star\otimes_{E,\widetilde{\iota}}\C)\simeq \GL(\dim_E\star)_\C$, given by
$$\begin{tikzcd}[row sep= tiny]({h}_{\star})_\C\colon  \mathds{S}_\C(\C)\arrow{r}&\T_{\mathfrak{B}_{{\star},\C}}(\C) \\
 (z_1, z_2) \arrow[r, mapsto]& \left( \begin{pmatrix}
z_1/z_2&\\
&\text{Id}_{\dim_E\star-1}
\end{pmatrix},\text{Id}_{\dim_E\star},\cdots,\text{Id}_{\dim_E\star} \right).  \end{tikzcd}$$
For every $?\in \{\G,\H,\G_{V},\G_W\}$ and any field $\cK\subset \overline{\Q}$, $\mathcal{M}_{?}(\cK)$ denotes the set of $?(\cK)$-conjugacy classes of (algebraic group) homomorphisms $\Gm_{m,\cK}\to  ?_\cK$. Note that by construction, ${h}_{\star}\colon \mathds{S} \to \G_{\star}$ factors through the torus $\T_{\mathfrak{B}_{\star}} \hookrightarrow \G_{\star}$.

By \cite[(b) Lemma 1.1.3]{Ko1}, we know that $\mathcal{M}_{\G_{\star}}(\C)=\mathcal{M}_{\G_{\star}}(\overline{\Q})$. Let $\cX_{{\star},\overline{\Q}}\in \mathcal{M}_{\G_{\star}}(\overline{\Q})$ be the class that corresponds to the $\G_{\star}(\C)$-conjugacy class of $\mu_{\star}$ \footnote{The $\G_{\star}(\C)$-conjugacy class of $\mu_{\star}$ does not depend on the representatives ${h}_{\mathfrak{B}_{\star}}\in \cX_{\star}$.}. The reflex field $E(\G_{\star},\cX_{\star})$ is defined to be the fixed field of the subgroup of $\Gal(\overline{\Q}/\Q)$ fixing $\cX_{{\star},\overline{\Q}}$. 
We have (see proof of \cite[Lemma 1.1.3]{Ko1})
$$\mathcal{M}_{\G_{\star}}(\overline{\Q})\simeq X_*(\T_{\mathfrak{B}_{\star}})/W(\G_{\star},\T_{\mathfrak{B}_{\star}})$$
where, $W(\G_{\star},\T_{\mathfrak{B}_{\star}})$ denotes the absolute Weyl group of $\T_{\mathfrak{B}_{\star}}$ in $\G_{\star}$. Therefore, the reflex field $E(\G_{\star},\cX_{\star})$ is also the field of definition of the $W(\G_{\star},\T_{\mathfrak{B}_{\star}})$-orbit of ${\mu}_{\mathfrak{B}_{\star}}$. We now exhibit a rather explicit description of the $W(\G_{\star},\T_{\mathfrak{B}_{\star}})\rtimes\Gal{(\overline{\Q}/\Q)}$-module $X_*(\T_{\mathfrak{B}_{\star}})$ which will ease the computation of the reflex field. 
Recall that the duality between $X_*(\T_{\mathfrak{B}_{\star}})$ and $X^*(\T_{\mathfrak{B}_{\star}})$ is compatible with the $\Gal{(\overline{\Q}/\Q)}$-action on both sides, in addition 
\begin{align*}X_*(\T_{\mathfrak{B}_{\star}})&\simeq \Hom_{\Z-\text{mod}}(X^*(\T_{\mathfrak{B}_{\star}}),\Z)\\ 
&=\bigoplus_{i=1}^{i={\dim \star}}\Hom_{\Z-\text{mod}}(X^*(\Res_{F/\Q}\U(Ew_i)),\Z).\end{align*}
thus, $$X_*(\T_{\mathfrak{B}_{\star}})\simeq \bigoplus_{i=1}^{i={\dim \star}} \Hom_{\Z-\text{mod}}(X^*(\T^1)_i,\Z),$$
where, $X^*(\T^1)_i$ is just a copy of $X^*(\T^1)$ indexed by $1\le i\le \dim_E {\star}$. On the other hand,
we have $$\T(\overline{\Q})=\Res_{E/\Q}\Gm_{m,E} (\overline{\Q})=(E\otimes_\Q \overline{\Q})^\times=\bigoplus_{{\,\iota}\in \Sigma_K} (\overline{\Q}_{{\,\iota}})^\times,$$
where $\overline{\Q}_{{\,\iota}}$ is just $\overline{\Q}$ indexed by elements of $\Sigma_E=\Hom_\Q(E\,, \overline{\Q})$ and endowed with the $E$-algebra structure given by the embedding $\iota \colon E \hra \R$. Moreover, projections on each factor $(\overline{\Q}_{{\,\iota}})^\times$ is an algebraic character that we denote by $f_{\,\iota}$. Hence, $\{f_{\,\iota}\colon {\,\iota} \in\Sigma_E\}$ is a $\overline{\Q}$-basis for $X^*(\T)$. We have then a canonical isomorphism of $\Gal(\overline{\Q}/\Q)$-modules
$$X^*(\T)_{\overline{\Q}}\simeq \bigoplus_{{\,\iota} \in \Sigma_E} \Z f_{\,\iota},$$
with the canonical Galois module on the right hand side. Define $\{f^\vee_{\,\iota}\colon {\,\iota}\in \Sigma_E\}$ to be the dual basis of $\{f_{\,\iota}\colon {\,\iota} \in\Sigma_E\}$ in $\Hom_{\Z-\text{mod}}(X^*(\T),\Z)$. A homomorphism $f\in \Hom_{\Z-\text{mod}}(X^*(\T),\Z)$ is completely determined by the $\Z$-values it attaches to the basis $\{f^\vee_{\,\iota}\colon {\,\iota}\in \Sigma_E\}$ or equivalently, to $\{{\,\iota} \in \Sigma_E\}$. Therefore, we obtain the isomorphism
$$\begin{tikzcd}\{f\colon \Sigma_E \to \Z\} \arrow{r}{\simeq}& X_*(\T)_{\overline{\Q}} \quad(\simeq \Hom_{\Z-\text{mod}}(X^*(\T)_{\overline{\Q}},\Z)),\end{tikzcd}$$ 
given by 
$$f \longmapsto \left(\lambda_f\colon \overline{\Q}\to \T(\overline{\Q})=\bigoplus_{{\,\iota}\in \Sigma_E} (\overline{\Q}_{{\,\iota}})^\times, \quad  z\mapsto \prod_{{\,\iota}\in \Sigma_E} f_{\,\iota}^\vee(z)^{f({\,\iota})}\right).$$
The embedding $\T^1 \hookrightarrow \T$ induces an injection $X_*(\T^1) \hookrightarrow X_*(\T)$. We will try here to describe the submodule of $\{f\colon \Sigma_E \to \Z\} $ corresponding to $X_*(\T^1)$. We begin by describing the $\overline{\Q}$-points of $\T^1$ as follows
\begin{align*}\T^1(\overline{\Q})&=\left\{z\in \T(\overline{\Q})\colon \chi(z)=1,\, \forall \chi\in X^*(\T)^{\Gal(E/F)}\right\}
\\ &=\left\{z\in\bigoplus_{{{\,\iota}}\in {\Sigma}_E} (\overline{\Q}_{{{\,\iota}}})^\times\colon f_{{{\,\iota}}}(z)f_{{{\,\iota}}^\tau}(z)=1,\,\, \forall {{{\,\iota}}}\in {{\Sigma}_E}\right\}.
\end{align*}
Therefore, one can identify $X_*(\T^1)$ with $\{f\colon \Sigma_E \to \Z| f({\,\iota}) + f({\,\iota}^\tau)=0, \,\forall {{{\,\iota}}}\in {{\Sigma}_E}\}$. Recall, that for each $1\le i \le d$, we have fixed a $\widetilde{{\,\iota}}_i\in  \Sigma_E$ extending the fixed ${\,\iota}_i\in \Sigma_F$. 

In conclusion, we have a isomorphism of $\Gal(\overline{\Q}/\Q)$-modules between $X_*(\T^1)$ and $\Hom(\widetilde{\Sigma}_E,\Z)$ and thus an isomorphism 
$$\begin{tikzcd}\Hom(\widetilde{\Sigma}_E,\Z^{{\dim \star}})\arrow{r}{\simeq }&X_*(\T_{\mathfrak{B}_{\star}}).\end{tikzcd}$$
The representative ${\mu}_{\mathfrak{B}_{\star}}\in X_*(\T_{\mathfrak{B}_{\star}})$ of the class $\cX_{{\star},\overline{\Q}}$ corresponds, under the isomorphism above, to the function
$$f_{\star}\colon \widetilde{\Sigma}_E \to \Z^{{\dim \star}}, \quad \widetilde{{\,\iota}}_i\mapsto\begin{cases}(1,0,\cdots,0)& \text{ if } i=1,\\
(0,\cdots,0)&\text{ if } 2\le i\le d,\end{cases}$$
or, equivalently
$$f_{\star}\colon {\Sigma}_E \to \Z^{{\dim \star}}, \quad \widetilde{{\,\iota}}\mapsto\begin{cases}(1,0,\cdots,0)& \text{ if } \widetilde{{\,\iota}}=\widetilde{\iota}_1,\\
(-1,0,\cdots,0)& \text{ if } \widetilde{{\,\iota}}=\widetilde{\iota}_1\,^\tau,\\
(0,\cdots,0)&\text{ if }  \widetilde{{\,\iota}} \not\in\{ \widetilde{\iota}_1, \widetilde{\iota}_1\,^\tau\}. \end{cases}$$
Since all considered tori split over $E$, the absolute Galois group $\Gal(\overline{\Q}/\Q)$ will then act on these objects by its projection on $\Gal(E/\Q)$. The absolute Weyl group $W(\G_{\star},\T_{\mathfrak{B}_{\star}})\simeq \mathcal{S}_{{\dim \star}}^d$ acts on $\Hom(\Sigma_E,\Z^{{\dim \star}})$ by permuting the components in $\Z^{{\dim \star}}$, e.g. $W(\G_{\star},\T_{\mathfrak{B}_{\star}}) f_{\star}=\{f_{{\star},k}\colon \widetilde{\Sigma}_E \to \Z^{{\dim \star}}, \text{ with } {1\le k\le {\dim \star}}\}$ where
$$f_{{\star},k}\colon \widetilde{\Sigma}_E \to \Z^{{\dim \star}}, \widetilde{{\,\iota}}_i\mapsto \begin{cases}(0,\cdots,0,\underbrace{1}_{k^{\text{th}}\text{ position}},0,\cdots,0)& \text{ if } i=1,\\
(0,\cdots,0)&\text{ if } 2\le i\le d,\end{cases}.$$
Consequently, a Galois element fixing the Weyl orbit $W(\G_{\star},\T_{\mathfrak{B}_{\star}})\cdot f_{\star}$ must fix the ${\,\iota}_1$ component, hence is contained in $\Gal(\overline{\Q}/{\,\iota}_1(F))$. Now since, $f_{\star}^\tau\neq f_{\star}$, the Galois subgroup that fixes $W(\G_{\star},\T_{\mathfrak{B}_{\star}})\cdot f_{\star}$ is precisely $\Gal(\overline{\Q}/{\,\widetilde{\iota}}_1(E))$ and the field of definition of $W(\G_{\star},\T_{\mathfrak{B}_{\star}}) \cdot f_{\star}$ is
$$E(\G_{\star},\cX_{\star})=\widetilde{\iota}_1(E).$$
In conclusion, we also have $E(\H,\cY)=E(\G,\cX)=\widetilde{\iota}_1(E)$.
\subsection{Reflex norm maps}\label{Reflexnormmaps5}
For $\star\in \{V,W\}$, the discussion of \S \ref{reflexfield} shows that ${\mu}_{\mathfrak{B}_\star}=({h}_{\mathfrak{B}_\star})_\C\circ \mu \in X_*(\T_{\star})_\C$ is actually defined over $\widetilde{\iota}_1(E)=\widetilde{\iota}\,^\tau_1(E)$. Put $\iota_1\T:= \Res_{\widetilde{\iota}_1(E)/\Q}\mathds{G}_{m,\widetilde{\iota}_1(E)}$ and define the reflex norm map $r(\mu_{\mathfrak{B}_\star})$ to be the composition
$$\begin{tikzcd}[column sep= large]\T\arrow{r}{\jmath_1}[swap]{\simeq}&\iota_1\T \arrow{rr}{\Res_{\widetilde{\iota}_1(E)/\Q}(\mu_{\mathfrak{B}_\star})}&&\Res_{\widetilde{\iota}_1(E)/\Q}(\T_{\mathfrak{B}_\star})_{{\iota}_1(F)} \arrow{r}{\Norm}&\T_{\mathfrak{B}_\star} \arrow[two heads]{r}{\det} &\T^1\end{tikzcd}$$
Therefore, $r(\mu_{\mathfrak{B}_\star})= \Res_{F/\Q} \Nm_{\U_{E/F}(1)}\circ \jmath_1$, where $\Nm_{\U_{E/F}(1)}$ is 
$$\begin{tikzcd} \Res_{\widetilde{\iota}_1(E)/F}\mathds{G}_{m,\widetilde{\iota}_1(E)} (R)\arrow{r}&\U_{E/F}(1)(R),& s \arrow[r, mapsto]& \frac{s}{s^{\tau}}.\end{tikzcd}$$
for any $F$-algebra $R$. This shows, in particular, that $r(\mu_{\mathfrak{B}_\star})$ is independent of the choice of the fixed basis $\mathfrak{B}_\star$, from now on we will denote this map by $\nu:=r(\mu_{\mathfrak{B}_\star})\nomenclature[F]{$\nu$}{Reflex norm map $\T \to \T^1$}$. In addition, we have \footnote{On the one hand, the Weil restriction is left exact on the category of $\Q$-groups. On the other hand, the right surjectivity is a consequence of \cite[Corollary A.5.4(1)]{CGP10}.} an exact sequence of $\Q$-tori
\begin{equation}\label{exactsequencenu}\begin{tikzcd}1\arrow{r}& \Zbf \arrow{r}&\T\arrow{r}{\nu}&\T^1\arrow{r}& 1.\end{tikzcd}\end{equation}
The above discussion justifies the possibility of omitting the distinguished embeddings $\widetilde{\iota}_1$ and ${{\,\iota}}_1$. Accordingly, we identify the abstract number fields $F$ with ${\,\iota}_1(F)\subset \R$ and $E$ with $\widetilde{\iota}_1(E)=\widetilde{\iota}_1^\tau(E)\subset \C$.
\subsection{The Shimura varieties}\label{unitaryshimuravar}
Let us begin by proving some properties of the pairs $(\G,\cX)$ and $(\H,\cY)$:
\begin{proposition} The pairs $(\G_\star,\cX_\star)$, $\star\in\{V,W\}$ are Shimura data (See \cite[Definition 5.5]{milne:shimura}), that is 
\begin{enumerate}[nosep]
\item[SV1] The only characters of the induced representation $\Ad\circ h_{\mathfrak{B}_\star} \colon \mathds{S}\to\GL(\text{\em Lie}(\G_{\star,\C}))$, are $z\mapsto z/\overline{z},1, \overline{z}/z$.
\item[SV2] $\Ad\circ h$ is a Cartan involution of $\G_{\star,\R}^{\der}$ for all $h\in \cX_\star$.
\item[SV3] $\G_{\star}^{\ad}$ does not have any direct $\Q$-factor $\mathbf{L}$ on which $h_{\mathfrak{B}_\star}$ is trivial.
\end{enumerate}
\end{proposition}
\proof
For $(\G_\star,\cX_\star)$, $\star\in\{V,W\}$: 
\begin{enumerate}[topsep=0pt,itemsep=0pt,parsep=0pt,partopsep=0pt] 
\item[SV1] Recall that we have defined the $\R$-algebraic homomorphisms $h_{\mathfrak{B}_\star}\colon \mathds{S}\to \G_{\star,\R}\simeq \U(\dim_E\star-1,1)_\R\times\U(\dim_E \star)_\R^{d-1}$, given by
$$h_{\mathfrak{B}_\star}\colon \mathds{S}(\R)=\C^\times \ni z\mapsto \left(
\begin{pmatrix}
z/\overline{z}&\\
&1_{\dim_E\star-1}
\end{pmatrix},1_{\dim_E\star},\cdots,1_{\dim_E\star} \right)$$
with respect to the basis $\mathfrak{B}_\star$. It is then, straightforward to see that the only characters of the induced representation $\Ad\circ h_{\mathfrak{B}_\star} \colon \mathds{S}\to\GL(\text{Lie}(\G_{\star,\C}))$, are $z\mapsto z/\overline{z},1, \overline{z}/z$.
\item[SV2] The ss that $\Ad\circ h$ is a Cartan involution of $\G_{\star,\R}^{\der}$ for all $h\in \cX_\star$. Since, all Cartan involutions are conjugate by an inner automorphism, we just need to verify this axiom for $h=h_{\mathfrak{B}_\star}$, that is to show that the following Lie group
\begin{align*}\hspace{-2cm}\G_{\star,\R}^{\ad(h_{\mathfrak{B}_\star}(i))}(\R)&:=\{g\in\SU(\star)(\C)\colon g h_{\mathfrak{B}_\star}(i)=h_{\mathfrak{B}_\star}(i) \overline{g}\}\\
&=\{g\in\SU(\star)_{F\,,{\,\iota}_1}(\C)\colon g h_{\mathfrak{B}_\star}(i)=h_{\mathfrak{B}_\star}(i) \overline{g}\}\times \prod_{i=2}^{i=d} \SU(\star)_{F\,,{\,\iota}_i}(\C),\\
&\simeq \{g\in \SL(\star_{F\,,{\,\iota}_1}\otimes \C)\colon g J_{\mathfrak{B}_\star} {}^t\overline{g}=J_{\mathfrak{B}_\star}, gh_{\mathfrak{B}_\star}(i)=h_{\mathfrak{B}_\star}(i) \overline{g}\}  \times \SU(\dim_E\star) (\C)^{d-1}
\end{align*}
is compact. Indeed, for all $2\le i \le d$, the subgroup $\SU(\star)_{F\,,{\,\iota}}(\C)\simeq \SU(\dim_E\star) (\C)$ is compact. While, for ${\,\iota}_1$, the fact that $h_{\mathfrak{B}_\star}(i)=J_{\mathfrak{B}_\star}$ shows\footnote{The two conditions above, gives the equality $g J_{\mathfrak{B}_\star} {}^t\overline{g}=J_{\mathfrak{B}_\star} \overline{g}{}^t\overline{g}=J_{\mathfrak{B}_\star}$, hence ${g}{}^t{g}=\text{Id}_{\star_{{\,\iota}_1,\R}\otimes \C}$.}
$\{g\in \SL(\star_{F\,,{\,\iota}_1}\otimes \C)\colon g J_{\mathfrak{B}_\star} {}^t\overline{g}=J_{\mathfrak{B}_\star} \text{ and } gh_{\mathfrak{B}_\star}(i)=h_{\mathfrak{B}_\star}(i) \overline{g}\}$ is equal to $$\SU(\star_{F\,,{\,\iota}_1})(\C)\cap \mathbf{O}(\star_{F\,,{\,\iota}_1}\otimes \C),$$
and this latter is a closed subgroup of the compact orthogonal group {$\mathbf{O}(\star_{F\,,{\,\iota}_1}\otimes \C)$}. This proves the claim.
\item[SV3] The third axiom requires that $\G_{\star}^{\ad}$ does not have any direct $\Q$-factor $\mathbf{L}$ on which $h_{\mathfrak{B}_\star}$ is trivial. By \cite[Remark 4.6.]{milne:shimura}, this holds if and only if $\G_{\star}^{\ad}$ is of noncompact type \cite[Definition 3.18]{milne:shimura}. Consider the isogeny $\G_{\star}^{\der}=\SU(\star)\to \G_{\star}^{\ad}$, clearly $\SU(\star)(\R)$ is noncompact due to the signature of $(\star_{F, {\,\iota}_1},\psi_1)$ which is $(\dim_E\star-1,1)$. This proves $\G_{\star}^{\ad}$ is of noncompact type.\qed\end{enumerate}
The proposition above, shows that 
$$ (\G,\cX)= (\G_V,\cX_V)\times (\G_W,\cX_W)\text{ and }(\H,\cY)=\Delta(\G_W,\cX_W)$$
are also Shimura data and implies, in particular, that the connected\footnote{The identification we have seen in \S \ref{Hermsymdom} between $\cX_V$ and $\cX_W$ and complex open balls shows that these spaces are connected and consequently $\cX$ and $\cY$ are connected too.} spaces $\cX_V,\cX_W,\cX$ and $\cY$ are Hermitian symmetric domains. 
\begin{remark}\label{discreetnessT1}\label{additionalaxioms} Let $\mathbf{X} \in \{V,W\}$, then the Shimura datum $(\G_{\star}, \cX_{\star})$ verifies the following additional axioms \emph{\cite[Additional axioms, p. 63]{milne:shimura}}:
\begin{enumerate}[nosep]
\item[SV4] The weight homomorphism
$$\begin{tikzcd}w_{\cX_{\star}}\colon \mathds{G}_{m,\R} \arrow{r}{r \mapsto r^{-1}}& \mathds{S} \arrow{r}{h_{\mathfrak{B}_V}}&\G_{{\star},\R} \end{tikzcd}$$
is clearly trivial\footnote{As for any "connected" Shimura variety.}, since $h_{\mathfrak{B}_V}$ factors through $\U(1)$, thus it is tautologically defined over $\Q$.
\item[SV5] The rational points of the center $\T^1(\Q)$ are discrete in $\T^1(\A_f)$. Indeed, $\T^1$ is anisotropic over $\Q$ and remains anisotropic over $\R$ too, thus the set of real points $\T^1(\R)$ must be compact which shows that $ \T^1(\Q)$ is discrete in $\T^1(\A_f)$ \emph{\cite[\S 5 - Arithmetic subgroups of tori]{milne:shimura}}.
\item[SV6] By definition, the center $\T^1$ splits over the CM-field $E$.\qedhere
\end{enumerate}
\end{remark}
For every compact subgroup $K_\G \subset \G(\A_f)$ (resp. $K_\H\subset \G(\A_f)$), the Shimura variety $\Sh_{K_\G}(\G,\cX)$ (resp. $\Sh_{K_\H}(\H,\cY)$) is the complex analytic space $$\G(\Q)\backslash(\cX \times (\G(\A_f)\slash K_\G))\quad(\text{resp. }\H(\Q)\backslash(\cY \times (\H(\A_f)\slash K_\H))),$$ where $\G(\Q)$ (resp. $\H(\Q)$) acts diagonally on $\cX \times (\G(\A_f)\slash K_\G)$ (resp. $\cY \times (\H(\A_f)\slash K_\H))$.
\begin{proposition}
The Shimura variety $\Sh_{K_\G}(\G,\cX)$ (resp. $\Sh_{K_\G}(\H,\cY)$) has the following decomposition
$$\bigsqcup_{g\in \mathcal{C}_\G} \Gamma_g\backslash\cX \quad(\text{resp.}\quad \bigsqcup_{h\in \mathcal{C}_\H} \Gamma_h\backslash\cY),$$
where for ${\textbf{X}}\in \{\G,\H\}$, we have used $\mathcal{C}_{\textbf{X}}\nomenclature[F]{$\mathcal{C}_{\textbf{X}}$}{The finite class group ${\textbf{X}}(\Q)\backslash {\textbf{X}}(\A_f)/K_{\textbf{X}}$ for ${\textbf{X}}\in \{\G,\H\}$}$ to denote the finite class group ${\textbf{X}}(\Q)\backslash {\textbf{X}}(\A_f)/K_{\textbf{X}}$ and for $ x\in \mathcal{C}_{\textbf{X}}$ we put $\Gamma_x:=x K_{\textbf{X}}x^{-1} \cap {\textbf{X}}(\Q)\subset {\textbf{X}}(\R)$.  
\end{proposition}
\begin{proof}
This is an application of \cite[Lemma 5.13.]{milne:shimura}.
\end{proof}
The subgroups $\Gamma_x$ for $x\in \mathcal{C}_\G$ (resp. $x\in\mathcal{C}_\H$) are congruence\footnote{We have an embedding $\pi\colon\G\hookrightarrow \GL(V\oplus W)\simeq \GL_{2n+1}$. A congruence subgroup $\Gamma\subset \G(\Q)$ is a subgroup containing a finite index subgroup of the form $$\Gamma(N)=\pi(\G)(\Q)\cap \{g\in \GL_{2n+1}(\Z)\colon g\equiv I_{2n+1} \mod N \},$$ for some integer $N$. The subgroup $\Gamma\subset \G(\Q)$ is arithmetic if $\pi(\Gamma)$ is commensurable with $\pi(\G)(\Q)\cap \GL_{2n+1}(\Z)$. Arithmetic and congruence subgroups of $\H(\Q)$, are defined similarly.} arithmetic subgroups of $\G(\Q)$ (resp. $\H(\Q)$) (See  \cite[Proposition 4.1]{milne:shimura}).

By the work of Baily and Borel \cite{BB66}, each connected component $\Gamma_g\backslash\cX$, (${g\in \mathcal{C}_\G}$) (resp. $\Gamma_h\backslash\cY$, (${h\in \mathcal{C}_\H}$)) can be endowed with a natural structure of a complex quasi-projective variety, hence also $\Sh_{K_\G}(\G,\cX)$ and $\Sh_{K_\H}(\H,\cY)$. Moreover, if $\Gamma_x$ (for $x$ in $\mathcal{C}_\G$ or $\mathcal{C}_\H$) is small enough (for example, if it is torsion-free) then $\Gamma_g\backslash\cX$ is smooth and its algebraic structure is unique. In this paper, we will be considering a stronger condition on the compact open subgroups $K_\G \subset  \G(\A_f)$ and $K_\H \subset \H(\A_f)$, namely being neat \cite[\S 0.1]{Pink89}, it prevents $\G(\Q)$ (resp. $\H(\Q)$) from having fixed point in $\cX\times \G(\A_f)/K$ (resp. ($\cY\times \H(\A_f)/K_\H$)), in which case $\Sh_{K_\G}(\G,\cX)$ and $\Sh_{K_\H}(\H,\cY)$ are smooth.

We have a better understanding of the algebraic structure governing the Shimura varieties $\Sh_{K_\G}(\G,\cX)$ and $\Sh_{K_\H}(\H,\cY)$. Indeed, It follows from results of Deligne, Borovoi, Milne and Moonen \cite{deligne:travaux,deligne:shimura,Bor83,Milne1983,Milne:1999,Moonen98}, that every Shimura varietiy has a canonical model \footnote{This canonical model depends only on the associated Shimura data. In particular, canonical models do not depend on the chosen open compact.} defined over its reflex field.

From now on, we will reserve the notation $\Sh_{K_\G}(\G,\cX)$ and $\Sh_{K_\H}(\H,\cY)$ for the corresponding canonical model over the reflex field $E(\G,\cX)=E(\H,\cY)=E$. Therefore, the initial definition given in terms of double cosets gives actually the complex points of these models:$$\Sh_{K_\G}(\G,\cX)(\C)=\G(\Q)\backslash(\cX \times (\G(\A_f)\slash K_\G)),\quad\Sh_{K_\H}(\H,\cY)(\C)=\H(\Q)\backslash(\cY \times (\H(\A_f)\slash K_\H)).$$
\subsection{The projective system and Hecke action}\label{ProjsysandHecke}
For any two neat compact open subgroups $K'\subset K$, we have an obvious quotient map $$\pi_{K',K}\colon \Sh_{K'}(\G,\cX)(\C)\to \Sh_{K}(\G,\cX)(\C),$$ which defines a finite \'{e}tale morphism between $\Sh_{K'}(\G,\cX) \to \Sh_{K}(\G,\cX)$. Taking the projective system over neat compact subgroups, $$\Sh(\G,\cX)=\varprojlim_{K \text{ neat}} \Sh_K(\G,\cX),$$ we obtain a a scheme over $\C$, endowed with a continuous action of $\G(\A_f)$ (see \cite[2.1.4 and 2.7]{deligne:shimura} and \cite[II.2 and II.10]{Milne90}). Define also the quotient map $\pi_{\G,K}\colon \Sh(\G,\cX)\longrightarrow \Sh(\G,\cX)/K=\Sh_K(\G,\cX)$.
The action of $g\in \G(\A_f)$ denoted $T_g\colon \Sh(\G,\cX) \to \Sh(\G,\cX) $ on the system $(\Sh_{K}(\G,\cX))_{K \text{ neat}}$ defines an isomorphism of algebraic varieties between $\Sh_{K}(\G,\cX)$ and $\Sh_{g^{-1}Kg}(\G,\cX)$ (See \cite[Remark 5.29 \& Theorem 13.6]{milne:shimura}). 
For a fixed finite level structure $K\subset \G(\A_f)$, The action of $T_g$ on $\Sh(\G,\cX)$ descends to a correspondence $\mathcal{T}_g\subset \Sh_{K}(\G,\cX)\times_E \Sh_{K}(\G,\cX)\nomenclature[F]{$\mathcal{T}_g$}{Hecke correspondence for $g\in \G(\A_f)$}$ given by the diagram:
$$\begin{tikzcd}\Sh(\G,\cX)\arrow{rr}{T_g}\arrow{d}{\pi_{\G,K}} && \Sh(\G,\cX)\arrow{d}{\pi_{\G,K}}\\\hspace*{\fill}
\Sh_{K}(\G,\cX)&& \Sh_{K}(\G,\cX)  \end{tikzcd}$$
The correspondence $\mathcal{T}_g=(\pi_{\G,K})_*\cdot T_g\cdot (\pi_{\G,K})^*$ is finite, i.e. $$\mathcal{T}_g\in \mathcal{C}_{\text{fin}}(\Sh_{K}(\G,\cX)(\C)\times_E \Sh_{K}(\G,\cX)(\C)).$$
It is called the Hecke correspondence and is usually given by the diagram:$$\begin{tikzcd}[]  \Sh_{K_g}(\G,\cX)\arrow{d}{{\pi_{K_g,K}}}\arrow{r}{\times g}&\Sh_{K_g^{-1}}(\G,\cX) \arrow{d}{{\pi_{K_{g^{-1}},K}}}\\
\Sh_{K}(\G,\cX)& \Sh_{K}(\G,\cX)\end{tikzcd}$$
where, we have used the notation $K_g=K\cap gKg^{-1}$ for (any) $g\in \G(\A_f)$.  
\subsection{Deligne's Reciprocity law for tori}\label{reciprocitytori}
\subsubsection{Artin map}\label{generalartinmap}
Let $L$ be any number field and set $\mathbf{X}_L:=\Res_{L/\Q} \mathds{G}_{m,L}$. Class field theory states the existence of the so called Artin map, that is a continuous surjective homomorphism
\[\begin{tikzcd}  \Art_L\colon \mathbf{X}_L(\A_{\Q}) \arrow[r, twoheadrightarrow]& \Gal(L^{ab}/L),\end{tikzcd}\nomenclature[F]{$\Art_L$}{Artin map for any number field $L$}\]
sending uniformizers to geometric Frobenius elements.
Let $\mathbf{X}_L(\R)^+$ denotes the identity component for the real points and $\mathbf{X}_L(\Q)^+=\ker(\mathbf{X}_L(\Q) \to \pi_0(\mathbf{X}_L(\R)))$. The kernel of $\Art_L$ is precisely the closure\footnote{Here, the subgroups $\mathbf{X}_L(\R)^+ $ is identified with the subgroup of $\mathbf{X}_L(\A_\Q)$ of idele elements $\mathbf{z}=(z_v)$ in $\A_L^\times$ such that $z_v=1$ if the place $v$ is finite and $z_v>0$ if the place $v$ is real.} of $\mathbf{X}_L(\Q) \mathbf{X}_L(\R)^+$, see {\cite[\S2 Chap. VIII]{NSW2008}}. 
{The cited theorem affirms that the closure of the above product in the ideal class group is the kernel. Indeed the quotient map $\mathbf{X}_L(\A_\Q) \twoheadrightarrow \mathbf{X}_L(\A_\Q)/\mathbf{X}_L(\Q)$ is open and continuous, thus the preimage of the closure of $\mathbf{X}_L(\Q) \mathbf{X}_L(\R)^+\subset \mathbf{X}_L(\A_\Q)/\mathbf{X}_L(\Q)$ is the image of the closure of $\mathbf{X}_L(\Q) \mathbf{X}_L(\R)^+\subset \mathbf{X}_L(\A_\Q) $. Recall that by definition of $\Art_L$, we have $\mathbf{X}_L(\Q)$ is in its kernel and clearly $\mathbf{X}_L(\R)^+ $ is also in the kernel, therefore since the homomorphism $\Art_L$ is continuous and the target group $\Gal(L^{ab}/L)$ is Hausdorff the kernel must be closed and hence contains $\left(\mathbf{X}_L(\Q) \mathbf{X}_L(\R)^+\right)^-\subset  \mathbf{X}_L(\A_\Q)$ and this shows the equality.}

By the real approximation theorem (See \cite[Appendix, p. 152]{milne:shimura}), $\mathbf{X}_L(\Q)$ is dense in $\mathbf{X}_L(\R)$ 
thus one has a surjective map $\mathbf{X}_L(\Q) \twoheadrightarrow \pi_0(\mathbf{X}_L(\R))$, thus $\mathbf{X}_L(\R)/\mathbf{X}_L(\R)^+ \simeq \pi_0(\mathbf{X}_L(\R)) \simeq \mathbf{X}_L(\Q)/\mathbf{X}_L(\Q)^+\simeq \prod_{v\in \Hom_\Q (L,\R)} \{\pm\}$. 
{\begin{lemma}
The closure $\mathbf{X}_L(\Q)^-\subset \mathbf{X}_L(\A_f)$ is equal to $L^\times (\O_L^\times)^-$.
\end{lemma}
\proof Let $x=(x_v) \in \mathbf{X}_L(\Q)^-\subset \mathbf{X}_L(\A_f)$ and consider $(x_n=(x_{v,n})_v)_{n\in \N})\in \mathbf{X}_L(\Q)$ a sequence that converges to $x$. There exists an open neighbourhood of $x$ of the form $\cO_x=\cO_x^S \times \prod_{v\not\in S} \cO_{L_v}^\times$ for some open subgroup $\cO_x^S \subset L^{S,\times}$ such that for $n\ge N=N(\cO_x)$ one has $x_n \in \cO_x$ and $\ell_x := \prod_{v\in S} x_v$ lies in $\mathbf{X}_L(\Q) =L^\times$. For each $v\in S$, $x_{v,n}$ being convergent to $x_v$, then $|x_{v,n}|_v= |x_v|_v$ is constant for $n$ greater than some integer $N_v$ that we choose to be greater than $N$. Put for $n\ge N_x:=\max_{v\in S} N_v$,
$$y_n:= \ell_x^{-1}x_{v,n}=  (\frac{1}{x_v}x_{v,n})_{v\in S} \otimes (x_{v,n})_{v\not \in S}\in \widehat{\cO}_L^\times \cap L^\times =\O_L^\times.$$
Then, one has
\begin{align*} x&= \ell_x \lim_{n\to \infty} y_n \in L^\times (\cO_L^\times)^-.\qedhere \end{align*}
Consider the homomorphism 
$$\Art_{L,f}\colon\mathbf{X}_L(\A_f) \to\mathbf{X}_L(\A_\Q)/ (\mathbf{X}_L(\Q)\mathbf{X}_L(\R)^+)^-, z \mapsto [(1_\infty, z)].$$
Let us show that it is surjective. Using the density of $\mathbf{X}_L(\Q)$ in $\mathbf{X}_L(\R)$, observe that:
\begin{align*}\mathbf{X}_L(\Q)\mathbf{X}_L(\R)^+=\{(x,\Delta y)\in \mathbf{X}_L(\R)\times \mathbf{X}_L(\A_f) \colon & x\in \mathbf{X}_L(\R), y \in \mathbf{X}_L(\Q) \text{ s.t. } \\&\pi_0(x)=\pi_0(y)\in \pi_0(\mathbf{X}_L(\R)) \}. (\star)\end{align*}
Let $z=(z_\infty,z_f) \in \mathbf{X}_L(\A_\Q)$, we may modify $z$ by an element in $\mathbf{X}_L(\Q)\mathbf{X}_L(\R)^+$ to kill $z_\infty$. 
For this, consider $z'=(z_\infty^{-1},\Delta z_f')$ where $z_f'\in \mathbf{X}_L(\Q) $ is any element such that $\pi_0(z_f')=\pi_0(z_\infty'^{-1})$ (such element exists by density of the $\Q$-points in the $\R$-points of $\mathbf{X}_L$). 
Hence, $$z \mod \mathbf{X}_L(\Q)\mathbf{X}_L(\R)^+ = zz' \mod \mathbf{X}_L(\Q)\mathbf{X}_L(\R)^+,$$ and $zz' \in \mathbf{X}_L(\A_f)$. Accordingly 
$$[z]  \in \text{Im} (\Art_{L,f}),$$
which shows the surjectivity of $\Art_{L,f}$. Let us show that the kernel 
$$\ker \Art_{L,f}=\mathbf{X}_L(\A_f) \cap (\mathbf{X}_L(\Q)\mathbf{X}_L(\R)^+)^-=(\mathbf{X}_L(\Q)^+)^-,$$
where, the super-script ${}^-$ denotes the closure in $\mathbf{X}_L(\A_{\Q})$ first, then in $\mathbf{X}_L(\A_{f})$. 
Let $z \in (\mathbf{X}_L(\A_f) \cap \mathbf{X}_L(\Q)\mathbf{X}_L(\R)^+)^-$, 
one then has (using $\star$) a sequence of $(x_n, \Delta y_n) \in \mathbf{X}_L(\Q)\mathbf{X}_L(\R)^+$ converging to $z$, so in particular $\lim_{n\to \infty}x_n =1$. Hence, there is an integer $N$ big enough such that for $n\ge N$ one has $\pi_0(x_n) = \mathbf{X}_L(\R)^+$, 
and accordingly $\pi_0(y_n) = \mathbf{X}_L(\R)^+$, i.e. $y_n \in \mathbf{X}_L(\Q)^+$. 
Therefore, 
$$z=\lim_{n\to \infty} (x_n,y_n) = \lim_{n\to \infty} (1_\infty,y_n)$$
which shows that $z\in  (\mathbf{X}_L(\Q)^+)^-$ where the closure is taken in $\mathbf{X}_L(\A_f)$. 
This shows that $\ker \Art_{L,f} \subset (\mathbf{X}_L(\Q)^+)^-$. The other inclusion is obvious.}

In summary, this shows that
\begin{align*}\Gal(L^{ab}/L) &\simeq \pi_0(\mathbf{X}_L(\A_{\Q})/\mathbf{X}_L(\Q))
\simeq \mathbf{X}_L(\A_{f})/ (\mathbf{X}_L(\Q)^+)^-.
\end{align*}
By the above discussion we get an exact sequence
\begin{equation}\label{exactsequenceartin}\begin{tikzcd} 1\arrow{r} & (\mathbf{X}_L(\Q)^+)^- \arrow{r}&  \mathbf{X}_L(\A_f)\arrow{r}{r_L}& \Gal(L^{ab}/L) \arrow{r}&1.\end{tikzcd}\end{equation}
It follows from the description of the connected components of Id\`{e}le class group in \cite[\S 9, Theorem 3]{artin-tate:cft} that the kernel of $r_L$ is isomorphic, as $\text{Aut}(L/\Q)$-module, to 
$$L^{\times,+}\otimes (\A_f/\Q)\simeq L^{\times,+} (\cO_L^{\times,+}\otimes (\A_f/\Q)) \simeq L^{\times,+}(\cO_L^{\times,+}\otimes (\widehat{\Z}/\Z))$$
yielding the exact sequence
$$\label{exactsequenceartin2}\begin{tikzcd} 1\arrow{r}& \cO_L^{\times,+}\otimes (\A_f/\Q) \arrow{r}&  \A_{L,f}^\times/L^{\times,+}\arrow{r}{r_L}& \Gal(L^{ab}/L) \arrow{r}&1.\end{tikzcd}$$
\subsubsection{The cases \textbf{T} and \textbf{Z}}\label{caseTZ}
Now let us consider our two fields $E$ and $F$:
\[\begin{tikzcd}  \Art_E\colon \T(\A_{\Q}) \arrow[r, twoheadrightarrow]& \Gal(E^{ab}/E), & \Art_F\colon \Zbf(\A_{\Q}) \arrow[r, twoheadrightarrow]& \Gal(F^{ab}/F).\end{tikzcd}\]
The kernel of $ \Art_E$ (resp. $ \Art_F$) is the closure of $\T(\R)^+ \T(\Q) = \T(\R)\T(\Q) \simeq E^\times (\C^\times)^d $ (resp. $ \Zbf(\R)^+ \Zbf(\Q)\simeq  (\R_{>0}^\times)^d F^\times \simeq (\R^\times)^d (F^\times)^+$), in $\T(\A_\Q)^+$ (resp. $\Zbf(\A_\Q)$), here we have used the fact that $E$ is totally imaginary and that $F$ is totally real. 
Therefore,
\begin{align*}\Gal(E^{ab}/E) 
&\simeq  \A_{E,f}^\times / E^\times(\O_E^\times)^-,
\end{align*}
where, $(\O_E^\times)^-$ denotes the closure of $\O_E^\times$ in $\widehat{\O_E}^\times$. Likewise,
\begin{align*}\Gal(F^{ab}/F)
&\simeq \A_{F,f}^\times /  F^{\times,+}(\cO_F^{\times,+})^-.
\end{align*}

\subsubsection{Deligne's reciprocity law for $\T^1$}
Consider the "zero"-dimensional Shimura datum $(\T^1,\{\det {\mu}_{\mathfrak{B}_\star}\})$. Using the reflex norm map computed in \S \ref{Reflexnormmaps5} 
 $$\begin{tikzcd}\nu \colon \T\arrow[two heads]{rr}{ z \longmapsto \frac{z}{\overline{z}}} && \T^1\end{tikzcd}$$
we obtain the following reciprocity map
$$r=r(\T^1,\{\det {\mu}_{\mathfrak{B}_\star}\}) \colon\Gal(E^{ab}/E)\longrightarrow \pi_0(\T^1(\A_\Q)/\T^1(\Q))\simeq \T^1(\A_f)/\T^1(\Q)^-,$$
by composing the inverse of $\Art_E\colon   E^{\times} (\O_E^\times)^-\backslash \A_E^\times\xrightarrow{\simeq}  \Gal(E^{ab}/E)$ and $\nu$. But, since $\T^1(\Q)$ is discrete and hence closed in $\T^1(\A_f)$(See remark \ref{discreetnessT1}), the target of the map $r_{\text{fin}}(\T^1,\{\det \mu_{\mathfrak{B}_W}\})$ is actually $\T^1(\A_f)/\T^1(\Q)$.

 Let $K_{\T^1}\subset \T^1(\A_f)$ be a open compact subgroup. The action of $\sigma \in \Gal(E^{ab}/E)$ on $$[\det {\mu}_{\mathfrak{B}_\star},t]\in\Sh_{K_{\T^1}}(\T^1,\det {\mu}_{\mathfrak{B}_\star})(\C)=\T^1(\Q)\backslash (\{\det {\mu}_{\mathfrak{B}_\star}\}\times \T^1(\A_f)/K_{\T^1})$$ is defined by 
 $$\sigma([\det {\mu}_{\mathfrak{B}_\star},t]):= [{r_{\infty}}(\sigma) \det{\mu}_{\mathfrak{B}_\star}, r_{\text{fin}}(\sigma) \, t].$$
 Here, we have ${r_{\infty}}(\sigma) \det{\mu}_{\mathfrak{B}_\star} = \det\mu_{\mathfrak{B}_\star}$ and the map $r$ factors then through its finite part $r_{\text{fin}}=r_{\text{fin}}(\T^1,\{\det\mu_{\mathfrak{B}_\star}\})\colon \A_{E}^\times \longrightarrow \T^1(\A_f)$.
\subsection{Galois action on connected components}\label{galoisconnectedcomp}
In this section we describe the action of $\Gal(E^{ab}/E)$ on the connected component of $\Sh_{K_\H}(\H,\cY)$, for a fixed open compact subgroup $K_\H$ of $\H(\A_f)$. The derived group $\G^{\text{der}}$ is simply connected\footnote{Indeed, $\G^{\text{der}}_{\overline{\Q}}$ is isomorphic to $\SL_{n,\overline{\Q}}$.}. Therefore, one has a simplified description of the connected components of the Shimura variety:
\begin{align*}\pi_0(\Sh_{K_\H}(\H,\cY))&=\pi_0\left(\H(\Q)\backslash (\cY\times (\H(\A_f)/K_\H))\right)\\\hspace*{\fill}
&\simeq\H(\Q)\backslash  \H(\A_f)/K_\H\\\hspace*{\fill}
&=\H(\Q)\backslash  \H(\A_f)/K_\H \, \H^{\der}(\A_f)\\\hspace*{\fill}
&\simeq {\T^1}(\Q)\backslash \T^1(\A_f)/\det(K_\H)\\\hspace*{\fill}
&=\Sh_{\det(K_\H)}(\T^1,\det {\mu}_{\mathfrak{B}_W})(\C)
\end{align*}
The previous equalities uses the connectedness of $\cY$ first, then density of $\H^{\der}(\Q)$ in $\H^{\der}(\A_f)$ (See \cite[\S 5 - The structure of a Shimura variety]{milne:shimura}). 

Using the action of Galois group $\Gal(E^{ab}/E)$ on the complex points of $\Sh_{det(K_\H)}(\T^1,\det {\mu}_{\mathfrak{B}_W})$ as described in \S \ref{reciprocitytori}, we obtain that the action of $\sigma \in \Gal(E^{ab}/E)$ on, $$[\det {\mu}_{\mathfrak{B}_W},t]\in \pi_0(\Sh_{K_\H}(\H,\cY))\simeq \Sh_{\det(K_\H)}(\T^1,\{\det\mu_{\mathfrak{B}_W}\})(\C),$$ is given by
$\sigma\left([\det {\mu}_{\mathfrak{B}_W},t]\right)= [\det {\mu}_{\mathfrak{B}_W}, r_{\text{fin}}(\sigma) \, t]$.

\section{Preliminary results on special cycles}
\subsection{The set of special cycles \texorpdfstring{$\cZ_{\G,K}(\H)$}{ZGK(H)}}\label{defcycles}
The inclusion homomorphism $\H\hookrightarrow \G$, induces a map $\H(\R) \hookrightarrow \G(\R)$ that sends the Hermitian symmetric domain $\cY$ into $\cX$. Therefore, we get a morphism of Shimura data $\varphi\colon(\H,\cY)\rightarrow(\G,\cX)$ in the sense of \cite[Definition 5.15]{milne:shimura}.
\begin{remark}\label{immersionshim} By \cite[Theorem 1.15]{deligne:travaux}, the injective morphism of Shimura data $\varphi\colon(\H,\cY)\rightarrow(\G,\cX)$ induces a closed immersion of Shimura varieties $\Sh(\varphi)\colon\Sh(\H,\cY)\rightarrow\Sh(\G,\cX)$ meaning: For every "sufficiently" small compact open subgroup $K_\G\subset \G(\A_f)$, there is a compact open subgroup $K_\H\subset \H(\A_f)$ such that the map
$${\Sh(\varphi)_{K_\H,K_\G}}\colon\Sh_{K_\H}(\H,\cY)\rightarrow\Sh_{K_\G}(\G,\cX),$$
given naturally on $\C$-points by $\H(\Q)(y,hK_\H)  \mapsto  \G(\Q)(y,hK_\G)$ for any $y \in \cY$ and $h\in \H(\A_f)$, is a closed immersion. We then get the following commutative diagram 
$$\begin{tikzcd}
    \Sh(\H,\cY)\arrow{rr}{\Sh(\varphi)}\arrow{d}{\pi_{\H,K_\H}} && \Sh(\G,\cX)\arrow{d}{\pi_{\G,K_\G}}\\\hspace*{\fill}
       \Sh_{K_\H}(\H,\cY)\arrow{rr}{\Sh(\varphi)_{K_\H,K_\G}} && \Sh_{K_\G}(\G,\cX).
\end{tikzcd}$$
To ease the notation we will omit the subscripts referring to the compacts in the map $\Sh(\varphi)_{K_\H,K_\G}$.
\end{remark}

For the remaining of this section, we will fix a neat compact open subgroup of $K\subset\G(\A_f)$.
\begin{definition}[Special cycles]\label{specialcycles} We call a closed subvariety $\cZ\subset \Sh_K(\G,\cX)$ a $\H$-{special cycle}, if there exists an element $g\in \G(\A_f)$ such that $\cZ$ is an irreducible component of the image of the map
$$\begin{tikzcd}
    \Sh(\H,\cY)\arrow{r}{\Sh(\varphi)} &\Sh(\G,\cX)\arrow{r}{T_g}& \Sh(\G,\cX)\arrow{r}{\pi_{\G,K}}& \Sh_K(\G,\cX).
\end{tikzcd}$$
\end{definition}
\begin{lemma}\label{specialcycles2}
A closed subvariety $\cZ\subset \Sh_K(\G,\cX)$ is a $\H$-{Special cycle} if and only if there exists an element $g\in \G(\A_f)$, such that  $$\cZ=[\cY\times gK]\subset \Sh_K(\G,\cX)(\C).$$
\end{lemma}
\begin{proof}
The equivalence above follows immediately, using Remark \ref{immersionshim} and by observing that 
$$\cZ=T_g\circ \Sh(\varphi)(\H(\Q)(K_{\H,g}\times \cY)),$$
where $K_{\H,g}:=\H(\A_f)\cap gKg^{-1}$, see \cite[Remark 2.6]{linearityI}.
\end{proof}
For every $g\in \G(\A_f)$, we will denote by $\mathfrak{z}_g\nomenclature[F]{$\mathfrak{z}_g$}{The ${n}$-codimensional $\H$-special cycle $[\cY\times gK] \subset \Sh_K(\G,\cX)(\C)$}$ the ${n}$-codimensional $\H$-special cycle $[\cY\times gK] \subset \Sh_K(\G,\cX)(\C)$, as defined in Definition \ref{specialcycles}. Set,
$$\cZ_{\G,K}(\H):=\{\mathfrak{z}_g: g \in \G(\A_f)\}$$
\begin{lemma}The natural map $\G(\A_f)\to \cZ_{\G,K}(\H)$, induces the bijection
$$\cZ_{\G,K}(\H) \simeq \H(\Q){Z}_{\G}(\Q)\backslash\G(\A_f)/K,$$
where, ${Z}_{\G}\simeq \T^1\times \T^1$ denotes the center of $\G$.
\end{lemma}
\begin{proof}
Using the total geodesicity of $\cY$ and the Baire's category theorem, one proves that $$\cZ_K \simeq \text{Stab}_{\G(\Q)}(\cY)\backslash \G(\A_f)\slash {K}.$$
Then, one shows that $\text{Stab}_{\G(\Q)}(\cY)=\H(\Q){Z}_{\G}(\Q)=N_{\G}(\H)(\Q)$. For more details, see \cite[Lemma 2.3]{jetchev:unitary}. 
\end{proof}
\begin{lemma}
In the following bijection
$$\cZ_{\G,K}(\H) \simeq \H(\Q)\mathbf{Z}_{\G}(\Q)\backslash\G(\A_f)/K,$$
we can replace $\H(\Q)\mathbf{Z}_{\G}(\H)(\Q)$ by its closure in $\G(\A_f)$:
$$\cZ_{\G,K}(\H) = \left(\H(\Q)\mathbf{Z}_{\G}(\Q)\right)^-\backslash\G(\A_f)/K= \mathbf{Z}_{\G}(\Q) \H(\Q)\H^{\text{der}}(\A_f)\backslash\G(\A_f)/K.$$
\end{lemma} 
\begin{proof}The second\footnote{The first equality uses the fact that if $G$ is a topological group (might even be locally compact) and $K$ an open compact subgroup (being open implies the discreteness of the quotient $G/K$), then for every subgroup $H$ of $G$ we have $H\backslash G/K=\overline{H}\backslash G/K$, where $\overline{H}$ is the closure of $H$ in $G$.} equality is due to the fact that the closure of $\H(\Q)$ is $\H(\Q)\H^{\text{der}}(\A_f)$. 
The latter fact is \cite[corollaire 2.0.9]{deligne:shimura}, because $\H^{\text{der}}$ verifies the strong approximation for $\{\infty\}$ since by assumption $\H^{\text{der}}$ is simply connected, semisimple by definition and of noncompact type. Observe also, that $\mathbf{Z}_{\G}(\Q)^-=\mathbf{Z}_{\G}(\Q)$, since it is a copy of $\T^1(\Q)$, which is discrete (Remark \ref{discreetnessT1}) and thus closed in $\T^1(\A_f)$. \end{proof}

\subsection{Fields of definition of cycles in \texorpdfstring{$\cZ_{\G,K}(\H)$}{ZGK(H)}}\label{fieldofdefinition1}
\begin{lemma}\label{cycleconcomp}
For every $g\in \G(\A_f)$, the cycle $\mathfrak{z}_g$ is the image of the connected component 
$$\H(\Q)\cap K_{g,\H}\backslash \cY\in \pi_0(\Sh_{K_{g,\H}}(\H,\cY)).$$
\end{lemma}
\begin{proof}
For obvious topological reasons, $K_{g,\H}$ is a compact open subgroup of $\H(\A_f)$. The pre-image of $\mathfrak{z}_g$ is the component over $[(1,1)]$: $$\H(\Q)\cap K_{g,\H}\backslash \cY\simeq \Sh_{K_{\H,g}^{\text{der}}}^\circ(\H^{\text{der}},\cY).$$
Here, $K_{\H,g}^{\text{der}}\subset \H^{\text{der}}(\A_f)$ is some open compact subgroup containing $K_{\H,g}\cap \H^{\text{der}}(\A_f)$.
\end{proof}
Recall that, for any neat compact subgroup $K_{\H}\subset\H(\A_f)$, we have 
$$\pi_0(\Sh_{K_{\H}}(\H,\cY))\simeq {\T^1}(\Q)\backslash \T^1(\A_f)/\det(K_{\H}).$$
The set of classes $\T^1(\Q)\backslash \T_\H(\A_f )/K_{\T^1}$ of $\T^1$ with respect to any compact open subgroup $K_{\T^1}\subset \T^1(\A_f)$ form an \emph{Abelian group}. Therefore, the connected components of $\Sh_{K_{\H}}(\H,\cY)$ are all defined over abelian extensions of $E$. More precisely the field of definition $E_g\nomenclature[F]{$E_g$}{Field of definition of $\mathfrak{z}_g$}$ of the component $\Sh_{K_{\H}}(\H,\cY)^+$ over $[(1,1)] \in \Sh_{\det(K_{\H})}(\T^1,\{\det \mu_{\mathfrak{B}_W}\})$ is the finite abelian extension of $E$ fixed by $$\Art_{E}(r^{-1}(\T^1(\Q) \T^1(\R)\det(K_{\H}))),$$ where $r=r(\H,\cY)\colon \A_{E}\longrightarrow \T^1(\A_\Q)$ is the reciprocity map constructed in \S \ref{reciprocitytori} and \S \ref{galoisconnectedcomp}.

But since, for every $g\in \G(\A_f)$, the induced morphism $\Sh_{K_{\H,g}}(\H,\cY)\to \Sh_{K}(\G,\cX)$ is defined over $E$  \cite[remark 13.8]{milne:shimura} and every cycle $\mathfrak{z}_g$ is then defined over the subfield $E_g$ of $E^{ab}$ that satisfies (the map $r$ factors through its finite part):
$$\Gal(E_g/E)=\T^1(\A_f)/\T^1(\Q)\det(K_{\H,g}).$$
\subsection{Transfer fields and reciprocity law}\label{ringtransferclassfield}
In this section, we will describe a bit further the field of definition of cycles $\cZ_{\G,K}(\H)$, by computing the kernel of the reciprocity map constructed in \S \ref{reciprocitytori} a map
$$r_{\text{fin}} \colon \Gal(E^{ab}/E)\to \T^1(\A_f)/\T^1(\Q)^-.$$
\subsubsection{The kernel of the Verlagerung map}
The inclusion $F\hra E$ induces a commutative diagram 
\begin{equation}\label{verlagerung}\begin{tikzcd}
\Zbf(\A_f)\arrow[hook]{r}{i_{F/E}}\arrow[two heads]{d}{\text{Art}_F}&\T(\A_f)\arrow[two heads]{d}{\text{Art}_E}\\
\Gal(F^{ab}/F)\arrow{r}{\text{Ver}}&\Gal(E^{ab}/E),
\end{tikzcd}\end{equation}
where, $\text{Ver}$ is the group-theoretic transfer map also called the Verlagerung map, see \cite[p. 26]{Neukirch1986}.
\begin{lemma}\label{kernelver}
The kernel of the Verlagerung map is
$$\ker(\text{Ver}\colon\Gal(F^{ab}/F)\to \Gal(E^{ab}/E)) \simeq F^\times /F^{\times,+}\simeq (\Z/2\Z)^{d}.$$
\end{lemma}
\proof 
We have
\begin{align*}
\ker(\text{Art}_E)\cap \text{im}(i_{E/F}) &= (\cO_E^\times)^-E^\times \cap i_{E/F}(\A_{F,f}^\times)\\
&=( i_{E/F}(\cO_F^{\times,+}))^-E^\times \cap i_{E/F}(\A_{F,f}^\times)\\
\overset{(1)}&{=} i_{E/F}\left((\cO_F^{\times,+})^-F^\times \cap \A_{F,f}^\times\right)\\
&= i_{E/F}\left((\cO_F^{\times,+})^-F^\times\right)
\end{align*}
where, we have used $E\cap i_{E/F}(\A_{F,f}^{\times,+}) =i_{E/F}({F}^\times)$ for (1) and then Dirichlet's unit theorem; since ${[\O_E^\times: \O_F^\times]}<\infty$, let $t_1, \cdots, t_?$ be a set of coset representatives for $\O_E^\times/ \O_F^\times$, then
$$E^\times(\O_E^\times)^-=E^\times(\cup_i t_i \O_F^\times)^- =E^\times(\O_F^\times)^-= E^\times(\O_F^{\times,+})^-.$$
By commutativity of the above diagram (and injectivity of $i_{E/F}$) we get
\begin{align*}\ker(\text{Ver})&=\text{Art}_F(i_{F/E}^{-1}(\ker(\text{Art}_E)\cap \text{im}(i_{E/F})))\\
&= \text{Art}_F((\cO_F^{\times,+})^-F^\times).
\end{align*}
Thus, 
$$\ker(\text{Ver}) \simeq  (\cO_F^{\times,+})^-F^\times / \ker(\text{Art}_F)= (\cO_F^{\times,+})^-F^\times/(\cO_F^{\times,+})^-F^{\times,+}.$$
Accordingly, (See footnote~\ref{footnote6} page~\pageref{footnote6})
$$\ker(\text{Ver}) \simeq F^\times/F^{\times,+}.$$
Finally, we get $F^\times/F^{\times,+}\simeq (\Z/2\Z)^{d}$ via the archimedean signature map:
$$\text{sgn}_\infty\colon F^\times \to F_\R^\times/(F_\R^\times)^2 \simeq  \prod_{\iota \in \Sigma_F} \{\pm 1\}, \quad z \mapsto (\iota(z)/|\iota(z_\iota)|)_{\iota \in \Sigma_F},$$
which is surjective with kernel $F^{\times,+}$.\qed

\subsubsection{The kernel of the reciprocity map}
\begin{proposition}\label{kernelreciprocity}
We have a long exact sequence
$$\begin{tikzcd}F^{\times}/F^{\times,+} \arrow[hook]{r}&\Gal(F^{ab}/F) \arrow{r}{\text{Ver}} &\Gal(E^{ab}/E) \arrow[twoheadrightarrow]{r}{r_{\text{\em{fin}}}} & \T^1(\A_f)/\T^1(\Q) .\end{tikzcd}$$
\end{proposition}
\begin{proof}
As we have already computed the kernel of the Verlagerung map in Lemma \ref{kernelver}, it remains to show that the kernel of the map $r_{\text{fin}}$ is the image of the Verlagerung map. Recall that we have an exact sequence (see \ref{exactsequencenu})
\begin{equation}\begin{tikzcd}[column sep= large]1\arrow{r}& \Zbf \arrow{r}&\T\arrow{r}{\nu\colon z \mapsto \frac{z}{\overline{z}}}&\T^1\arrow{r}& 1.\end{tikzcd}\end{equation}
Using Galois cohomology and Hilbert Theorem 90, we deduce from it the following exact diagram
$$\begin{tikzcd}
    1\arrow{r} & \Zbf(\Q)\arrow{r}{}\arrow{d}{} &  \T(\Q)\arrow{r}{\nu} \arrow{d}& \T^1(\Q) \arrow{r}{}\arrow[]{d} & 1 \\
     1\arrow{r} & \Zbf(\A_\Q)\arrow{r}{} &  \T(\A_\Q)\arrow{r}{\nu}& \T^1(\A_\Q) \arrow{r}& 1.
\end{tikzcd}$$
Taking the closure in the adelic points of the first row yields the following commutative diagram\footnote{Recall that $\T^1(\Q)$ is discrete in $\T^1(\A_f)$ by Remark \ref{additionalaxioms}.} and using the exact sequence (\ref{exactsequenceartin}) seen in \S \ref{generalartinmap}, we get the following commutative diagram, where the lower right square is the definition of $r_{\text{fin}}=r(\T^1,\{\det \mu_{\mathfrak{B}_W}\})$ as in \S \ref{reciprocitytori} 
$$\begin{tikzcd}
     &1\arrow{d}{}&  1\arrow{d}{}& 1\arrow{d}{}& .\\
     1\arrow{r} &(\Zbf(\Q))^-\arrow{r}{} \arrow{d}{}&  \T(\Q)^- \arrow{r}{\nu} \arrow{d}{}& \T^1(\Q) \arrow{r}\arrow{d}{}& 1.\\
    1\arrow{r} & \Zbf(\A_f)\arrow{r}{} & \T(\A_f)\arrow{r}{\nu}\arrow{d}{}[swap]{\Art_E}&\T^1(\A_f)\arrow{r}{}\arrow[]{d}{} & 1\\
     & &  \Gal(E^{ab}/E) \arrow{r}{r_{\text{fin}}}\arrow{d} & \T^1(\A_f)/\T^1(\Q) \arrow{r}\arrow{d}& 1\\
        &&  1& 1&
\end{tikzcd}$$
Therefore, 
\begin{align*}\ker(r_{\text{fin}})&=\Art_E(\nu^{-1}(\T^1(\Q)))
\\&=\Art_E\circ i_{E/F}(\Zbf(\A_f) \T(\Q)^-)\\
&=\Art_E(i_{E/F}(\Zbf(\A_f)))\\
&\overset{(\text{\ref{verlagerung}})}{=}\text{Ver} \circ \Art_F(\Zbf(\A_f))\\
&=\text{Ver} (\Gal(F^{ab}/F))\simeq \A_{F,v}^\times E^\times/E^\times \simeq \A_{F,f}^\times /F^\times.\qedhere
\end{align*}
\end{proof}
\subsubsection{Transfer fields of definition}
In the remainder of this section, we focus on the relation between $E(\infty)$ and ring class fields. For every $\O_F$-order $\O$ in $\O_E$, we denote by $\widehat{\O}^\times$ the group of units of its profinite completion. By extending $1\in \cO_E$ to a $\cO_F$-basis we see that such order necessarily of the form $\O=\cO_F + \fc \cO_E$ where $\fc\subset \cO_F$ is a non-zero ideal of $\cO_F$ and $$\widehat{\O}_\fc^\times=({\O}_\fc\otimes \widehat{\Z})^\times=(\widehat{\cO}_F+\fc \widehat{\cO}_{E})^\times\subset \A_{E,f}^\times= (E\otimes \widehat{\Z})^\times.$$
\begin{definition}
We attach to each $\O_F$-order $\O_{\fc}=\cO_F + \fc \cO_E$ two subfields $E(\fc) \subset E[\fc]\subset E^{ab}\nomenclature[F]{$E(\fc) \subset E[\fc]$}{Transfer field and ring class field of conductor $\fc$}$:
\begin{enumerate}[nosep]
\item The {\em ring class field of conductor $\fc$} denoted $E[\fc]$, is the fixed field of $\Art_E(\widehat{\O}_{\fc}^\times E^\times/E^\times)$, i.e.
$$\Gal(E[\fc]/E)\simeq E^\times \backslash \A_{E,f}^\times/\widehat{\O}_{\fc}^\times\simeq \Pic(\O_{\fc}).$$
\item The {\em transfer field of conductor $\fc$} denoted $E(\fc)$, is the field whose norm subgroup is $E^\times\cdot \A_{F,f}^\times \cdot\widehat{\O}_{\fc}^\times$, i.e.
$$\Gal(E({\fc})/E)\simeq E^\times  \A_{F,f}^\times  \backslash \A_{E,f}^\times/\widehat{\O}_{\fc}^\times.$$
\end{enumerate}
We have,
$$\Gal(E[\O_\fc]/E(\O_\fc))\simeq \frac{E^\times \A_{F,f}^\times \widehat{\O}_\fc^\times}{E^\times \widehat{\O}_\fc^\times}\simeq \frac{\A_{F,f}^\times}{\A_{F,f}^\times \cap E^\times \widehat{\O}_\fc^\times}.$$
which is a quotient of  $\Pic(\O_F)$.
\end{definition}
Now, If $F$ has class number one (e.g. $\Q$), then for any $\O_F$-order $\O$ we have $E[\O]=E(\O)$. Moreover, (\ref{transringclassfield}) implies that $E(\infty)=E[\infty]$.
For $\fc, \fc'$ two non-zero ideals of $\cO_F$, we have $\widehat{\cO}_{\fc}^\times\cdot \widehat{\cO}_{\fc'}^\times=\widehat{\cO}_{\text{gcd}(\fc,\fc')}^\times$ and
$$E^\times \A_{F,f}^\times \widehat \cO_{\text{lcm}(\fc,\fc')}^\times \subset E^\times \A_{F,f}^\times \widehat \cO_{\fc}^\times\cap E^\times \A_{F,f}^\times  \widehat\cO_{\fc}^\times,\quad  E^\times \widehat \cO_{\text{lcm}(\fc,\fc')}^\times \subset E^\times  \widehat\cO_{\fc}^\times\cap E^\times \widehat\cO_{\fc}^\times$$
thus,
\begin{equation}\label{formulatransferfield}E[\fc]\cap E[\fc']=E[{\text{gcd}(\fc,\fc')}], \quad E(\fc)\cap E(\fc')=E({\text{gcd}(\fc,\fc')})\end{equation}
and
$$E[\fc] E[\fc']\subset E[{\text{lcm}(\fc,\fc')}], \quad E(\fc) E(\fc')\subset E({\text{lcm}(\fc,\fc')}).$$
Set $E[\infty]=\cup_{\O}E[\O]$ and $E(\infty)=\cup_{\O}E(\O)$, where the union is taken over all $\O_F$-orders of $\O_E$. There exists a descending chain of $\cO_F$-orders $\{\cO_{\fc_i}\}_{i \ge 1}$ such that 
$$\widehat{\cO}_F^\times = \bigcap_i \widehat{\cO}_{\fc_i}^\times, E[\infty]= \bigcup_i E[\widehat{\cO}_{\fc_i}^\times]\text{ and }E(\infty)=\bigcup_i E(\widehat{\cO}_{\fc_i}^\times).$$ Accordingly, $\Art_E(\widehat{\cO}_F^\times)=\cap_i\Art_E(\widehat{\cO}_{\fc_i}^\times)$. From which one deduces that 
$$E[\infty]= (E^{ab})^{\Art_E(\widehat{\cO}_F^\times)} \text{ and } E(\infty)= (E^{ab})^{\Art_E(\A_{F,f}^\times)}.$$
In other words, the transfer field $E(\infty)$ is the subfield of $E^{ab}$ fixed by $\Ver(\Gal(F^{ab}/F))$. The field $E[\infty]$ is a finite Galois extension of $E(\infty)$, with
\begin{equation}\label{transringclassfield}\Gal(E[\infty]/E(\infty))\simeq F^\times \backslash \A_{F,f}^\times/\widehat{\O}_F^\times\simeq  \Pic(\O_F).\end{equation}
Moreover, the extension $E(\infty)/F$ is Galois and its Galois group $\Gal(E(\infty)/F)$ is the Galois dihedral extension:
$$\begin{tikzcd} 1 \arrow{r}& \Gal(E(\infty)/E)\arrow{r} & \Gal(E(\infty)/F) \arrow{r}& \Gal(E/F)\arrow{r}&1 \end{tikzcd}$$
equipped with the canonical splitting given by the complex conjugation
$$\Gal(E(\infty)/F)\simeq \Gal(E(\infty)/E)\rtimes \{1,c\}.$$ 
From now on, we will use the notation $$\Art_E^1\colon\T^1(\A_f)/\T^1(\Q) \xrightarrow{\simeq} \Gal(E(\infty)/E),\nomenclature[F]{$\Art_E^1$}{$=r_{\text{fin}}^{-1}\colon\T^1(\A_f)/\T^1(\Q) \xrightarrow{\simeq} \Gal(E(\infty)/E)$}$$
for the inverse image of $r_{\text{fin}}(\T^1,\{\det \mu_{\mathfrak{B}_W}\})$. Proposition \ref{kernelreciprocity} has the following immediate consequence:
\begin{corollary}\label{Egcontainedtransferfield}For every $g\in\G(\A_f)$, the field of definition $E_g$ of the cycle $\mathfrak{z}_g \in \cZ_{\G,K}(\H)$ is contained in the transfer field $E(\infty)$.
\end{corollary}
\begin{remark}We warn the reader, that there exists two notions of "ring class fields" in the literature. Our use of this terminology follows the one used in \cite{cornut-vatsal,cornut-vatsal:durham}. The second terminology, which is used in \cite{zhang:gross-zagier,Nekovar2007}, is what we called "transfer fields".  
\end{remark}
\subsection{Galois action via \texorpdfstring{$\H(\A_f)$}{HAf}}\label{galoisH}
\begin{proposition}\label{galoisusingH}For every $\sigma \in \Gal(E(\infty)/E)$, let $h_\sigma \in \H(\A_f)$ be any element satisfying $\Art_E^1\left(\det(h_\sigma) \cdot \T^1(\Q)\right)=\sigma|_{E(\infty)}$. For every $g\in \G(\A_f)$, we have
$$\sigma(\mathfrak{z}_g)=\mathfrak{z}_{h_\sigma g}.$$ 
\end{proposition}
\begin{proof}
Fix an element $g\in \G(\A_f)$. In \S \ref{galoisconnectedcomp}, we saw the description of the action of $\Gal(E(\infty)/E)$ on the set of connected components of $ \Sh_{K_{\H,g}}(\H,\cY)(\C)$:
$$\bigsqcup_{h\in \mathcal{C}_{\H,g}} \Gamma_h\backslash\cY, $$ where $ \mathcal{C}_{\H,g} = \H(\Q)\backslash \H(\A_f)/K_{\H,g}$ and $\Gamma_h =  \H(\Q) \cap hK_{\H,g} h^{-1}=\H(\Q) \cap K_{\H,hg}$.

This description says that, for every $\sigma \in \Gal(E(\infty)/E)$, if we let $h_\sigma \in \H(\A_f)$ be any element verifying $\Art_E^1(\det(h_\sigma) \cdot \T^1(\Q))=\sigma|_{E(\infty)}$, then 
$$\left(\Gamma_h\backslash\cY\right)^\sigma= \Gamma_{h_\sigma h}\backslash\cY.$$
On the other hand by Lemma \ref{cycleconcomp}, the cycle $\mathfrak{z}_g$ is the image of $\Gamma_1\backslash \cY=K_{\H,g}\backslash \cY$ by the closed immersion 
$\Sh(\phi) _{K_\H,K_\G} \colon \Sh_{K_{\H,g}}(\H,\cY) \rightarrow\Sh_{K_\G}(\G,\cX)$. Reccall that $\Sh(\phi) _{K_\H,K_\G}$ is defined over $E$  \cite[remark 13.8]{milne:shimura}. Therefore, $\sigma(\mathfrak{z}_g)=\mathfrak{z}_{h_\sigma g}.$
\end{proof}
 Consequently, the left action of $\H(\A_f)$ on the set of $\H$-special cycles
$$\cZ_{\G,K}(\H)=\mathbf{Z}_{\G}(\Q)\H(\Q)\H^{\der}(\A_f)\backslash \G(\A_f)/K_\G,$$
descends to an action of 
$$\H(\Q)\H^{\text{der}}(\A_f)\backslash \H(\A_f) \simeq \T^1(\Q)\backslash\T^1(\A_f)\simeq \Gal(E(\infty)/E),$$
which yields,
$$\Gal(E(\infty)/E)\backslash \cZ_{\G,K}(\H) \simeq \mathbf{Z}_\G(\Q)\H(\A_f)\backslash \G(\A_f)/K.$$
Using the above proposition, we see that for every $g\in \G(\A_f)$, the cycle $\mathfrak{z}_g$ is defined over the subfield $E_g' \subset E(\infty)\nomenclature[F]{$E_g'$}{Field of definition (II) of $\mathfrak{z}_g$}$, given by\footnote{The field $E_g'$ is contained in the field $E_g$ defined at the end of \S \ref{fieldofdefinition1}.}
$$\Gal(E(\infty)/E_g')=\text{Art}_E^1(\det\left((Z_\G(\Q)\H(\Q)\H^{\der}(\A_f) gKg^{-1}) \cap \H(\A_f)\right).$$
\subsection{Hecke action on \texorpdfstring{$\cZ_{\G,K}(\H)$}{ZGK(H)}}\label{Heckoncycles}
Consider the map 
$$\begin{tikzcd}[row sep= small] \mathcal{T}\colon \G(\A_f) \arrow{r}& \mathcal{C}_{\text{fin}}(\Sh_{K}(\G,\cX)(\C)\times_E \Sh_{K}(\G,\cX)(\C))\\
g \arrow[mapsto]{r}{} &\mathcal{T}_g
\end{tikzcd}$$
where, $\mathcal{T}_g$ is the Hecke correspondence defined in \S \ref{ProjsysandHecke}. The map $\mathcal{T}$ factors through the double quotient $K\backslash \G(A_f)/K$. 

For any $g \in \G(A_f)$, we choose a system $(g_i)$ of representatives of $KgK/K$, i.e. $KgK=\sqcup_i g_iK$. 
We get an operator on $\Z[\cZ_{\G,K}(\H)]$ defined as follows 
$$\begin{tikzcd}\mathcal{T}_g\colon \Z[\mathcal{Z}_{\G,K}(\H)]\arrow{r}&\Z[\mathcal{Z}_{\G,K}(\H)],& \mathfrak{z}_{g'}\arrow[r, mapsto] &\sum_i \mathfrak{z}_{g'g_i}.\end{tikzcd}$$
Define $\cH_K=\cH(\G(\A_f)\sslash K)$ to be the global Hecke algebra generated over $\Z$ by $\{KgK\}$ for all ${g\in\G(\A_f)}$, equipped with the classical convolution product. By definition $\Z[\cZ_{\G,K}(\H)]$ is a $\Gal(E(\infty)/E)\times \cH_K$-module given that the actions of $\Gal(E(\infty)/E)$ and $\cH_K$ on $\Z[\cZ_{\G,K}(\H)]$ commute.

\subsection{Compact subgroups and base cycles}\label{compactsfiniteset}
\subsubsection{Further notations}\label{notationdist}
For any finite set $S$ of finite places of our fixed totally real field $F$, set $\O_{F}^{S}\nomenclature[G]{$\O_{F}^{S}$}{Ring of $S$-units}$ for the ring of $S$-units that is the set of $x \in F$ such that $v(x)\ge 0$ for all finite places $v\not \in S$ and put $\cO_{E}^{S}=\cO_E\otimes_{\cO_F}\cO_{F}^{S}\nomenclature[G]{$\cO_{E}^{S}$}{$\cO_E\otimes_{\cO_F}\cO_{F}^{S}$}$. We will also use 
$$F_S:=\prod_{v \in S} F_v, \quad\text{and}\quad\A_{F,f}^S:= \prod_{v \not\in S}'F_v$$
the restricted product of the additive groups $F_v^\times$ for all finite places $v\not \in S$, with respect to the local integers $\cO_{F_v}^\times$. One can write the finite adeles of $F$ as the product $\A_{F,f}=F_S\times \A_{F,f}^S$. For any place $v$ of $F$, let $F_v$ be the completion of $F$ at $v$ and let $\O_{F_v}$ be its ring of integers with uniformizer $\varpi_{v}\nomenclature[G]{$\varpi_v$}{Uniformizer for $\cO_{F_v}$}$ and maximal ideal $\p_v=\varpi_{v} \O_{F_v}\nomenclature[G]{$\p_v$}{Maximal ideal $\O_{F_v}$}$. For any $F$-algebra $R$, let $ R_{v}=R\otimes_F F_{v}$.

Recall that we have fixed at the beginning of \S \ref{SectionShimura}, for each finite place $v$ of $F$, an embedding $\begin{tikzcd}\iota_v\colon \overline{F} \arrow[r, hook] &  \overline{F}_v\end{tikzcd}$ and we set $w_v\nomenclature[G]{$w_v$}{The place of $E$ defined by $\iota_v\colon \overline{F}\hra\overline{F}_v$}$ for the unique place of $E$ defined by $\iota_v$, If $v$ splits in $E$, by abuse of notation denote the other place by $\overline{w}_v$. When the place $v$ is understood from the context, we will omit the subscript $v$ and simply write $w$ and $\overline{w}$.
\subsubsection{Integral models}
Recall that we identify $\U(W)$ with the subgroup of $\U(V)$ given by
$$ \{g \in \U(V)(R)\subset \GL(V\otimes_F R)\colon g \cdot x=x, \quad\forall x \in  D \otimes_F R\},$$
for any $F$-algebras $R$. We have a faithful representations of $\U(V)$ and $\U(W)$
$$\begin{tikzcd} \U(W) \arrow[hook]{r}{\iota}& \U(V) \arrow[r, hook]& \GL(V_F)\end{tikzcd}$$
where $V_F$ is the underlying $F$-vector space of the hermitian $E$-space $(V,\psi)$ (\S \ref{groups}). We identify $\U_V$ and $\U_W$ with closed subgroups of $\GL(V_F)$. Let $\underline{\U}_V$ and $\underline{\U}_W$ be the schematic closures of $\U(V)$ and $\U(W)$ in $\mathcal{G}=\GL(V_F)_{\cO_F}$ \footnote{Here, we are picking up implicitly a lattice in $V_F$ to obtain the $\cO_F$ structure.}. 

For $\star\in \{V,W\}$, the schematic closure $\underline{\U}_\star$ is a model for ${\U}(\star)$ over $\cO_{F}$ \cite[Lemma 2.4.1]{GetHah2019}. But by \cite[Lemma 2.4.2]{GetHah2019}, there is a finite set ${S_{\star}^0}$ of finite places of $F$ such that $\underline{\U}_{\star, \cO_{F}^{{S_{\star}^0}}}$ becomes smooth over $\cO_{F}^{{S_{\star}^0}}$. Subsequently, \cite[Proposition 3.1.9]{Conrad2014} ensures that for a large enough finite set ${S_{\star}^1}$ of finite places of $F$ containing $S_{\star}^0$, all the fibers of $\underline{\U}_{\star, \cO_{F}^{{S_{\star}^1}}}$ will be connected and reductive, i.e. $\underline{\U}_{\star, \cO_{F}^{{S_{\star}^1}}}$ is a reductive $\cO_{F}^{{S_{\star}^1}}$-model of $\U(\star)$. Accordingly, the homomorphism $\iota$ extends to the models over $\cO_{F,S^1}$ for $S^1=S_V^1\cup S_W^1$ and we get
$$\begin{tikzcd}\underline{\U}_{W, \cO_{F}^{{S}^1}} \arrow[hook]{r}{\iota}&\underline{\U}_{V, \cO_{F}^{{S}^1}} \arrow[r, hook]& \mathcal{G}_{\cO_{F,S^1}}.\end{tikzcd}$$
Using \cite[Lemma 2.4.1]{GetHah2019}, we get for any finite place $v \not \in S^1$ of $F$ a hyperspecial maximal compact subgroup
$$\underline{\U}_V(\cO_{F_v})={\U}(V)(F_v) \cap \mathcal{G}({\cO_{F_v}}) \text{ and,}$$\begin{equation}\label{intersectioncompacts}\underline{\U}_W(\cO_{F_v})={\U}(W)(F_v) \cap \mathcal{G}({\cO_{F_v}})={\U}(W)(F_v) \cap \underline{\U}_V(\cO_{F_v}).\nomenclature[G]{$\underline{\U}_\star$}{Integral models for $\star\in \{V,W\}$}\end{equation}
\begin{remark}\label{latticesandmodel}
{In the previous discussion we chose on purpose to use general arguments. 
Nevertheless, a more specific argument goes as follows: We pick an $\cO_E$ lattice $\mathcal L_W$ in $W$, $\mathcal L_D$ in $D$ and set $\mathcal L_V=\mathcal L_W \oplus \mathcal L_D$ in $V$. We choose them such that they are contained in their duals, with the quotients as small as possible. Now, we take $S^1$ to be the bad places, those occurring in the quotient $\mathcal L_V^\vee /\mathcal L_V$. Away from $S^1$, the restriction of $\psi$ on $\mathcal L_V$, $\mathcal L_W$ is a perfect integral hermitian pairing, giving rise to the desired smooth reductive models which are unitary groups.}
\end{remark}
\subsubsection{Compact subgroups}\label{compactsubgroups122}
Let $K_\star\nomenclature[G]{$K_\star$}{Fixed open compact subroups, $\star \in \{V,W\}$}$, with $\star \in \{V,W\}$, be any open compact subgroup of $\U_\star(\A_{F,f})$. It inersects then $\underline{\U}_\star(\prod_{v \not\in S^1}\cO_{F_v})=\prod_{v \not\in S^1}\underline{\U}_\star(\cO_{F_v})$ in an open compact subgroup. Therefore, there exists a finite set $S^2=S^2(K_V,K_W)$\footnote{We want to insist here that the enlarged set $S$ depends on both compacts $K_V$ and $K_W$.} containing $S^1$, for which both compact subgroups can be written as
$$K_\star=K_{\star,S^2}\times \prod_{v \not\in S^2}\underline{\U}_\star(\cO_{F_v}), \forall \star \in \{V,W\} $$
where $K_{\star,S^2}$ is some open compact subgroup of $\U(\star)(F_{S^2})=\prod_{v \in S^2} \underline{\U}_\star({F_v})$. In particular, if we set $K_{\star,v}:=\underline{\U}_\star(\cO_{F_v})$ for all finite places $v$ of $F$ away from $S^2$, we have by (\ref{intersectioncompacts})
$$K_{W,v}=\U(W)(F_v) \cap K_{V,v}.$$
\subsubsection{A base cycle}\label{basecycle}
Now for each $\star \in \{V,W\}$, let $g_{\star}=(g_{\star,v})$ be any fixed element of $\U_\star(\A_f)$. By definition of the adelic points of $\U_\star$, there exists a finite set of places $S=S(K_V,K_W, g_{V},g_{W})$ containing $S^2\cup \text{Ram}(E/F)$ and such that $g_{\star,v} \in K_{\star,v}$ for all $v \not\in S$ and $g_{\star,S}  \in \U_\star(F_S)$. But as far as the formation of the cycle $\mathfrak{z}_{g}$ for $g= (g_V,g_W)$ is concerned, nothing is lost by considering $g_0=(g_{0,V},g_{0,W}) \in \G(\A_f)$, such that $g_{0,\star}^S= 1 \in \U_\star(\A_{F,f}^S)$ and $g_{0, \star,S}= g_{\star,S} \in \U_\star(F_S)$, since $\mathfrak{z}_{g}=\mathfrak{z}_{g_0}$.

Now that we have fixed our set of finite places $S$, let us consider $$\underline\T^1 := \ker \Norm \colon\underbrace{\Res_{\cO_{E}^{S}/\cO_{F}^{S}}\mathds{G}_{\cO_{E}^{S}}}_{:=\underline{\T}}\to \underbrace{\mathds{G}_{\cO_{F}^{S}}}_{:=\underline{\Zbf}}.\nomenclature[G]{$\underline{\Zbf}, \underline{\T}, \underline\T^1$}{${\mathds{G}_{m,\cO_{F}^{S}}}$, ${\Res_{\cO_{E}^{S}/\cO_{F}^{S}}\mathds{G}_{m,\cO_{E}^{S}}}$ and $\ker (\Norm \colon \underline{\T} \to\underline{\Zbf}) $}$$
Each of the above reductive groups over $\cO_{F}^{S}$ are models for the obvious corresponding groups in $\T_F^1, \T_F$ and $\Zbf_F$, so in particular
$$ \T^1(\A_f)=\underline{\T}^1(\A_{F,f}), \T(\A_f)=\underline{\T}(\A_{F,f}) \text{ and } \Zbf(\A_f)=\underline{\Zbf}(\A_{F,f}).$$
 We view again $\underline\T^1$ as the center of $ \underline{\U}_W$ and also $ \underline{\U}_V$. Let $\underline{\U}_W^{\der}$ denote the kernel of the determinant map $\det \colon \underline{\U}_W \to \underline\T^1$. Write $\underline\nu\colon \underline\T\twoheadrightarrow \underline\T^1$ for the homomorphism given on $\cO_{F}^{S}$-points by $z\mapsto z/\overline{z}$. 

Define $\underline{\G}:=\underline{\U}_V\times\underline{\U}_W\nomenclature[G]{$\underline{\G},\underline{\H}$}{$\O_F^S$-models for $\G$ and $\H$}$ and $\underline{\H}:=\Delta(\underline{\U}_W) \subset \underline{\G}$\footnote{The diagonal homomorphism $\Delta$ extends to the models over $\cO_{F}^{S}$ since $\iota$ does.}. We then have
$$\G(\A_f)=\underline{\G}(\A_{F,f}), \quad \text{and}\quad\H(\A_f)=\underline{\H}(\A_{F,f}).$$
Set $K_v:=K_{V,v}\times K_{W,v}\nomenclature[G]{$K_v,K,K_{\H}$}{The fixed level structures}$ for any place $v \not \in S$, $K_S=K_{V,S}\times  K_{W,S}$, $K^S=K_V^S \times K_W^S$. $K=K_V \times K_W$ and $K_\H=\Delta(K_W)= K_{\H,S} \times K_\H^S$.
\begin{remark}\label{determinantsurjective}
{For each place $v \not \in S$, we have an exact sequence of $\cO_{F_v}$-groups
$$\begin{tikzcd} 1 \arrow{r} & \underline{\U}_\star^{\der} \arrow{r}& \underline{\U}_\star \arrow{r}{\det} & \underline{\T}^1\arrow{r}&1 \end{tikzcd}.$$
It induces the long exact sequence in \'{e}tale cohomology 
$$\begin{tikzcd} 1 \arrow{r} & \underline{\U}_\star^{\der}(\cO_{F_v}) \arrow{r}& \underline{\U}_\star(\cO_{F_v}) \arrow{r}{\det} & \underline{\T}^1(\cO_{F_v})\arrow{r}&\text{H}_{\text{\'{e}t}}^1(\cO_{F_v},\underline{\U}_\star^{\der}). \end{tikzcd}$$
On the one hand, we have $\text{H}_{\text{\'{e}t}}^1(\cO_{F_v},\underline{\U}_\star^{\der})=\text{H}^1(\mathds{F}_{q_v},\underline{\U}_\star^{\der})$ \cite[4.5 \S III]{milne:etale}. On the other hand, Lang's theorem implies $\text{H}^1(\mathds{F}_{q_v},\underline{\U}_\star^{\der})=0$ \cite{lang:finitefields}. Therefore,
\begin{align*}\det(\underline{\U}_\star(\cO_{F_v})) &= \underline{\T}^1(\cO_{F_v}).\qedhere\end{align*}
}
\end{remark}
\subsection{The field $\cK$}\label{fieldkappa}
The stabilizer of $\mathfrak z_{g_0}$ in $\H(\A_f)$ is 
$$\text{Stab}_{\H(\A_f)}\mathfrak z_{g_0}=\left(Z_\G(\Q)\cdot\H(\Q)\cdot g_0Kg_0^{-1}\right)\cap \H(\A_f).$$
\begin{lemma}\label{fielddefinitiong0}
We may rewrite the stabilizer as
$$\text{Stab}_{\H(\A_f)}\mathfrak z_{g_0}=\H(\Q)\cdot\left({K_{\H,g_0,S}^Z}\times K_\H^S\right),$$
where $K_{\H,g_0,S}^Z:=\left((Z_\G(\Q)\cap K^S) \cdot g_{0,S}K_Sg_{0,S}^{-1}\right) \cap {\H}(F_S)$\footnote{Note that by SV5 of Remark \ref{additionalaxioms} the subgroup $Z_\G(\Q)$ is discrete in $Z_\G(\A_f)$ and consequently the intersection $Z_\G(\Q)\cap K^S$ must be finite.}.
\end{lemma}
\proof
We have
\begin{align*}\text{Stab}_{\H(\A_f)}\mathfrak z_{g_0}\overset{(0)}&{=}\H(\Q)\cdot\left((Z_\G(\Q)\cap Z_\H(\Q)K^S)\cdot g_0Kg_0^{-1}\cap \H(\A_f)\right)\\
\overset{(1)}&{=}\H(\Q)\cdot\left((Z_\G(\Q)\cap K^S)\cdot g_0Kg_0^{-1}\cap \H(\A_f)\right)\\
\overset{(2)}&{=}\H(\Q)\cdot\left({\left((Z_\G(\Q)\cap K^S) \cdot g_{0,S}K_Sg_{0,S}^{-1}\right) \cap {\H}(F_S)}\times K_\H^S\right)\\
&=\H(\Q)\cdot\left({K_{\H,g_0,S}^Z}\times K_\H^S\right),
\end{align*}
\begin{enumerate}[nosep]
\item[where, (0)] is a straightforward consequence of the fact that $g_{0}^S=1 \in \underline{\G}(\A_{F,f}^S)$ and $K_\H^S = K^S \cap \underline{\H}(\A_{F,f}^S)$.
\item[(1)] Let $z_G=z_H k \in  Z_\G(\Q) \cap Z_\H(\Q) K^S=Z_\G(\Q) \cap Z_\H(\Q) \underline{\G}(\cO_F^S)$. Write $z_G=(z_V,z_W)$ and $z_H k= \Delta(z_{W}')(k_V,k_W)= k z_H$, hence\footnote{Recall that $w_{n+1}$ is the fixed generator of the global $E$-hermitian line $D$ which is orthogonal to $W$.}
$$z_{V} w_{n+1}= k_V z_{W}' w_{n+1} =k_V w_{n+1},$$
so $z_V \in \underline{\T}^1(\cO_F^S) = Z_{\underline{\U}_V}(\cO_F^S)$. Accordingly $z_W' \in \underline{\T}^1(\cO_F^S) = Z_{\underline{\U}_W}(\cO_F^S)$ and $z_G \in K^S$, i.e. $Z_\G(\Q) \cap Z_\H(\Q) K^S=Z_\G(\Q) \cap K^S$.
\item[(2)] Let $(z g_{0,S}k_Sg_{0,S}^{-1}, k^S) \in K_{\H,g_0,S}^Z \times K_\H^S \subset \underline{\H}(\A_{F,f})$, for some $z \in Z_\G(\Q)\cap K^S$. Hence
\begin{align*}(z g_{0,S}k_Sg_{0,S}^{-1}, k^S)&=z( g_{0,S}k_Sg_{0,S}^{-1}, \underbrace{z^{-1}k^S}_{:=k^{\prime S}\in K^S})\\
&=z( g_{0,S}k_Sg_{0,S}^{-1}, k^{\prime S}) \\&= z g_0 (k_S,k^{\prime S})g_0^{-1} \in (Z_\G(\Q)\cap K^S)  g_0Kg_0^{-1}\cap \H(\A_f),
\end{align*}
and so $((Z_\G(\Q)\cap K^S)\cdot g_0Kg_0^{-1})\cap \H(\A_f)=K_{\H,g_0,S}^Z\times K_\H^S$.\qed
\end{enumerate}
Set $\mathcal{K}:=E_{g_0}'$, the field of definition of the cycle $\mathfrak{z}_{g_0}$, it is the subfield of $E_{g_0}\subset E(\infty)$\footnote{Recall that $K_{\H,g_0}=g_0 K g_0^{-1} \cap \H(\A_f)$ and $E_{g_0}\subset E(\infty)$ is the subfield fixed by $\Art_E^1\left(\T^1(\Q)\det(K_{\H,g_0})\right)$.} fixed by (see \S \ref{galoisH})
$$\Art_E^1\left(\det\left(\H(\Q)\cdot\left(K_{\H,g_0,S}^Z \times K_\H^S\right)\right)\right)=\Art_E^1\left(\T^1(\Q)\det\left(K_{\H,g_0,S}^Z \times K_\H^S\right)\right),$$
where, $\Art_E^1$ is the map defined above Corollary \ref{Egcontainedtransferfield}. Therefore,
$$\begin{tikzcd}[column sep=large]
\dfrac{\A_{E,f}^\times}{ E^\times \A_{F,f}^\times (\cO_{g_0,S}^\times \times \widehat{\cO_E^S}^\times)} \arrow{r}{\underline{\nu}}[swap]{\simeq}&\dfrac{\T^1(\A_f)}{\T^1(\Q) (U_{g_0,S} \times U^S)} \arrow{r}{\Art_E^1}[swap]{\simeq}&
\Gal(\mathcal{K}/E)
\end{tikzcd}$$
where, $U_{g_0,S}:= \det K_{\H,g_0,S}^Z$, $U^S:=\det K_\H^S$ and $\cO_{g_0,S}^\times \subset  (E\otimes F_S)^\times$ such that  $\underline{\nu}( \cO_{g_0,S}^\times)= U_{g_0,S}$ and finally $\widehat{\cO_E^S}= \cO_E \otimes_{\cO_F^S} \widehat{\cO_F^S} =\prod_{v\not\in S} \cO_{E_v}^\times \subset \A_{E,f}^S$ (see Remark \ref{determinantsurjective}), where we have denoted, slightly abusively, $\O_{E_v}=\cO_E\otimes_{\cO_F} \cO_{F_v}$ the maximal order of the quadratic \'{e}tale algebra $E_v=E\otimes_F F_v$.
\subsection{Ramification in transfer fields extensions}
Let $v$ be a finite place of $F$ that is unramified in the extension $E/F$, hence
$$E_v^\times/F_v^\times\cO_{E_v}^\times\simeq \begin{cases}
\varpi_{v}^\Z &\text{If $v$ split in $E$,}\\
\{1\}& \text{If $v$ is inert in $E$.}
\end{cases}$$
\begin{lemma}\label{ramificationtransferfields}
Let $\fc\subset \cO_F$ be any non-zero ideal of $\cO_F$. 
\label{1ramificationtransferfields} Any prime of $E$ not dividing $\fc\cO_E$ is unramified in $E[\fc]/E$, hence also in $E(\fc)/E$.
\end{lemma}
\begin{proof}
Let $v$ be any place of $E$ and $E_v$ the completion of $E$ at the place $v$. The extension $E[\fc]/E$ being Galois, then the various completions $E[\fc]_w$ with $w$ a place of $E[\fc]$ extending $v$ are isomorphic, let $E[\fc]_w$ be (any) one of these completions. Let $\mathfrak{q}$ be the prime of $E[\fc]$ corresponding to the place $w$ and identify the local Galois group $\Gal(E[\fc]_w/E_v)$ with the decomposition group of $D_{\mathfrak{q}}(E[\fc]/E) \subset \Gal(E[\fc]/E)$ (the stabilizer of $\mathfrak{\mathfrak{q}}$ in $\Gal(E[\fc]/E)$).
The decomposition group $D_{\mathfrak{q}}(E[\fc]/E)$ is the image of the composition of the following two maps
$$\begin{tikzcd}E_{v}^\times \arrow[r,hook] & \A_{E,f}^\times \arrow[r, two heads]&\A_{E,f}^\times/E^\times \widehat{\cO}_{\fc}^\times \arrow{r}{\simeq}&\Gal(E[\fc]/E),\end{tikzcd}$$
and the image of $\cO_{E,v}^\times$ in $\Gal(E[\fc]/E)$ is precisely the inertia group $I(E[\fc]/E)$. Now, because $v\nmid \fc\cO_E$ we see that the image of $\cO_{E,v}^\times$ is trivial, i.e. $v$ is unramified in the extension $E[\fc]/E$. 
\end{proof}

\subsection{Ramification in $\cK$-transfer fields extensions}\label{kappatransferfields}
For any nonzero ideal $\ff \subset \O_F$ that is prime to $S$, we consider the field $ \cK(\ff)\subset E(\infty)\nomenclature[G]{$\cK(\ff)$}{$\cK$-transfer field of conductor $\ff$}$, fixed by 
$\Art^1_E(U_\ff)$, where $U_\ff:=U_{g_0,S}  \times U_\ff^S\nomenclature[G]{$U_\ff$}{$=U_{g_0,S} \times U_\ff^S$}$ with $U_\ff^S=\prod_{v\not\in S}U_v^1({\text{ord}}_{F_v}(\ff))$ and
$$U_v^1(c):=\underline{\nu}\big((\underbrace{\O_{F_v}+\varpi_{v}^c\O_{E_v}}_{:=\O_{v,c}})^\times\big), \quad v \not\in S, \,c \in \N.$$
The fields $\cK(\ff)$ will be called the $\cK$-transfer field of conductors $\ff$. Set $$\mathfrak{O}_\ff:= \cO_{g_0,S}  \times (\cO_\ff \otimes_{\cO_F^S} \widehat{\cO_F^S} )=\cO_{g_0,S}  \times  \prod_{v \not \in S} \cO_{v, {\text{ord}}_{F_v}(\ff)}\subset \widehat{\cO_E},\nomenclature[G]{$\mathfrak{O}_\ff$}{$\subset\widehat{\cO_E}$}$$
in particular, we have $\underline{\nu}(\mathfrak{O}_\ff^\times)=U_\ff$. Note that $\cK(1)=\cK$. Moreover, for every two $\cO_F$-ideals $\mathfrak{n}, \ff \subset \O_F$ prime to $S$, we have $\cK(\ff\mathfrak{n})\supset \cK(\ff)$ and isomorphisms:
$$\begin{tikzcd}[column sep=large]
\frac{E^\times \A_{F,f}^\times \mathfrak{O}_\ff^\times}{E^\times \A_{F,f}^\times \mathfrak{O}_{\ff\mathfrak{n}}^\times}\arrow{r}{\underline{\nu}}[swap]{\simeq}&\frac{ \T^1(\Q) U_\ff}{\T^1(\Q)U_{\ff\mathfrak{n}}} \arrow{r}{\Art_E^1}[swap]{\simeq}&
\Gal(\cK(\ff\mathfrak{n})/\cK(\ff)).
\end{tikzcd}$$
We then obtain an exact sequence\footnote{\label{footnote6}This is a consequence of the following elementary fact: if one has three subgroups $A$ and $C \subset B$ of some abelian group, then the inclusion maps yield an exact sequence $$1\longrightarrow \frac{A\cap B}{A\cap C} \longrightarrow \frac{B}{C} \longrightarrow \frac{A B}{A C}\longrightarrow 1.$$}
\begin{equation}\label{galoistransfersequence}\begin{tikzcd}1 \arrow{r}&\frac{E^\times \cap \A_{F,f}^\times \mathfrak{O}_{\ff}^\times }{E^\times \cap \A_{F,f}^\times \mathfrak{O}_{\ff\mathfrak{n}}^\times }\arrow{r}&
\frac{\mathfrak{O}_{\ff}^\times}{\mathfrak{O}_{\ff\mathfrak{n}}^\times}\arrow{r}{}[swap]{}&
\Gal(\cK(\ff\mathfrak{n})/\cK(\ff))\arrow{r}& 1
\end{tikzcd}\end{equation}
with
\begin{equation*} \frac{\mathfrak{O}_{\ff}^\times}{\mathfrak{O}_{\ff\mathfrak{n}}^\times}\simeq \frac{(\mathfrak{O}_{\ff}^S)^\times }{(\mathfrak{O}_{\ff\mathfrak{n}}^S)^\times}\simeq \prod_{v\colon \p_{v} \mid {\mathfrak{n}}}\frac{\O_{v,{\text{ord}}_{F_v}(\ff)}^\times}{\O_{v,{\text{ord}}_{F_v}(\ff \mathfrak{n})}^\times}.\end{equation*}
Moreover, one can explicitly describe the left global error term appearing in the above exact sequence:
\begin{lemma}
The natural inclusion map yields an inclusion
$$\begin{tikzcd}\dfrac{\cO_E^\times \cap  \mathfrak{O}_{\ff}^\times }{\cO_E^\times \cap  \mathfrak{O}_{\ff\mathfrak{n}}^\times }\arrow[r, hook] &\dfrac{E^\times \cap \A_{F,f}^\times \mathfrak{O}_{\ff}^\times }{E^\times \cap \A_{F,f}^\times \mathfrak{O}_{\ff\mathfrak{n}}^\times }\end{tikzcd},$$
with finite cokernel of size smaller than $2$.
\end{lemma}
\proof This is \cite[Proposition (2.9)]{Nekovar2007}. \qed

For later use, we set $d_\ff:=\#\left(\frac{E^\times \cap \A_{F,f}^\times \cO_{1}^\times }{E^\times \cap \A_{F,f}^\times \mathfrak{O}_{\ff}^\times }\right)$ for any $\cO_F$-ideal $\ff$.
\begin{remark}\label{containementfieldofdefinitions}
Observe that by (\ref{formulatransferfield}), for any ideal $\ff\subset \cO_F$ prime to $S$, there exists a smallest non-zero ideal $\fc_\ff\subset \cO_F\nomenclature[G]{$\fc_\ff$}{Non-zero $\cO_F$-ideal such that $E(\fc_\ff)$ is the smallest transfer field containing $\cK(\ff)$}$ with respect to divisibility such that $$\widehat{\cO}_{\fc_\ff} \subset\mathfrak{O}_\ff \subset \widehat{\cO}_\ff\subset \widehat{\cO}_E$$ Equivalently, $E(\fc_\ff)$ is the smallest transfer field containing $\cK(\ff)$: 
\begin{align*} E(\ff) &\subset \cK(\ff)\subset E(\fc_\ff).\qedhere \end{align*}
\end{remark}
\begin{lemma}\label{transferkappa}
For any ideal $\ff\subset \cO_F$ prime to $S$, the field $\cK(\ff)$ is contained in the transfer field $E(\fc_1 \ff)$.
\end{lemma}
\proof We have $\mathfrak{O}_1 \cap \widehat{\cO_\ff}=\mathfrak{O}_\ff$, but $\mathfrak{O}_1 \supset \widehat{\cO}_{\fc_1}$ (Remark \ref{containementfieldofdefinitions}). It follows that
$$\mathfrak{O}_\ff\supset \widehat{\cO}_{\fc_1} \cap \widehat{\cO_\ff} \supset \widehat{\cO}_{\fc_1 \ff},$$
hence $E(\fc_1 \ff) \supset \cK(\ff)$.\qed

Using Lemmas \ref{ramificationtransferfields} and \ref{transferkappa} we get the following immediate consequence:
\begin{corollary}\label{unramifiedprimesintransfer}
Let $\ff\subset \cO_F$ be a non-zero ideal of $\cO_F$  prime to $S$. If $\p$ is a prime ideal of $\cO_F$ not dividing $\fc_1 \ff$ then each prime of $E$ above $\p$ is unramified in $\cK(\ff)$.
\end{corollary}
\begin{lemma}\label{I0nekovar}
Let $\ff$ be any non-zero ideal of $\cO_F$ satisfying
$$\ff\cO_E \nmid I_0 := \text{lcm}\{(u-1)\colon u \in (\cO_E^\times)_{\text{tors}}, u\neq 1\}.\nomenclature[G]{$I_0$}{$\text{lcm}\{(u-1)\colon u \in (\cO_E^\times)_{\text{tors}}, u\neq 1\}$}$$
For any ideal $\mathfrak{n}\subset \cO_F$, we have
$${E^\times \cap \A_{F,f}^\times \mathfrak{O}_{ \mathfrak{n} \ff}^\times }\simeq F^\times \quad {\cO_E^\times \cap \mathfrak{O}_{\mathfrak{n} \ff}^\times }\simeq \cO_F^\times.$$
\end{lemma}
\proof This is \cite[Proposition 2.10]{Nekovar2007}.\qed

Applying Lemma \ref{I0nekovar} to (\ref{galoistransfersequence}) yields an isomorphism of groups:
\begin{corollary}\label{galoisgroupsI0}
Let $\ff$ be any non-zero ideal of $\cO_F$ prime to $S$ and $\ff\cO_E \nmid I_0$. For any $\cO_F$-ideal $\mathfrak{n}\subset\cO_F$, we have:
$$ \Gal(\cK(\mathfrak{n} \ff)/\cK(\ff)) \simeq \dfrac{\mathfrak{O}_{ \ff}^\times}{\mathfrak{O}_{\mathfrak{n} \ff}^\times}\simeq \prod_{v\in \Spec(\cO_F)\colon \p_{v} \mid {\mathfrak{n}}}\frac{\O_{v,{\text{ord}}_{F_v}(\ff)}^\times}{\O_{v,{\text{ord}}_{F_v}(\ff\mathfrak{n})}^\times}.$$
\end{corollary}
\subsection{Interlude on orders}\label{interludeorders}
Define the Artin symbol, $$\left(\frac{\mathfrak{d}_E}{\p_v}\right)=\begin{cases}-1 &\text{ if $\p_v$ remains inert in $E$,}\\
1 &\text{ if $\p_v$ splits in $E$,} \end{cases}$$
where $\mathfrak{d}_E$ is the different ideal of $E$. For $\p_{v}\in \Spec(\cO_F)$, with $v\not\in S$ (in particular unramified in $E/F$), 
set $\mathds{F}^k(v):=\O_{F_v}/\p_{v}^k\nomenclature[G]{$\F^k(v)$}{$k^{\text{th}}$-neighbourhoud at $v$, in $v$}$. When $k=1$ this is the residue field of $\O_{F_v}$ whose size is $q_v:=\#\mathds{F}^1(v)$. For any integer $k \ge 1$, set $\F^k(v)[\upepsilon]:= \F^k(v)[X]/\langle X^2\rangle$ to be the infinitesimal deformation $\F^k(v)$-algebra and set $\mathds{E}^k(v):=\O_{E_v}/\p_v^k\O_{E_v}\nomenclature[G]{$\mathds{E}^k(v)$}{}$. We then have ring isomorphisms
$$\mathds{E}^1(v)\simeq\begin{cases}\text{a quadratic extension of } \F^1(v) &\text{ if }\left(\frac{\mathfrak{d}_E}{\p_v}\right)=-1,\\
 \F^1(v)\oplus  \F^1(v)&\text{ if }\left(\frac{\mathfrak{d}_E}{\p_v}\right)=1.\end{cases}$$
Here, we summarize a few facts on $\cO_{F_v}$-orders of $E_v$. 
The map that sends an ideal $\p_v^{c_v} \subset \cO_{F_v}$ to the order $\cO_{v,c_v}= \cO_{F_v} + \p_v^{c_v}  \cO_{E_v}$ induces a bijection between the set of ideals of $\cO_{F_v}$ and the set of $\cO_{F_v}$-orders in $E_v$. These orders are all Gorenstein and local whenever $c_v \ge 1$ with maximal ideal $\mathfrak{P}_{v,c_v}:=\p_v \cO_{v,c_v-1}$ and 
$$\cO_{v,c_v}/\mathfrak{P}_{v,c_v} \simeq \F^1(v).$$
When $c_v=0$, the order $\cO_{v,0}$ is only semi-local if $v$ splits in $E$, since $\cO_{E_v} \simeq \cO_{F_v}\oplus  \cO_{F_v}$.

Let $\Tr \colon \O_{E_v} \to \cO_{F_v}$ be the usual trace map $z \mapsto z + \overline{z}$. Let $\alpha_v\in \cO_{E_v}^\times$ be any generator of the rank $1$ $\cO_{F_v}$-module $\ker \Tr$\footnote{For example, if $v$ split then one has a decomposition $\cO_{E_v}\simeq \cO_{F_v}\oplus \cO_{F_v}$, such that $\tau \in \Gal(E/F)$ (or just the complex conjugation since we have identified $E$ with $\iota_1(E)$) acts by swapping the components. Therefore, $\ker \Tr= \cO_{F_v}\cdot (1,-1) $ and one can take $\alpha_v =(1,-1)$, in addition, observe that $\overline{(1,-1)}=-(1,1)$ and $(1,-1)^2=(1,1) \in \cO_{F_v}$ where $ \cO_{F_v}$ is embedded diagonally in $\cO_{E_v}$.}, therefore $\overline{\alpha}_v=-\alpha_v, \alpha_v^2 \in \cO_{F_v}$ and for every $c_v\ge 0$ we have $\cO_{v,c_v}= \cO_{F_v}\oplus  \p_v^{c_v}\alpha_v$.
\begin{lemma}
Let $c>0$, we have an isomorphism of groups
$$\dfrac{\cO_{v,0}^\times}{\cO_{v,c}^\times}\simeq \mathds{E}^c(v)^\times/\F^c(v)^\times.$$
\end{lemma}
\proof
Consider the following composition of reduction maps
$$\begin{tikzcd}
\cO_{E_v} =\cO_{v,0} \arrow[r,two heads]& \mathds{E}^k(v)^\times\arrow[r,two heads]&  \mathds{E}^k(v)^\times/\F^k(v)^\times.
\end{tikzcd}$$
The quotient map of rings $\cO_{v,0} \twoheadrightarrow \cO_{v,0}/ \p_v^c \cO_{v,0}$ induces a surjective homomorphism of groups $\cO_{v,0}^\times \twoheadrightarrow\left(\cO_{v,0}/ \p_v^c \cO_{v,0}\right)^\times$ with kernel $1+ \p_v^c \cO_{v,0}$, i.e. 
$$\cO_{v,0}^\times/1+ \p_v^c \cO_{v,0} \simeq \left(\cO_{v,0}/ \p_v^c \cO_{v,0}\right)^\times.$$
But since the diagonal image of $\cO_F \to \cO_{v,0}/ \p_v^c \cO_{v,0}$ is precisely $\F^k(v)$, we deduce that the kernel of the composition following of reduction maps
$$\begin{tikzcd}
f\colon\cO_{E_v} =\cO_{v,0} \arrow[r,two heads]& \mathds{E}^k(v)^\times\arrow[r,two heads]&  \mathds{E}^k(v)^\times/\F^k(v)^\times.
\end{tikzcd}$$ 
is precisely $\cO_F^\times (1+ \p_v^c \cO_{v,0}) = \cO_{v,c}^\times.$ \qed
{\begin{lemma}\label{sizeordersvertical}Let $c \ge k >0$, we have an isomorphism
$$\cO_{v,c}^\times/\cO_{v,c+k}^\times\simeq \left(\F^k(v)[\upepsilon]\right)^\times\big/ \F^k(v)^\times.$$
\end{lemma}
\begin{proof}
 As in the previous proof, consider the quotient map of rings $\cO_{v,c} \twoheadrightarrow \cO_{v,c}/ \p_v^k \cO_{v,c}$. It induces a surjective homomorphism of groups $\cO_{v,c}^\times \twoheadrightarrow\left(\cO_{v,c}/ \p_v^k \cO_{v,c}\right)^\times$ with kernel $1+ \p_v^k \cO_{v,c}$, i.e. 
$$\cO_{v,c}^\times/1+ \p_v^k \cO_{v,c} \simeq \left(\cO_{v,c}/ \p_v^k \cO_{v,c}\right)^\times.$$
We may also consider the map of rings 
$\cO_{v,c+k} \twoheadrightarrow \cO_{v,c+k}/ \p_v^k \cO_{v,c}$. It is clearly surjective and induces the isomorphism of groups
$$\cO_{v,c+k}^\times/\big(1+\p_v^k \cO_{v,c}\big) \simeq \left(\cO_{v,c+k}/ \p_v^k \cO_{v,c}\right)^\times.$$
Therefore, 
\begin{align*}
\cO_{v,c}^\times/\cO_{v,c+k}^\times &\simeq \cO_{v,c}^\times/ (1+ \p_v^k \cO_{v,c}) \big/ \cO_{v,c+k}^\times/ (1+ \p_v^k \cO_{v,c})\\
&\simeq \left(\cO_{v,c}/ \p_v^k \cO_{v,c}\right)^\times \big/ \left(\cO_{v,c+k}/ \p_v^k \cO_{v,c}\right)^\times
\\
&\simeq \left(\cO_{v,c}/ \p_v^k \cO_{v,c}\right)^\times \big/ \left(\cO_{F_v}/ \p_v^k \cO_{F_v}\right)^\times
\end{align*}
Recall $\F^k(v)=\cO_{F_v}/ \p_v^k \cO_{F_v}$ and consider the homomorphism of rings
$\F^k(v)[X]\to \cO_{v,c}/ \p_v^k \cO_{v,c}$, given by $X\mapsto (\varpi_v^{c}\alpha_v \mod \p_v^k\cO_{v,c})$. 
The kernel contains $\langle X^2\rangle$ (because $c \ge k$) and gives a surjective map between two sets with the same order. 
Hence, for any $c\ge k>0$ that
\begin{align*}\cO_{v,c}^\times/\cO_{v,c+k}^\times\simeq & \left(\F^k(v)[\upepsilon]\right)^\times\big/ \F^k(v)^\times.\qedhere \end{align*}
\end{proof}}
In particular, $\mathbb{G}_{v} :=\mathds{E}^1(v)^\times/\F^1(v)^\times \simeq \O_{v,0}^\times/\O_{v,1}^\times $
is (cyclic if $v$ is inert in $E/F$) of order $q_v-\left(\frac{\mathfrak{d}_E}{\p_v}\right)$ and $ \mathbb{G}_{v}(\upepsilon):={\F^1(v)[\upepsilon]^\times/\F^1(v)^\times}\simeq \O_{v,c}^\times/\O_{v,c+1}^\times$.
Finally, for every $c>0$, we have the following short exact sequence (of abelian groups)
$$\begin{tikzcd}[]
    1\arrow{r}&\frac{\O_{v,1}^\times}{\O_{v,c}^\times}  \arrow{r} &\frac{\O_{v,0}^\times}{\O_{v,c}^\times}\arrow{r}& \frac{\O_{v,0}^\times}{\O_{v,1}^\times}\arrow{r}&1,
\end{tikzcd}$$
thus
$$\#\left(\dfrac{\O_{v,0}^\times}{\O_{v,c}^\times}\right)=|\mathbb{G}_{v}(\upepsilon)|^{c-1}\cdot |\mathbb{G}_{v}|= q_v^{c-1} \left(q_v-\left(\frac{\mathfrak{d}_E}{\p_v}\right)\right).$$
\subsection{Galois groups}
\begin{definition}
Set $$\mathcal{P}:=\left\{\p \in \Spec(\cO_F)\colon \p \text{ is unramified in }E/F, \p\not\in S, \p \nmid \fc_1, \p \cO_E \nmid I_0\right\}\nomenclature[G]{$\mathcal{P},\mathcal{P}_{sp},\mathcal{P}_{in}$}{Primes unramified, away from $S$,$\fc_1$ and from $I_0$, split, inert},$$
and let $\mathcal{P}_{sp}$, (resp. $\mathcal{P}_{in} $) denotes the subset of $\mathcal{P}$ of prime ideals of $F$ that are split, (resp. inert) in $E/F$. 

Denote by $\mathcal{N}= \bigcup_{r\ge 1} \mathcal{N}^r =\bigcup_{r\ge 1} \cup_{s=0}^{s= r} \left(\mathcal{N}_{sp}^{s}  \cdot  \mathcal{N}_{in}^{r-s}\right)$ the set of square-free products of ideals in $\mathcal{P}$:
$$\forall r\ge 1 \colon \mathcal{N}_{?}^r:=\{\p_{1} \cdots \p_{r}\colon \p_j \in \mathcal{P}_{?} \text{ distinct}\} \text{ for } ?\in \{\text{in, sp}\}.\nomenclature[G]{$\mathcal{N}_{?}^r$}{$\{\p_{1} \cdots \p_{r}\colon \p_j \in \mathcal{P}_{?} \text{ distinct}\} \text{ for } ?\in \{\text{in, sp}\}$}$$
\end{definition}
So in particular, one has:
\begin{proposition}\label{VII11}
Let $\ff=\prod_{i}\p_i\in\mathcal{N}^r$ and $\p\in\mathcal{P}$ prime to $\ff$, i.e. $\p\ff \in \mathcal{N}^{r+1}$. The extension $\cK(\p\ff)/\cK(\ff)$ is of degree $(q_v-\left(\frac{\mathfrak{d}_E}{\p_v}\right))/u(r)$, where 
$$u(0)=[E^\times \cap \A_{F,f}^\times \mathfrak{O}_1: F^\times]= c  [\cO_E^\times \cap \mathfrak{O}_1 : \cO_F^\times], \text{ with }c\in \{1,2\}\nomenclature[G]{$u(r)$}{$u(0)=[E^\times \cap \A_{F,f}^\times \mathfrak{O}_1: F^\times]$ and $u(r)=1$ for $r\ge1$}$$
and $u(r)=1 $ if $r\ge1$.
\end{proposition}
\begin{proof}
If $r\ge 1$ this is a special case of Corollary \ref{galoisgroupsI0}, which says
$$\Gal(\cK(\p\ff)/\cK(\ff)) \simeq \mathbb{G}_{v},$$
thus of order $q_v-\left(\frac{\mathfrak{d}_E}{\p_v}\right)$.
If $r=0$, recall that by (\ref{galoistransfersequence}) and since $\p\cO_E \nmid I_0$, we have an exact sequence
\begin{equation}\label{surjectioningalois}\begin{tikzcd} \frac{E^\times \cap \A_{F,f}^\times \mathfrak{O}_{1}^\times }{F^\times}\arrow[r,hook]&
\frac{\mathfrak{O}_{1}^\times}{\mathfrak{O}_{\p}^\times}\simeq \frac{\cO_{v_\p,0}^\times}{\cO_{v_\p,1}}\arrow[r,two heads]&
\Gal(\cK(\p)/\cK),
\end{tikzcd}\end{equation}
hence, the extension $\cK(\p)/\cK$ is of degree $(q_v-\left(\frac{\mathfrak{d}_E}{\p_v}\right))/[E^\times \cap \A_{F,f}^\times \mathfrak{O}_{1}^\times :F^\times]$.
\end{proof}
{\subsection{Global to local Galois/Hecke action on cycles}
We have a surjection (see \S \ref{defcycles}) of $\underline{\T}^1(\A_f)\times \cH_K$-modules
$$\begin{tikzcd}\pi_{\text{cyc}}\colon\Z[\H^{\der}(\A_f)\backslash \G(\A_f)/ K]\arrow[r, twoheadrightarrow] & \Z[\cZ_{\G,K}(\H)].\end{tikzcd}$$
The left-hand-side module is factorizable (see \cite[Def. 5.8]{GetHah2019}), i.e.:
$$\Z[\H^{\der}(\A_f)\backslash \G(\A_f)/ K] = \Z[\underline{\H}^{\der}(F_S)\backslash \underline{\G}(F_S)/ K_S] \otimes \bigotimes_{v \not\in S}^\prime\Z[\underline{\H}^{\der}(F_v)\backslash \underline{\G}(F_{v})/K_v],$$
where, $\otimes_{v \not\in S}^\prime$ is the restricted product with respect to the elements\footnote{The notation $[g]_{v}\nomenclature[G]{$[g]_{v}$}{The class ${{\H}^{\der}(F_v)\cdot g \cdot K_v}\in {\H}^{\der}(F_v)\backslash {\G}(F_{v})/K_v$}$, ($g\in {\G}(F_{v})$), is for ${{\H}^{\der}(F_v)\cdot g \cdot K_v}\in {\H}^{\der}(F_v)\backslash {\G}(F_{v})/K_v$.} $$\{[1]_{v} \in \underline{\H}^{\der}(F_v)\backslash \underline{\G}(F_{v})/K_v\}_{v \not\in S},$$and the equality above intertwines the action of  $\cH_K$ (resp. $\underline{\T}^1(\A_f)$) with the action of $$\cH_{K_S} \otimes \bigotimes_{v \not\in S}^\prime \cH_{K_v}\quad (\text{resp. } \underline{\T}^1(F_S)\times \prod_{v \not\in S}^\prime \underline{\T}^1(F_v)),$$
where, $\cH_{K_S}:=\text{End}_{\Z[\underline{\G}(F_S)]}\Z[\underline{\G}(F_S)/K_S]$ and ${\cH_{K_v}=\text{End}_{\Z[\underline{\G}(F_v)]}\Z[\underline{\G}(F_v)/K_v]}$, $v \not\in S$.
where $\otimes_{v \not\in S}^\prime$ is  the restricted product with respect to
$$\{\text{Id}_{\underline{\G}(F_v)/K_v}\}_{v \not\in S} \quad (\text{resp. } \{\underline{\T}^1(\cO_{F_v})\}_{v \not\in S}).$$

\section{Proof of distribution relations in split cases}\label{splitunitarycase}
Let $v\in \mathcal{P}_{sp}$. We recall once again that an embedding $\iota_v \colon \overline{F} \hra \overline{F}_v$ has been fixed in at the beginning of \S \ref{SectionShimura}. Let $w$ the place of $E$ above $v$ determined by $\iota_v$ and $\overline{w}$ its conjugate. We abuse notation and also write $w$ for the place above $w$ determined by this choice in any field extension of $E$ contained in the fixed algebraic closure $\overline{F}$. Let $\varpi_v$ be a uniformizer for $F_v = E_w$, $q_v$ for the cardinality of the residue field and $p$ for the rational prime below $v$. 
\subsection{Normalization isomorphism in the split case}\label{normalization}
We identify the group ${\U}(V)_{/{F_v}} \times {\U}(W)_{/{F_v}}$ with $\GL(V_{w})_{/{F_v}}  \times \GL(W_{w})_{/{F_v}} $ as follows: Recall that for any Hermitian $E$-space $\mathcal V$ 
$$
{\U}({\mathcal{V}})(F_v) = \{g \in \GL(\mathcal V)(E \otimes_F F_v) \colon \psi_{v}( gx, gy ) = \psi_{v}( x, y), \ \forall x, y \in \mathcal V \otimes_F F_v \}, 
$$
where $\psi_{v}= \psi_{F_v}$ (see footnote \ref{tensorhermitianspace} in on page ~\pageref{tensorhermitianspace}). We have
$$
E_v = E \otimes_F F_v = E_{{w}} \oplus E_{\overline{w}} \simeq F_v \oplus F_v,
$$
where, the action of complex conjugation on the left-hand side corresponds to the involution $(s, t) \mapsto (t, s)$ on the right-hand side. Thus, one has $ \mathcal V \otimes_F F_v = \mathcal V_{ {w}} \oplus \mathcal V_{\overline{w}}$,
and
$$
\GL(\mathcal V)(E \otimes_F F_v) = \GL(\mathcal V_{{w}}) \times \GL(\mathcal V_{\overline{w}}) \simeq \GL(\mathcal V)(F_v) \times \GL(\mathcal V)(F_v).
$$
The hermitian form $\psi_v$ takes values in $E_v=E_w\oplus E_{\overline{w}}$, write $\psi_v=(\psi_w,\psi_{\overline{w}})$ for its two component. For any $x, y\in\mathcal V \otimes_F F_v=\mathcal V_{ {w}} \oplus\mathcal  V_{\overline{w}}$, write $x = x_{{w}} + x_{\overline{w}}$ and $y = y_{ {w}} + y_{\overline{w}}$. Recall that the original hermitian form $\psi$ is semi-linear on the right. By definition $\psi_v(x, y)= \left (\psi_{w}(x,y),\psi_{\overline{w}}(x,y) \right )$, accordingly we have
\begin{align*}
\psi_v(x, y)
&=
\psi_v\left(x_{{w}},y_{\overline{w}} \right)+ \psi_v\left(x_{\overline{w}},y_{{w}}\right)+ \psi_v\left(x_{{w}},y_{{w}}\right)+\psi_v\left(x_{\overline{w}},y_{\overline{w}}\right)\\
\overset{(1)}&{=}\psi_v\left(x_{{w}},y_{\overline{w}} \right)+ \psi_v\left(x_{\overline{w}},y_{{w}}\right)&\\
&= \left(\psi_{w}(x_w,y_{\overline{w}}),\psi_{\overline{w}}(x_w, y_{\overline{w}})\right)+\left(\psi_{{w}}(x_{\overline{w}}, y_w),\psi_{\overline{w}}(x_{\overline{w}}, y_w)\right)\\
&=\left(\psi_{w}(x_w,y_{\overline{w}}),0\right)+\left(0,\psi_{\overline{w}}(x_{\overline{w}}, y_w)\right)\\
&=\left(\psi_{w}(x_w,y_{\overline{w}}),\psi_{\overline{w}}(x_{\overline{w}}, y_w)\right)
\end{align*}
{For (1), we have used the fact that $\psi_v=0 \text{ on }\mathcal V_w\times \mathcal V_w \text{ and } \mathcal V_{\overline{w}}\times \mathcal V_{\overline{w}}$. To show this, let $x,y$ be any elements in $\mathcal V_w$ then $x=(1,0)x$ where $(1,0)\in F_v\oplus F_v$, so $\psi_v(x,y)=\psi_v((1,0)x,(1,0)y)=(1,0){(1,0)}^\tau \psi_v(x,y)=(0,0)\psi_v(x,y)=0$. We show similarly the annihilation of $\psi_v=0$ on $ \mathcal V_{\overline w}\times \mathcal V_{\overline w}$.} Accordingly, we have
\begin{enumerate}[wide=0pt, nosep]
\item[i.] $\psi_w=0$ on $\mathcal V_{\overline{w}}\times \mathcal V$ and $\mathcal V\times \mathcal V_w$ and induces a perfect pairing ${\psi_w\colon \mathcal V_{{w}}\times \mathcal V_{\overline{w}}\to F_v}$.
\item[ii.] $\psi_{\overline w}=0$ on $\mathcal V_{{w}}\times \mathcal V$ and $\mathcal V\times \mathcal V_{\overline w}$ and induces a perfect pairing ${\psi_{\overline w}\colon \mathcal V_{\overline {w}}\times \mathcal V_{{w}}\to F_v}$.
\item[iii.] $\psi_{\overline{w}}(x_{\overline{w}}, y_w)=\psi_{{w}}( y_w,x_{\overline{w}})$, since $\psi_v( x,y)=\tau({\psi_v( y,x)})$.
\end{enumerate}
Let $g= (g_1,g_2) \in \GL(\mathcal V_{{w}}) \times \GL(\mathcal V_{\overline{w}})$
\begin{align*}
\psi_v((gx,gy))&=\psi_v\left((g_1, g_2) (x_w+x_{\overline w}), (g_1, g_2) (y_w+y_{\overline w})\right)\\
&=\psi_v\left(g_1x_{ w}+g_2 x_{\overline w} x, g_1y_{ w}+ g_2 y_{\overline w}\right)\\
&=\left(\psi_{w}(g_1x_w,g_2y_{\overline{w}}),\psi_{\overline{w}}(g_2x_{\overline{w}}, g_1y_w)\right)
\\&= \left(\psi_{w}(g_1x_w,g_2y_{\overline{w}}),\psi_{{w}}( g_1y_w,g_2x_{\overline{w}})\right)\end{align*}
Now, if $g= (g_1,g_2)$ is in $\underline{\U}_{\mathcal{V}}(F_v)\subset \GL(\mathcal V_{{w}}) \times \GL(\mathcal V_{\overline{w}})$, then
$$
\psi_v((gx,gy))=\psi_v( x, y) = \left ( \psi_w(  x_{{w}},y_{ \overline{w}} ), \psi_w (y_{ {w}}, x_{ {\overline{w}}}) \right ). 
$$
Hence, $g= (g_1,g_2) \in \underline{\U}_{\mathcal{V}}(F_v)$ if and only if $\psi_w( g_1x_{{w}} ,g_2y_{ \overline{w}})=\psi_w(x_{{w}}, y_{ \overline{w}} )$ for all $x_{{w}} \in\mathcal V_{{w}}$ and $y_{\overline{w}} \in\mathcal V_{\overline{w}}$. The discussion above shows that the projection $g = (g_1, g_2) \mapsto g_1$ defines an isomorphism $
\underline\U(\mathcal V)(F_v) \simeq \GL(\mathcal V_{{w}})$ that is actually defined over $\cO_{F_v}$. Now, since ${w}$ is the place of $E$ corresponding to the fixed embedding $\iota_v\colon \overline{F} \hra \overline{F}_v$, there is no ambiguity in writing $\mathcal V_{{w}}$ as $\mathcal V_v$ and viewing it as a vector space over $ F_v=E_{{w}} $. Therefore, we get the desired identifications.
\begin{remark}
If we follow Remark \ref{latticesandmodel} and pick up an $\cO_E$-lattice $\mathcal L_W$ in $W$, $\mathcal L_D$ in $D$ and $\mathcal L_V=\mathcal L_W \oplus \mathcal L_D$ in $V$, then working with Hermitian spaces over $\cO_{F_v}$ gives an identification:
$$\underline{\U}_{V,\O_{F_v}} \simeq \GL_{n+1,\O_{F_v}}\quad\text{and}\quad\underline{\U}_{W,\O_{F_v}} \simeq \GL_{n,\O_{F_v}}$$
using the previous formulas throughout. We will retain this notation in \S \ref{splitunitarycase}.
\end{remark}
As we have pointed out in \S \ref{interludeorders}, under the identification $E_v\simeq F_v \oplus F_v$, the maximal order $\cO_{E_v}$ of the \'etale algebra $E_v$ identifies with $\cO_{E_v} \simeq \cO_{E_w} \oplus \cO_{E_{\overline{w}}} \simeq \cO_{F_{\tau}} \oplus \cO_{F_{\tau}}$. The groups
$$U_v^1(c)=\underline{\nu}({\O_{v,c}}^\times)= \underline{\nu}({\O_{F_v}+\varpi_{v}^c\O_{E_v}}^\times)=\{(z,z^{-1}) \colon z \in 1+ \varpi_v^c \cO_{F_v}\}\text{ for }c\in \N,$$
defined in \S \ref{fieldkappa}, yield the decreasing filtration $(H_c := \det^{-1}(U_v^1(c)))_{c \in \N}$, on $\H(F_v)$. The $E$ determinant map on ${\U}({\star})$ becomes under the isomorphism ${\U}(V)_{/F_v}\simeq  \GL(\star)_{/F_v}$ the usual determinant map and the groups $U_v^1(c)$ become $1+\varpi^c \cO_{F_v}$.

By abuse of notation, we will also use the notation $H_c$ for the corresponding subgroup ${\det} ^{-1} (1 + \varpi_v^c \cO_{F_v}) \subset \GL_n(F_v)$ via the isomorphism $\H(F_v) \simeq \GL_n(F_v)$ fixed above.
\subsection{Action of Frobenii on unramified special cycles}\label{subsec-frob}
Let $g\in \G(\A_f)$. 
Assume that the field of definition of the cycle $\mathfrak{z}_g$ is unramified at $v$. 
The description of the Galois action on special cycles in \S \ref{galoisH} using $\H(\A_{f})$, together with the discussion in \S \ref{normalization} imply that the action of $\Fr_{w}\in \Gal(E_w^{un}/E_w)$ (the {geometric} Frobenius) on the cycle $\mathfrak{z}_g$ is given by 
$$
\Fr_{w} \cdot \mathfrak{z}_g = \mathfrak{z}_{\Delta_v(\textbf{Frob}_v) \cdot g}, 
$$
where, $\textbf{Frob}_v \in \GL(W_v) \simeq \U(W)(F_v)$ is any matrix verifying $$\text{ord}_{{F_v}}(\det(\textbf{Frob}_v)) = 1,  \text{ e.g.}\; \textbf{Frob}_v = \diag(\varpi_v,\cdots, 1),$$ and $\Delta_v$ is the composition of the following natural embeddings:
$$\begin{tikzcd}[column sep= small]\Delta_v\colon \U(W)(F_v) \arrow[r, hook] &\U(W)(\A_{F,f}) \arrow[r, hook] &\U(V)(\A_{F,f}) \times \U(W)(\A_{F,f}) \simeq \G(\A_f).\end{tikzcd}$$
\subsection{The Hecke polynomial for split places}\label{HeckPolSplit}
As usual, to ease notation, let $\star$ denote some/any element in $\{V,W\}$. We continue with the fixed place $v\not\in S$ of $F$ that splits in $E$ to $w \overline{w}$.
\subsubsection{A local pair}\label{localpair}

The choice of an embedding $\overline{\Q} \hra \overline{F}_v$ extending the distinguished fixed embedding $\iota_v\colon F_v \to \overline{F}_v$ induces an identification:
$$\begin{tikzcd}[row sep= tiny] \Hom(F,\R) = \Hom(F,\overline{\Q})\arrow{r}{\simeq}& \Sigma_{F,v}=\Hom(F,\overline{F}_v)\\  \widetilde{\iota}\arrow[r, mapsto]& \iota_v\circ \widetilde{\iota}.  \end{tikzcd}$$
We have
$$\G_{\star,\overline{F}_v}\simeq \prod_{\widetilde{\iota} \in {\Sigma}_{F,v}}\G_{\star, \widetilde{\iota}} \quad \text{ where } \quad\G_{\star,\widetilde{\iota}}=\GL(\star\otimes \overline{F}_v)\simeq \GL(\dim_E\star)_{\overline{F}_v}.$$
Recall that the conjugacy class of $\mu_h$ with $h=h_{\mathfrak{B}_V}\times h_{\mathfrak{B}_W}  \in \cX$ is independent of the choice of $h$ and is defined over the reflex field $E$ (\S \ref{reflexfield}), hence $[\mu_{h}]\in \mathcal{M}(E)$. Now using the fixed embedding \begin{tikzcd}[column sep=normal] E \arrow[hook]{r}{}&  \overline{F}\arrow[hook]{r}{\iota_{v}}& \overline{F}_v,\end{tikzcd} we get elements $[\mu_{\star,v}]\in \mathcal{M}_{\G_{\star}}(\overline{F}_v)$ and $[\mu_{h,v}]= [\mu_{V,v} \oplus \mu_{W,v}]\in \mathcal{M}_{\G}(\overline{F}_{v})$, where $\mu_{\star,v}$ is given on $\overline{F}_v$-points by
$$\begin{tikzcd}[column sep= small]\mu_{\star,v}\colon \Gm_{m,\overline{F}_{v}}\to  \G_{\overline{F}_{v}}  & t \arrow[r, mapsto]& \left( \begin{pmatrix}
t&\\
&\text{Id}_{\dim_E\star-1}
\end{pmatrix},\text{Id}_{\dim_E\star},\cdots,\text{Id}_{\dim_E\star} \right).  \end{tikzcd}$$
The only nontrivial component is the one corresponding to the distinguished embedding $\iota_1$ under the identification $\Hom(F,\overline{\Q}){\simeq} \Sigma_{F,v}$. We will then write by abuse of notation $\mu_{h,v}\colon \Gm_{m,\overline{F}_{v}} \to \U(V)_{\overline{F}_{v}}\times \U(W)_{\overline{F}_{v}}$ given by:
$$\begin{tikzcd}t \arrow[r, mapsto]& \left(\begin{pmatrix}
t&\\
&1_{n}
\end{pmatrix},\begin{pmatrix}
t&\\
&1_{n-1}
\end{pmatrix}\right) \in \T(\mathfrak{B}_{V})(\overline{F}_v)\times \T(\mathfrak{B}_{W})(\overline{F}_v),\end{tikzcd}$$
{Here, we have used the identification  ${\U}(\star)_{/\overline{F}_v} \simeq \GL(\star_{w})_{/\overline{F}_v} $ given by the choice of the place $w$ over $v$, accordingly $\T(\mathfrak{B}_{\star})_{\overline{F}_v}$ identifies with the maximal $\overline{F}_v$-torus of diagonal matrices in $\GL(\dim_E\star)_{\overline{F}_v}$ with respect to the local basis $\mathfrak{B}_{\star,v}$ induced from the fixed global basis $\mathfrak{B}_\star$ (\S \ref{Hermsymdom})}.

Note that $[\mu_{h,v}]$ is independent of the choice of $\iota_{v}$ and is invariant under the action of $\Gal(\overline{F}_v/E_{w})$ ($E_w=F_v$), thus, by \cite[(b) Lemma 1.1.3]{Ko1} one finds that $[\mu_{h,v}]$ is actually an element of $\mathcal{M}_{\underline{\G}_{{F_v}}}(E_{w})=\mathcal{M}_{\underline{\G}_{F_v}}(F_{v})$, i.e. defined over $F_v$. {In fact, the geometric conjugacy class $[\mu_{h,v}]\in \mathcal{M}_{\underline{\G}_{{F_v}}}(E_w)$ contains a cocharacter $$\mathds{G}_{m,\O_{E_{w}}} \to \underline{\G}_{\O_{E_w}},$$ defined over the valuation ring $\O_{E_w}$ (See \cite[Lemma 3.3.11]{Wan2018})}.
 
In summary, we get a pair $(\underline{\G}_{\cO_{F_v}},[\mu_{h,v}])= (\GL_{n+1,\cO_{F_v}}\times \GL_{n,\cO_{F_v}}, [\mu_{h,v}])$, composed of a $F_v$-reductive group and a minuscule $\underline{\G}(F_v)$-conjugacy class $[\mu_{h,v}] \in \mathcal{M}_{\underline{\G}_{F_v}}(F_v)$. Moreover, if we make the standard choices\footnote{We use a subscript $w$ here, to keep track of the place $w$ aboce $v$ used to indentofy $\underline{G}_v$ with $\GL_{n+1,F_v}\times \GL_{n,F_v}$.}
$${\T}_w:=\{\text{diagonal matrices}\} $$
and 
$$\B_w := \{\text{upper-triangular matrices}\} \subset \GL_{n+1,\cO_{F_v}}\times \GL_{n,\cO_{F_v}},$$
then the cocharacter $\mu_{h,v}$ given above, is the unique one in the class $[\mu_{h,v}]$ that is dominant with respect to $\B_w$. 

\subsubsection{Dual group}
The complex dual of $\GL_{n+1,F_v}\times \GL_{n,F_v}$ is 
$$\widehat{\GL}_{n+1,F_v}\times \widehat{\GL}_{n,F_v}=\GL_{n+1}(\C)\times \GL_{n}(\C).$$
Let $(\widehat{\B}_w,\widehat{\T}_w)$ be the standard Borel pair dual to $({\B}_w,{\T}_w)$, that is the upper triangular matrices and its maximal torus of diagonal matrices. For convenience, we write $\widehat{\B}_w=\widehat{\B}_{v,1}\times \widehat{\B}_{v,2}$, $\widehat{\T}_w=\widehat{\T}_{v,1}\times \widehat{\T}_{v,2}$ and $W(\widehat{\T}_w):=W(\GL_{n+1}(\C)\times \GL_{n}(\C) , \widehat{\T}_w)$ for the Weyl group. The Galois group $\Gal(F_v^{un}/F_v)$ acts trivially on dual of $ {\GL}_{n+1,F_v}\times {\GL}_{n,F_v}$, since it a split group, i.e.
$${}^L( {\GL}_{n+1,F_v}\times {\GL}_{n,F_v}) =\GL_{n+1}(\C)\times \GL_{n}(\C) \times \Gal(F_v^{un}/F_v).$$
\subsubsection{The character $\widehat{\mu}_{h,v}$}
Under the identification $X^*(\widehat{\T}_w) = X_*({\T}_w)$, the Weyl orbit of $\mu_{h,v}$ corresponds to a Weyl $W(\widehat{\T}_w)$-orbit of characters of $\widehat{\T}_w$. There is a unique $\widehat{\mu}_{h,v} \in X^*(\widehat{\T}_w)$ in this Weyl orbit that is dominant with respect to the Borel subgroup $\widehat{\B}_w$ and it is explicitly given on $\C$-points by
$$\begin{tikzcd}\widehat{\mu}_{h,v} \colon {\T}_w(\C)\arrow{r}& \C,& \left(\diag(z_1, \cdots, z_{n+1}), \diag(z_1', \cdots, z_n')\right)\arrow[r, mapsto]& z_1 z_1' .\end{tikzcd}$$
\subsubsection{The representation}The representation
$$\begin{tikzcd}r\colon\GL_{n+1}(\C)\times \GL_{n}(\C) \arrow{r}& \GL_{n(n+1)},\end{tikzcd}$$
we are interested in is the irreducible representation whose highest weight relative to the Borel pair $(\widehat{\B}_w,\widehat{\T}_w)$ is $\widehat{\mu}_{h,v}$ \cite[\S 5.1]{BR94}. Let $r_?$ be the standard $?$-dimensional representation of $\GL_{?}(\C)$, then $r$ is the representation on the $2$-fold tensor product  $r_{n+1}\otimes r_n$ defined as follows, for any $g=(g_1, g_2) \in \GL_{n+1}(\C)\times \GL_{n}(\C)$,
$$r(g)=r_{n+1}(g_1)\otimes r_{n}(g_2).$$
Finally, extend $r$ to a representation of ${}^L(\GL_{n+1,F_v}\times \GL_{n, F_v})$ (also called $r$) by letting the Galois group act trivially everywhere.

\subsubsection{The Hecke polynomial}\label{Heckepol}
\textbf{Definition.} Following \cite[Definition \ref{defheckepol}]{Seedrelations2020}, we attach to the pair $({\underline{\G}}_{F_v}, [\mu_{h,v}])$ the polynomial
$$
H_{w}(z)=\det\left(z-q_v^{{\langle\mu_{h,v}, \rho_{v} \rangle}} r(g) \right ) \in \C[\GL_{n+1}(\C)\times \GL_{n}(\C)] [z],
$$
where, $\rho_v$ is the halfsum of all positive roots of $({\B}_w,{\T}_w)$, thus $ \langle\mu_{h,v}, \rho_{v} \rangle= {2n-1}$. 
The ring of coefficients of this polynomial, is the ring of class functions on $\GL_{n+1}(\C)\times \GL_{n}(\C)$. 
A class function on $\GL_{n+1}(\C)\times \GL_{n}(\C)$ restricts to a Weyl-invariant function on $\widehat{\T}_w$. Conversely, by a classical argument of Chevalley, the subalgebra of Weyl-invariants functions on $\widehat{\T}_w$ consists precisely of functions which arise from class functions on $\GL_{n+1}(\C)\times \GL_{n}(\C)$. Therefore, one can identify $H_{w}(z)$ with a polynomial in $\C[\widehat{\T}_w]^{W(\widehat{\T}_w)}[t]$, which, abusing notation, we continue to denote by $H_{w}$:
$$H_{w}(t)=\det\left(z-q_v^{{\langle\mu_{h,v}, \rho_{v} \rangle}} r_{|\widehat{\T}_w}(t) \right ) \in \C[\widehat{\T}_w] [t].$$
Let $t=\left(\diag{(x_1, \cdots ,x_{n+1})},\diag{(y_1, \cdots, y_{n})}\right) \in \widehat{\T}_w$, the polynomial then identifies with 
$$
H_{w}(t)=\prod_{i=1}^{n+1}\prod_{j=1}^{ n}\left(z-q_v^{{2n-1}} x_iy_j\right) \in \C[\widehat{\T}_w][t]. 
$$
Here, $x_i y_j$ denotes the function $x_i y_j  \colon \widehat{\T}_w \ra \C$ defined by  
$$
(\diag{(x_1, \cdots, x_{n+1})}, \diag{(y_1,\cdots, y_{n})}) \mapsto x_iy_j.
$$  
Now, using the identification\footnote{This identification is obtained as follows: On the one hand, by definition of the dual torus, the ring of algebraic functions on $\widehat{\T}_w$ are precisely the elements of $\C[X_*(\T_w)]$. On the other hand, sending any $F_v$-algebraic cocharacter $\chi\colon \mathds{G}_{m,{F}_v} \ra \T_w$ to $1_{\chi(\varpi_v)\T_w(\cO_{F_v})}$, yields an identification $\C[X_*(\T_w)] \simeq \cH_\C(\T_w(F_v) \sslash \T_w(\cO_{F_v}))$.} $\C[\widehat{\T}_w]\simeq \mathcal H_\C({\T}_w(F_v) \sslash {\T}_w(\O_{F_v}))$, this function corresponds to the element $\mathbf{1}_{({g_i},{h_j})T_c} \in \cH_\C({{\T_w}}(F_v) \sslash {{\T_w}}(\O_{F_v}))$ where 
$$
g_i= \diag{(\underbrace{1, \dots 1}_{i-1}, \varpi_v, 1 \dots,1)}\text{ and } h_j=\diag (\underbrace{1 ,\dots 1}_{j-1}, \varpi_v, 1 ,\dots,1).
$$
Since the Weyl group $W(\widehat{\T}_w)$ permutes the $x_i$'s and $y_j$'s, it is clear that $H_w(t) \in \C[\widehat{\T}_w]^{W(\widehat{\T}_w)}[z]$. Subsequently, using the Satake isomorphism$$\C[\widehat{\T}_w]^{W(\widehat{\T}_w)} \simeq \cH_\C(\T_w(F_v) \sslash \T_w(\cO_{F_v}))^{W(\T_w)}\simeq \cH_\C(\underline{\G}(F_v) \sslash \underline{\G}(\cO_{F_v})),$$ 
we may also view $H_w$ as a polynomial with coefficients in the local spherical Hecke algebra
$$\cH(\GL_{n+1}(F_v)\sslash \GL_{n+1}(\cO_{F_v}),\Z[q_v^{\pm 1/2}])\times \cH(\GL_{n}(F_v)\sslash \GL_{n}(\cO_{F_v}),\Z[q_v^{\pm 1/2}]).$$
\textbf{Explicit polynomial.} In the remaining part of this section, we will give an explicit formulation for the Hecke polynomial with coefficients in the spherical local Hecke algebra, i.e. we will invert the Satake isomorphism. Let us begin by rewriting this polynomial in a more suitable form: 
\begin{align*}
H_{w}(t)
&=\prod_{j=1}^{ n} \prod_{i=1}^{n+1}\left(z-q_v^{{2n-1}} x_iy_j\right)\\
&=\prod_{j=1}^{ n}  (\sum_{k=0}^{n+1}(-1)^{k}q_v^{{(2n-1)k}}X_{k} {y_j^{k}} z^{n+1-k})\\
&=\sum_{k=0}^{n(n+1)} \big(\sum_{(a_i)\in p_{n}(k)} \prod_{i=1}^{n} (-1)^{a_i}q_v^ {{2n-1}a_i}X_{a_i} y_{i}^{a_i}\big)  z^{n(n+1)-k}
\end{align*}
where, $p_{n}(k):=\{(a_i)_{1\le i \le n} \in \N \colon \sum_{i=1}^n a_i=k, 0\le a_i\le n+1\}$ and $X_?$ is the symmetric monomial associated to the monome $x_1 x_2 \cdots x_?$, for $1 \le ? \le n+1$. Here, the symmetric permutation group $S_n$ acts on the set of partitions $p_{n}(j)$ and yields 
\begin{align*}H_{w}(t)&=\sum_{k=0}^{n(n+1)} \big(\sum_{(a_i)\in p_{n}(k)} \prod_{i=1}^{n} (-1)^{a_i}q_v^ {(2n-1)a_i}X_{a_i} y_{i}^{a_i}\big)  z^{n(n+1)-k}\\
&= \sum_{k=0}^{n(n+1)} \big(\sum_{(a_i)\in p_{n}(k)/S_n}  (-1)^{\sum a_i} q_v^{{2n-1}\sum_{i=1}^n a_i} \prod_{i=1}^n X_{a_i}    \cdot \prod_{i=1}^{n} (\sum_{j=1}^n y_{j}^{a_i})\big) z^{n(n+1)-k}\\
&= \sum_{k=0}^{n(n+1)}  (-1)^k q_v^{{k(2n-1)}}  \big(\sum_{(a_i)\in p_{n}(k)/S_n}  \prod_{i=1}^n X_{a_i}  \cdot \prod_{i=1}^{n} (\sum_{j=1}^n y_{j}^{a_i})\big) z^{n(n+1)-k}\\
&= \sum_{k=0}^{n(n+1)}  (-1)^k q_v^{{k(2n-1)}}  \big(\sum_{(a_i)\in p_{n}(k)/S_n}  \prod_{i=1}^n X_{a_i}  \cdot \prod_{i=1}^{n} Y^{(a_i)}\big) z^{n(n+1)-k}
\end{align*}
where, $Y^{(?)}$ denotes the power sum symmetric monomial $\sum_{i=1}^n y_i^?$, for $?\ge 1$. Set $Y_?$ for the symmetric monomial associated to the monome $y_1 y_2 \cdots y_?$, for $1 \le ? \le n$. For any $k \ge 1$, the Newton--Girard formula says:
$$Y^{(k)} =  \begin{cases} (-1)^k k Y_k +\sum_{i=1}^{k-1} (-1)^{k-1+i} Y_{k-i}Y^{(i)}& \text{ if } 1\le k \le n,\\
\sum_{i=k-n}^{k-1} (-1)^{k-1+i} Y_{k-i}Y^{(i)}& \text{ if } n<k.
\end{cases}$$
There exists then, a polynomial in $n$ variables
 $Q_k(z_1,\cdots ,z_n)\in \Z[z_1,\cdots ,z_n]$ such that the power sum $Y^{(k)}$ is given by
$$Y^{(k)}= Q_k(Y_1, \dots, Y_n).$$
Set $T_{k,\star}$ for the Hecke operator ${\GL_{\star}(\cO_{F_v})\cdot \diag(\underbrace{\varpi_v,\cdots,\varpi_v}_{k},1,\cdots,1)\cdot \GL_{\star}(\cO_{F_v})}$, for $\star\in \{V,W\}$. 
The Satake isomorphism $\C[\widehat{\T}_w]^{W(\widehat{\T}_w)} \simeq {\cH_\C(\underline{\G}(F_v) \sslash \underline{\G}(\cO_{F_v}))}$ yields the following identifications 
\begin{equation}\label{heckesatakeparameterGL}\begin{tikzcd}[row sep= tiny]
X_k , (1\le k\le n+1)\arrow[rr, leftrightarrow]&&q_v^{-\frac{(n+1-k)k}{2}} T_{k,V},\\
Y_k , (1\le k\le n)\arrow[rr, leftrightarrow]&&q_v^{-\frac{(n-k)k}{2}} T_{k,W}
\end{tikzcd}\end{equation}
Indeed, for any fixed integer $?\ge 1$ and every integer $1 \le i\le ?$ we consider the following minuscule\footnote{i.e. the representation $\Ad\circ \lambda_i$ of $\mathds{G}_m$ on $\text{Lie}(\GL_?)$ has no weights other than $1,0,-1$.} cocharcater $\lambda_{i} \colon \mathds{G}_{m,F} \to \GL_?$ given on $F_v$ points by
$$\begin{tikzcd} \lambda_{i}\colon F_v^\times \arrow{r}& \GL_?(F),& t \mapsto \diag(\underbrace{t, \cdots,t}_{i-\text{tuple}},1 \cdots,1).\end{tikzcd}$$
The above identification is an immediate application of Proposition \cite[Proposition \ref{satake}]{Seedrelations2020}, where we also use the fact that the Modulus function for $\GL_\star$ with respect to the Borel subgroup of upper triangular matirces $B_?$ is given by $\delta_{B_?}(\diag(\varpi^{a_i}))=q_v^{-\sum_{j=1}^{?} (? +1-2j)a_j}$, for any diagonal matrix $\diag(\varpi^{a_i}) \in \GL_?(F_v)$.

Consequently we can now write the Hecke polynomial with coefficients in the spherical local Hecke algebra:
\begin{align*}H_{w}(t)
&= \sum_{j=0}^{n(n+1)}  (-1)^k q_v^{\frac{3k(n-1)+\sum_{i=1}^{n} a_i^2}{2}} \big(\sum_{(a_i)\in p_{n}(k)/S_n}  \prod_{i=1}^n T_{a_i,V}  \otimes \prod_{i=1}^{n}T_W^{(a_i)}  \big) z^{n(n+1)-k}
\end{align*}
where, $T_W^{(a_i)}:=Q_{a_i}(q_v^{-\frac{n-1}{2}} T_{1,W},\dots,q_v^{-\frac{(n-j)j}{2}} T_{j,W},\dots, T_{n,W})$.
\begin{example}[Case $n=1$]\label{heckecase1}
we obtain in this situation the following polynomial
$$H_w(t)=t^2 - q_v^{{1/2}} (x_1+x_2)y_1 t+ q_v x_1x_2y_1^2\in \C[\widehat{\T}_w][t],$$
which corresponds to the Hecke polynomial
$$H_w(z)=z^2 - (T_{1,V}\otimes T_{1,W})\, z+ q_v \,(T_{2,V}\otimes T_{1,W}^2),$$
where $T_{1,V},T_{2,V}\in \cC_c(\GL_2(F_v)\sslash \GL_2(\cO_{F_v}), \Z)$ are the Hecke operators corresponding to $${\GL_2(\cO_{F_v})\diag(\varpi_v,1)\GL_2(\cO_{F_v})}, \diag(\varpi_v,\varpi_v)\GL_2(\cO_{F_v})$$ respectively,  and $T_{1,W}^?\in \cC_c(F_v^\times/\cO_{F_v}^\times, \Z)$ is the Hecke operator corresponding to $\varpi_v^?\cO_{F_v}^\times$.
\end{example}
\begin{example}[Case $n=2$]
In this case, using (\ref{heckesatakeparameterGL}) we show that the Hecke polynomial is
\begin{align*}H_{w}(t)
&= z^6 - T_{1,V} \otimes T_{1,W} z^5+ q\big( T_{2,V}\otimes(T_{1,W}^2- 2q_vT_{2,W})+ T_{1,V}^2  \otimes T_{2,W}\big)z^4 \\
&-q_v^{2}(1\otimes T_{1,W}) \big( T_{1,V}T_{2,V}  \otimes T_{2,W}+q_vT_{3,V}\otimes (T_{1,W} ^2- 3q_vT_{2,W})\big)z^3\\
&+q_v^4(1\otimes T_{2,W}) \big(T_{1,V} T_{3,V} \otimes T_{1,W}^2- 2q_v^2+T_{2,V}^2\otimes T_{2,W}\big)z^2\\
&-q_v^{6}\big(T_{2,V}T_{3,V} \otimes T_{1,W}T_{2,W}^2 \big)z\\
&+ q_v^{9}T_{3,V}^2\otimes T_{2,W}^3.
\end{align*}
\end{example}

\subsection{Split local setting}
{\em To ease the reading, we switch to a local notation and omit all subscripts $v$ and $w$.}

Using the fixed basis for $V$ (compatible with $W$), we get to the following situation $\H=\Delta\GL_{n}\hookrightarrow \G=\GL_{n+1}\times \GL_n$, where the embedding on the first factor is given by  
$$
  \iota\colon \NiceMatrixOptions{transparent}\begin{pmatrix}
    \tikzmark{left}{\phantom{a}} &  &     \\
       & *&	  \\
       &  & \tikzmark{right}{\phantom{b}} 
  \end{pmatrix}
  \Highlight[first]
  \mapsto 
 \NiceMatrixOptions{transparent}\begin{pmatrix}
    \tikzmark{left}{\phantom{a}} &  &  &  \\
       & *&	&  \\
       &  & \tikzmark{right}{\phantom{b}}& \\
     &&&1
  \end{pmatrix}
\Highlight[second]
$$ 
We would like to describe the quotient $H\backslash G \slash K$ where $K=\G(\O_F)$. Set $K_1=\GL_{n+1}(\O_F)$ and $K_2=\GL_{n}(\O_F)$. For every $(g_1,g_2)\in G$, we have $H(g_1,g_2)=H(\iota(g_2)^{-1}g_1,1)\in H\backslash G$. This induces a bijection $H\backslash G\simeq \GL_{n+1}(F)$, given by $H(g_1,g_2) \mapsto \iota(g_2)^{-1}g_1$. The natural right action of $G$ on the quotient space $H\backslash G$ corresponds to the following right action on $\GL_{n+1}(F)$:
$$\forall (g_1,g_2)\in G, \, \forall g \in \GL_{n+1}(F),\quad g\cdot  (g_1,g_2)= \iota(g_2)^{-1}g\,g_1 \in \GL_{n+1}(F).$$ This isomorphism of right $G$-spaces, induces the following bijection 
\begin{align*} H\backslash G \slash K &\simeq\{ g \cdot K \colon g \in \GL_{n+1}(F)\}\\
&\simeq \{ \iota(K_2)^{-1}g\,K_1  \colon g\in \GL_{n+1}(F) \}\\
&\simeq \iota(K_2)\backslash \GL_{n+1}(F)\slash K_1.
\end{align*}
Let $\B_1\subset \GL_{n+1}$ (resp. ${\B}_2\subset \GL_n$) be the Borel subgroup of upper triangular matrices. We denote their respective opposite Borel subgroups by ${\overline{\B}_1}$ and ${\overline{\B}_2}$. Consider the Borel subgroup ${\B^{\sim}}:=\B_1\times {\overline{\B}_2}\subset \G$. Let $\U_1$ (resp $\U_2$) be the unipotent radical of $\B_1$ (resp. $\B_2$) and $\T=\T_1\times \T_2\subset {\B^{\sim}}$ be the split maximal torus where $\T_1$ (resp. $\T_2$) is the split maximal diagonal torus of $\GL_{n+1}$ (resp. $\GL_n$). 
The set of ${\B^{\sim}}$-antidominant diagonals is ${T^{\sim}}={T_1^-}\times {T_2^+}$, where 
$${T_1^-}= \left\{\diag(\varpi^{a_k})_{1\le k \le n+1},\colon a_i \in \Z \text{ such that } a_1\ge \dots \ge a_{n+1}\right\} \cdot\T_1(\cO_F),$$
$${T_2^+}= \left\{\diag(\varpi^{b_k})_{1\le k \le n},\colon b_i \in \Z \text{ such that } b_1\le \dots \le b_{n}\right\}\cdot\T_2(\cO_F).$$
\begin{proposition}
Every class in $H\backslash G\slash K$ admits a representative of the form 
$$ \left(\NiceMatrixOptions{transparent}\begin{pmatrix}
     \varpi^{a_1}&  &  &1  \\
       &\ddots  &	& \vdots \\
       &  & \varpi^{a_n}&1 \\
     &&&\varpi^c
  \end{pmatrix} ,\NiceMatrixOptions{transparent}\begin{pmatrix}
 \varpi^{b_1}&  &     \\
       & \ddots&	  \\
       &  &\varpi^{b_n}
  \end{pmatrix}\right)\in G,
$$
for some $c\in \Z$ and $\left (\diag(\varpi^{a_k})_{1\le k \le n+1},\diag(\varpi^{b_k})_{1\le k \le n}\right) \in {T^{\sim}}$ (with $a_{n+1}:=0$).
\end{proposition}
\begin{proof}Let us first prove that the set
$$\left\{\varpi^c \NiceMatrixOptions{transparent}\begin{pmatrix}
     \varpi^{a_1}&  &  &\varpi^{b_1}  \\
       &\ddots  &	& \vdots \\
       &  & \varpi^{a_n}&\varpi^{b_n} \\
     &&&1
  \end{pmatrix}\colon c \in \Z,a_n \ge b_n \text{ and } k<l \Rightarrow a_k-a_l \ge b_k-b_l \ge 0 \right\},$$
is a set of class representatives for the quotient $$\iota(K_2)\backslash \GL_{n+1}(F)\slash K_1.$$ 
Observe that $\iota(\GL_n)\times \chi_{1,n+1}(\Gm_m)$ is a Levi factor of the parabolic subgroup:
$$\P:=   \NiceMatrixOptions{transparent}\begin{pmatrix}
    \tikzmark{left}{\phantom{a}} &  &  &*  \\
       &  \GL_n&	& \vdots \\
       &  & \tikzmark{right}{\phantom{b}}&* \\
     &&&\Gm_m
  \end{pmatrix}
  \Highlight[first] = \NiceMatrixOptions{transparent}\begin{pmatrix}
    \tikzmark{left}{\phantom{a}} &  &  & \\
       &  \GL_n&	&  \\
       &  & \tikzmark{right}{\phantom{b}}&\\
     &&&\Gm_m
  \end{pmatrix}
  \Highlight[first] \ltimes 
  \NiceMatrixOptions{transparent}\begin{pmatrix}
1 &  &  &*  \\
       &  \ddots&	& \vdots \\
       &  & &* \\
     &&&1
  \end{pmatrix},
$$
where, the second factor of the Levi decomposition above is the unipotent radical $\U_\P$ of $\P$. Therefore, every class in $\iota(K_2)\backslash \GL_{n+1}(F)\slash K_1$ has a representatives of the form $\varpi^c \iota(g) u $ for some $c\in \Z$, $g\in \GL_n(F)$ and $u\in \U_\P(F)$. By the Cartan decomposition for $\GL_n(F)$, we know there exist $k, k' \in K_2$ such that $k' g k=\diag(\varpi^{a_k})_{1\le k \le n}$ with $a_1\ge \dots \ge a_n$. Hence
$$\iota(K_2)\varpi^{c}gu K_1=\iota(K_2)\varpi^{c}\iota(k'gk)  \iota(k)^{-1}u\, \iota(k) K_1.
$$
Since $\iota(k)^{-1}u \,\iota(k)\in \U_\P(F)$, then each class in $\iota(K_2)\backslash \GL_{n+1}(F)\slash K_1$ has a representative of the form
$$\varpi^c \NiceMatrixOptions{transparent}\begin{pmatrix}
     \varpi^{a_1}&  &  &\varpi^{b_1} d_1  \\
       &\ddots  &	& \vdots \\
       &  & \varpi^{a_n}&\varpi^{b_n} d_n\\
     &&&1
  \end{pmatrix}$$ For some $c, a_k, b_k \in \Z$ such that $k<l \Rightarrow a_k-a_l \ge 0$ and, $d_k \in \O_F$.
\begin{itemize}[topsep=0pt]
\item $\underline{d_k\neq 0}:$ Suppose some of the $d_k$'s is zero. Using a unipotent matrix from the right we can replace $\text{Column}_{n+1}$ by $\text{Column}_{n+1}+\text{Column}_{k}$. Thus we get a new matrix with $d_k=1$ and $b_k$ is equal to $a_k$.
\item $\underline{d_k=1}:$ Suppose then that all $d_k$'s are non zero. Conjugation by the matrix $\diag(d_1, \dots, d_n,1)\in \iota(K_2)$ shows that one can ignore the $d_k$'s:
  $$\varpi^c \NiceMatrixOptions{transparent}\begin{pmatrix}
     \varpi^{a_1}&  &  &\varpi^{b_1}\\
       &\ddots  &	& \vdots \\
       &  & \varpi^{a_n}&\varpi^{b_n}\\
     &&&1
  \end{pmatrix}$$
\item $\underline{a_n\ge b_n}:$ If we have $a_n< b_n$, we may change $b_n$ to $a_n$ by replacing the $\text{Column}_{n+1}$ by $\text{Column}_{n+1}+\text{Column}_{n}$.
\item (*) $\underline{b_1\ge\dots\ge b_n}:$ Let $k<l$ and suppose that $b_k <b_l$. 
We can find a matrix in $\iota(K_2)$ taking $\text{Row}_{l}$ to $\text{Row}_{l}+ \text{Row}_k$, thus we get in the $k^{th}$ row $\varpi^{b_k}+\varpi^{b_l}=\varpi^{b_k}(1+\varpi^{b_l-b_k})$. We kill the invertible element $1+\varpi^{b_k-b_l}\in \cO_F^\times$ by adjoint action of a diagonal matrix in $\iota(K_2)$ having $(1+\varpi^{b_l-b_k})^{-1}$ in the $l^{th}$ component. Now we kill the element that appears in the $(l,k)$ position of the new matrix: use the action of $K_1$ on columns to take $\text{Column}_{k}$ to $\text{Column}_{k}-\varpi^{a_k-a_l}(1+\varpi^{b_l-b_k})^{-1}\text{Column}_{l}$. We thus have changed the matrix by replacing -only- the old $b_l$ by $b_k$. 
\item $\underline{a_k-a_l \ge b_k-b_l}:$ Fix a $k\le n$. Suppose that for some $l>k$ we have $c_{k,l}:=(b_k-b_l)-(a_k-a_l) > 0$. Let $l^{\prime}$ be such that $c_{k,l^\prime}=\max_{l\ge k}c_{k,l}$. Take $\text{Row}_{k}$ to $$\text{Row}_{k}+\varpi^{a_k-a_{l^\prime}}\text{Row}_{l^\prime}=\text{Row}_{k}+\varpi^{b_k-b_{l^\prime}-c_{k,l^\prime}}\text{Row}_{l^\prime},\quad (a_k-a_{l^\prime}\ge 0).$$
 Note that for every $l \ge k$ we have
$$(b_k-c_{k,l^\prime}-b_l)-(a_k-a_l)\le 0.$$
We clear the component appearing in the $(k,l)$ position by replacing $\text{Column}_{l^\prime}$ by $\text{Column}_{l^\prime}-\text{Column}_k$. Therefore, we may\footnote{The last coefficient of the $k$-th row is $\varpi^b_k + \varpi^{b_{l'} + a_k - a_{l'}}$
which is not of the form $\varpi^{b'_k}$. Writing this as $\varpi^{b_{l'} + a_k - a_{l'}} u$ with $u$ invertible, one can then use yet another conjugation to get it to $\varpi^{b'_k}$ with $b'_k = b_{l'} + a_k - a_{l'} =  b_k - c_{k,l'}$.} replace in the old matrix (only) $b_k$ by $b_k-c_{k,l^\prime}$. By doing so, we do not alter the previous required inequalities.

Suppose there is an $l\ge k$ such that the new $b_k$ is  $< b_l$. Using the step (*) above, we saw that we can replace $b_l$ by $b_k$. By doing so, we do not alter the required inequality $(b_k-b_l)-(a_k-a_l)\le-(a_k-a_l)\le 0$.
\end{itemize}
Now, every class $\iota(K_2) \varpi^c \NiceMatrixOptions{transparent}\begin{pmatrix}
     \varpi^{a_1}&  &  &\varpi^{b_1}\\
       &\ddots  &	& \vdots \\
       &  & \varpi^{a_n}&\varpi^{b_n}\\
     &&&1
  \end{pmatrix}K_1$ corresponds to $$H\left(\NiceMatrixOptions{transparent}\begin{pmatrix}
     \varpi^{a_1-b_1}&  &  &1  \\
       &\ddots  &	& \vdots \\
       &  & \varpi^{a_n-b_n}&1 \\
     &&&\varpi^c
  \end{pmatrix} ,\NiceMatrixOptions{transparent}\begin{pmatrix}
 \varpi^{-b_1-c}&  &     \\
       & \ddots&	  \\
       &  &\varpi^{-b_n-c}
\end{pmatrix}\right)K \in  H\backslash G/K,$$
such that $-b_1\le \dots \le -b_n$ and $a_1-b_1 \ge \dots \ge a_n-b_n\ge 0$. This completes the proof of the proposition.
\end{proof}
\begin{lemma}\label{stab1}
For $\star \in \{1, 2\}$, let
$$ g_\star=\left(\NiceMatrixOptions{transparent}\begin{pmatrix}
     \varpi^{a_{1,\star}}&  &  &1  \\
       &\ddots  &	& \vdots \\
       &  & \varpi^{a_{n,\star}}&1 \\
     &&&\varpi^{c_{\star}}
  \end{pmatrix} ,\NiceMatrixOptions{transparent}\begin{pmatrix}
 \varpi^{b_{1,\star}}&  &     \\
       & \ddots&	  \\
       &  &\varpi^{b_{n,\star}}
  \end{pmatrix}\right),
$$
for some $a_{k,\star},b_{k,\star},c_\star\in \Z$ such that for $\star\in\{1,2\}$ the sequences $(a_{k,\star})_k$ is non-increasing and $(b_{k,\star})_k$ is non-decreasing.
If $Hg_1K=Hg_2K$ then $c_1=c_2$ and ${a_{k,1}-b_{k,1}=a_{k,2}-b_{k,2}}$ for each $1\le k \le n$.
\end{lemma}
\begin{proof}Suppose that the classes of $g_1$ and $g_2$ are the same in $H\backslash G\slash K$. This is true if and only if there exists $k_2\in K_2$ and $k_1\in K_1$ such that
$$ \iota(k_2)^{-1}\NiceMatrixOptions{transparent}\begin{pmatrix}
     \varpi^{a_{1,1}-b_{1,1}}&  &  &\varpi^{-b_{1,1}}  \\
       &\ddots  &	& \vdots \\
       &  & \varpi^{a_{n,1}-b_{n,1}}&\varpi^{-b_{n,1}} \\
     &&&\varpi^{c_{1}} 
  \end{pmatrix}k_1=\NiceMatrixOptions{transparent}\begin{pmatrix}
     \varpi^{a_{1,2}-b_{1,2}}&  &  &\varpi^{-b_{1,2}}  \\
       &\ddots  &	& \vdots \\
       &  & \varpi^{a_{n,2}-b_{n,2}}&\varpi^{-b_{n,2}} \\
     &&&\varpi^{c_{2}} 
  \end{pmatrix}
$$
It is clear that $k_1\in K_1\cap \P(F)$, i.e. 
 $$k_1= \NiceMatrixOptions{transparent}\begin{pmatrix}
    \tikzmark{left}{\phantom{a}} &  &  &u_1  \\
       &  \,k_1'\,&	& \vdots \\
       &  & \tikzmark{right}{\phantom{b}}&u_n \\
     &&&u_{n+1}
  \end{pmatrix}
  \Highlight[first].$$
Here, $k_1'$ belongs to $K_2$ and $u_{n+1}\in \O_F^\times$. The equality above implies $\varpi^{c_2}= u_{n+1} \varpi^{c_1}$ (i.e., $c_1=c_2$) and 
$$k_2 \NiceMatrixOptions{transparent}\begin{pmatrix}
     \varpi^{a_{1,1}-b_{1,1}}&  &  \\
       &\ddots  &	  \\
       &  & \varpi^{a_{n,1}-b_{n,1}}\\
  \end{pmatrix} k_1'=  \NiceMatrixOptions{transparent}\begin{pmatrix}
     \varpi^{a_{1,2}-b_{1,2}}&  &  \\
       &\ddots  &	  \\
       &  & \varpi^{a_{n,2}-b_{n,2}}\\
\end{pmatrix}$$
Since both diagonal matrices are ${\overline{\B}_2}$-antidominant and represent the same class in $K_2\backslash \GL_n(F)/K_2$ they must be equal.
\end{proof}
\begin{lemma}\label{Hstab}
For any integers $c,a_1,\dots, a_n, b_1,\dots b_n$, consider the matrix
$$g=\left(\varpi^c\NiceMatrixOptions{transparent}\begin{pmatrix}
     \varpi^{a_{1}}&  &  &1  \\
       &\ddots  &	& \vdots \\
       &  & \varpi^{a_{n}}&1 \\
     &&&1
  \end{pmatrix} ,\NiceMatrixOptions{transparent}\begin{pmatrix}
 \varpi^{b_{1}}&  &     \\
       & \ddots&	  \\
       &  &\varpi^{b_{n}}
  \end{pmatrix}\right) \in G.$$
The stabilizer $\Stab_{H}(gK)$ is equal to
\begin{equation*}\begin{split}\big\{\Delta((h_{ij}))\in H\colon h_{ii}\in \O_F,&
 h_{ij}\in \varpi^{\max\{a_i-a_j,b_i-b_j\}}\O_F \forall i \neq j,\\
 &\sum_{j=1}^n h_{ij} \in 1+ \varpi^{a_i} \O_F,  1\le\forall i \le n \big\}.\end{split}\end{equation*}
\end{lemma}
\begin{proof}Set $\diag(a)=\diag(\varpi^{a_1},\dots, \varpi^{a_n},1)$ and similarly $\diag(b)=\diag(\varpi^{b_1},\dots, \varpi^{b_n})$. A matrix $h=(h_{ij})\in \GL_n(F)$ verifies $\Delta(h) \in \Stab_{H}(gK)$ by definition if and only if
$$\begin{cases}\diag(-a)u_0^{-1} \iota(h) u_0\diag(a)K_1&= K_1\\
\diag(-b)h\diag(b)K_2 &=K_2,
\end{cases}$$
which is equivalent to  
\begin{equation*}\begin{split}\NiceMatrixOptions{transparent}\begin{pmatrix}
    \tikzmark{left}{\phantom{a}} &  &  &\varpi^{-a_1}(-1+\sum_{j=1}^{n}h_{1j})  \\
       &  \,\diag(-a)h\diag(a)\,&	& \vdots \\
       &  & \tikzmark{right}{\phantom{b}}&\varpi^{-a_n}(-1+\sum_{j=1}^{n}h_{nj}) \\
     &&&1
  \end{pmatrix}
  \Highlight[first]\in K_1\text{ and } \diag(-b)h\diag(b)\in K_2.\end{split}\end{equation*}
Therefore $h\in \diag(a)K_2 \diag(-a) \cap \diag(b)K_2 \diag(-b)$ and for all $1\le i\le n$ we must have $\sum_{j=1}^n h_{ij} \in 1+ \varpi^{a_i} \O_F$. This proves that $\Stab_{H}(gK)$ is equal to
\begin{align*}\{\Delta((h_{ij}))\in H\colon h_{ii}\in \O_F,\, h_{ij}\in \varpi^{c_{ij}}\O_F\text{ and }\sum_{j=1}^n h_{ij} \in 1+ \varpi^{a_i} \O_F,1\le  \forall  i\le n\}.
\end{align*}
where, $c_{ij}:=\max\{a_i-a_j,b_i-b_j\}$ for $i \neq j$.
\end{proof}
In the following proposition we compute the determinant of $H$-stabilizers of cosets in $G/K$.
\begin{proposition}\label{determinantstab}Let $c,a_1,\dots, a_n, b_1,\dots b_n$ be any integers and consider the element $$g=\left(\varpi^c\NiceMatrixOptions{transparent}\begin{pmatrix}
\varpi^{a_{1}}&  &  &1  \\
       &\ddots  &	& \vdots \\
       &  & \varpi^{a_{n}}&1 \\
     &&&1
\end{pmatrix} ,\NiceMatrixOptions{transparent}\begin{pmatrix}
\varpi^{b_{1}}&  &     \\
       & \ddots&	  \\
       &  &\varpi^{b_{n}}
\end{pmatrix}\right) \in G.$$
Set $c_{ij}:=\max\{a_i-a_j,b_i-b_j\}$ if $i\neq j$. We then have $\det(\Stab_{H}(gK))=$
$$\begin{cases} \O_F^\times, \text{ if there exists $i$ such that } a_{i} \le 0,\\
\O_F^\times, \text{ if there exists a pair $(i,j)$ such that $i\neq j$ and } c_{ij} \le 0,\\
1+\varpi^{\min(\{a_i:  1\le i\le n \} \cup\{ c_{ij}:  1\le i\neq j \le n\})}\O_F \text{ if } c_{i,j}>0 \text{ for all $i\neq j$ and } a_i  > 0.
\end{cases}$$
\end{proposition}
\begin{proof}
For the first two cases, for each $t \in  \O_F^\times$ one can give explicitly an element $g_t\in H$ stabilizing $gK$ with determinant $t$:
\begin{itemize}
\item Let $i$ be an index for which $a_i\le0$. One can take $g_t = \diag(1, \dots, 1, t, 1, \dots, 1)\in K_2$, with $t$ in the $i^{th}$ position.
\item Let $i$ and $j$ be indices, for which $c_{ij} \le 0$. One can take $g_t$ to be the matrix with $t$ in the $i^{th}$ diagonal component, $1$'s elsewhere in the diagonal, $1-t$ in the $ij$ component and zeros everywhere else.
\end{itemize}
Suppose now, that for all $1\le i,j\le n$ we have $c_{ij}>0$ and, $ a_i  > 0$. In this case a matrix $h=(h_{ij})$ belongs to $\Stab_{H}(gK)$ only if $$h_{ij}\in \varpi^{c_{ij}} \O_F, \forall 1\le i\neq j\le n \text{ and } h_{ii} \in 1+\varpi^{\min\{a_i, c_{ij}\colon  1\le j \le n, j\neq i\}}\O_F,\forall  i\le n.$$
Therefore, for all $h=(h_{ij}) \in \Stab_{H}(gK) $ we have $\det(h)\in 1+\varpi^{\min\{a_i, c_{ij}\colon  1\le i \neq j \le n\}}\O_F$, i.e. 
$$\det(\Stab_{H}(gK))\subset 1+\varpi^{\min\{a_i, c_{ij}\colon  1\le i \neq j \le n\}}\O_F.$$
The reverse inclusion is as follows: If $\min\{a_i, c_{ij}\colon  1\le i \neq j \le n\}=a_k$ for some $k\le n$, then consider for every $t\in \cO_F$ the matrix 
$h_t=\diag(1,\dots, 1+\varpi^{a_k}t,\dots,1)\in \GL_n(F)$. If now $\min\{a_i, c_{ij}\colon  1\le i \neq j \le n\}=c_{kl}$ for some $1\le k\neq l \le n$, then consider for every $t\in \cO_F$ the matrix $h_t$ having $1+\varpi^{c_{kl}}t$ in the $k^{th}$ diagonal component, $1$'s elsewhere in the diagonal, $-\varpi^{c_{kl}}t$ in the $kl$ component and zeros everywhere else. In both cases, for every $t\in \cO_F$, we have $h_t\in \Stab_H(gK)$ and $\det (h_t)= 1+\varpi^{\min\{a_i, c_{ij}\colon  1\le i \neq j \le n\}}t$, which shows the reverse inclusion.
\end{proof}

\subsection{Local horizontal relation}\label{localhorizontal}
We prove a local horizontal distribution relation (Corollary \ref{relhor2}) using a local congruence relation (Theorem \ref{Hordist}).
\subsubsection*{Notation}
Let $\B={\B}_1\times \B_2\subset \G$ be the Borel subgroup with unipotent radical ${\U}={ \U}_1\times {\U_2}$. The set of ${\B}$-antidominant diagonals is ${T^{-}}={T_1^-}\times {T_2^-}$, where 
$${T_1^-}= \left\{\diag(\varpi^{a_k})_{1\le k \le n+1},\colon a_i \in \Z \text{ such that } a_1\ge \dots \ge a_{n+1}\right\}\cdot \T_1(\cO_F),$$
$${T_2^-}= \left\{\diag(\varpi^{b_k})_{1\le k \le n},\colon b_i \in \Z \text{ such that } b_1\ge \dots \ge b_{n}\right\}\cdot \T_2(\cO_F).$$
Consider the following Iwahori subgroup $$I={I}_1\times I_2=\{g \in K\colon (g \mod \varpi) \in B\}.$$
Let ${\mu}= \Delta \circ \mu_2 \in T$ be the cocharacter coming from the Shimura variety and denoted by $\mu_{h,v}$ in \S \ref{localpair}. Here $\mu_2 \in X_*(\T_2)$ is given by $t\mapsto \diag(t,1\dots,1)$. Set  $\textbf{{Frob}}:=\mu(\varpi)$, see \S \ref{subsec-frob} for the reason why we choose this notation. Let $\cU_{\mu}\in \End_{\Z[B]}\Z[G/K]$ be the $\mathbb{U}$-operator associated to ${\mathbf{1}_{I\mu(\varpi)I}\in \cC_c(G\sslash I,\Z)}$. 
By \cite[Theorem \ref{uoprootheckeseed}]{Seedrelations2020}, the operator $\cU_{\mu}$ is annihilated by the Hecke polynomial $H_w(X)=\sum_{k=0}^{n(n+1)} A_k X^k \in (\text{End}_{R[G]}R[G / K ])[X]$. 

Fix the classe $[1]=1\cdot K \in  \Z[G/K]$. By \cite[Lemma \ref{Uaction}]{UoperatorsII2021}, we have for every $k\ge 1$
$$\cU_{\mu}^k ([1])=\cU_{\mu^k} ([1])= \sum_{ h \in  I^{+}/ \mu(\varpi^{k}) I^{+} \mu(\varpi^{-k})} h \,\textbf{{Frob}}^k \cdot [1], \text{ where, $I^{+}={U}\cap I$}.$$ 
We fix for each element $[b]$ in $\cO_F/\varpi \cO_F$ a lift $b \in \cO_F$ (e.g. Teichmuller lift). Set $S_k:=\{a=\sum_i a_i \varpi^i \in \cO_F\colon a_i=0 \text{ for all } i\ge k \text{ and } [a_i] \in \cO_F/\varpi \cO_F \text{ for all } i< k\}$.  
\begin{lemma}\label{cosrep2} For every integer $k \ge 1$, the collection 
$$\left\{(u_{k,a},v_{k,b})=\left(\NiceMatrixOptions{transparent}\begin{pmatrix}
     1& a_1  & \cdots  & a_n\\
       &1  &	&  \\
       &  & \ddots&\\
     &&&1
  \end{pmatrix},\NiceMatrixOptions{transparent}\begin{pmatrix}
   1&b_1& \cdots &b_{n-1}   \\
       &1&  	&  \\
       &&\ddots&\\
    &&&1
  \end{pmatrix}\right)\right\}$$
for all ${(a,b)=\left((a_1, \dots, a_n),(b_1 ,\dots ,b_{n-1})\right)\in S_k^n\times S_k^{n-1}}$ forms a complete set of representatives for $I^{ +}/\mu(\varpi^k) I^+ \mu(\varpi^k)^{-1}$. Consequently, for every $k\ge 1$
$$\cU_{\mu}^k([1]) = \sum_{ (a,b) \in  S_k^n\times S_k^{n-1}} (u_{k,a},v_{k,b}) \,\textbf{{Frob}}^k \cdot [1].$$
\end{lemma}
\begin{proof}
This is a consequence of the fact that ${\mu}(\varpi^k) I^+ {\mu}(\varpi^{-k})$ equals
$$\NiceMatrixOptions{transparent}\begin{pmatrix}
       1& \varpi^k \cO_F& \cdots &\cdots&   \varpi^k  \cO_F \\
        &    1     &   \cO_F   &\cdots&\cO_F	  \\
       &         &      \ddots      &\ddots& \vdots \\
       &         &            &\ddots& \cO_F \\
      & &&&1 \\
\end{pmatrix} \times \NiceMatrixOptions{transparent}\begin{pmatrix}
       1& \varpi^k \cO_F& \cdots &\cdots&   \varpi^k  \cO_F \\
        &    1     &   \cO_F   &\cdots&\cO_F	  \\
       &         &      \ddots      &\ddots& \vdots \\
       &         &            &\ddots& \cO_F \\
      & &&&1 \\
\end{pmatrix} 
\subset  I^{+}.$$
Any $g\in I_1^+$, can be written as
$$g = \NiceMatrixOptions{transparent}\begin{pmatrix}
   1 & a_1 & \cdots & a_n \\
      &\tikzmark{left}{\phantom{a}}   &	&  \\
       &  &h & \\
    &&&\tikzmark{right}{\phantom{b}}
  \end{pmatrix}
  \Highlight[first],            
$$
for some $h\in K_1$ (which is upper triangular modulo $\varpi$) and $(a_i)\in \cO_F^n$. Therefore
$$g{\mu}(\varpi^k) I_1^+ {\mu}(\varpi^{-k})=g\iota(h^{-1}){\mu}(\varpi^k) I_1^+ {\mu}(\varpi^{-k})=u_{a^\prime,k}{\mu}(\varpi^k) I_1^+ {\mu}(\varpi^{-k})
$$
with some $a^\prime=(a_i^\prime)\in \cO_F^n$. The right action of $ {\mu}(\varpi^k) I_1^{ +} {\mu}(\varpi^{-k})$ on $I_1^+$ can only kill $\varpi^k\cO_F$ in each factor $a_i^\prime$. This shows that each class admits a representative of the form $u_{k,a^\prime}$ that is unique modulo $\varpi^k$. Similar statement hold for $I_2^+$. In total, this shows that is $S_k^n\times S_k^{n-1}$ is a complete set of representatives for the quotient $ {\mu}(\varpi^k)I^{ +} {\mu}(\varpi^{-k})$.
\end{proof}
\subsubsection{Divisibility in \texorpdfstring{$\Z[H_0\backslash G/K]$}{Z[H0G/K]} and proof of Theorem \ref{divisibilityheckepol}}
Consider now the following natural surjective homomorphisms of $\Z$-modules over the group algebra ${\cH(G\sslash K)[H]}$\footnote{The $H\times \cH(G\sslash K)$-equivariance is seen from the identification $\cH(G\sslash K)\simeq  \End_{\Z[G]}\Z[G\sslash K]$.}
$$\begin{tikzcd}
 \Z[G/K] \arrow[two heads]{r}{\phi}[swap]{}
\arrow[bend left, two heads]{rr}{\phi_0}
&\Z[H^{\der}\backslash G/K] \arrow[two heads]{r}{}[swap]{}&
\Z[H_0\backslash G/K]
\end{tikzcd}$$
where, $H^{\der} := \Hbf^{\der}(F)=\Delta(\SL_n)(F)$ and $ H_0\subset H$ is the normal subgroup $\det^{-1}(\cO_F^\times)\supset H^{\der}$. We derive the local horizontal distribution relations (Corollary \ref{relhor2}) from the following result 
\begin{theorem}[Local congruence relations]\label{Hordist}
We have
$${H}_{w}(\textbf{{Frob}})\cdot \phi_0([1])\equiv 0 \mod q^{n-1}(q-1) R[H_0\backslash G/K],  \quad\quad(n\ge 1).$$
\end{theorem}
\begin{proof}
Recall that $\cU_\mu \in \End_{\Z[B]}\Z[G/ K]$, ${H}_{w}(\cU_\mu)=0$ and $\textbf{{Frob}}\in T\subset B'$. Set $\widetilde{H}_{w}(X):=H_w(q^{n-1} X)$. We have, 
\begin{align*}\widetilde{H}_{w}(\textbf{{Frob}})\cdot \phi_0([1])&=\phi_0\left((\widetilde{H}_{w}(\textbf{{Frob}})-{H}_{w}(\cU_\mu))\cdot [1]\right)\\
&= \phi_0\big(\sum_{k=0}^{n(n+1)} A_k (q^{k(n-1)}\textbf{Frob}^k- \cU_\mu^k)\cdot [1]\big)\\
&= \sum_{k=0}^{n(n+1)} A_k \phi_0\left((q^{k(n-1)}\textbf{Frob}^k- \cU_\mu^k)\cdot [1]\right)
\end{align*}
where, we have used the $\G$-equavariance of $\text{since }A_k\in \End_{R[G]}R[G/K]$ for the last equality. 
In order to prove Theorem \ref{Hordist}, it is sufficient to show the following lemma.\end{proof} 
\begin{lemma}\label{divisibilitylemma}
For any integer $k\ge 1$ we have
$$\phi_0\left((\cU_\mu^k-q^{k(n-1)}\textbf{Frob}^k)\cdot [1]\right) \equiv 0 \mod q^{k(n-1)}(q-1) \Z[H_0\backslash G/K].$$
\end{lemma}
\proof
In steps A,B and C of the proof, we will be working only mod $H^{\der}$ and we will wait until step D to project our calculations mod $H_0$.
\begin{enumerate}[wide=0pt]
\item[A.]
We have,
\begin{align*}\phi(\cU_\mu^k([1])) 
&=\sum_{{ (a,b) \in  S_k^n\times S_k^{n-1}}} \phi( (u_{k,a},v_{k,b}) \,\textbf{{Frob}}^k \cdot [1])\quad \quad (\text{Lemma }\ref{cosrep2})\\
 &= \sum_{\substack{ (a,b) \in S_k^n\times S_k^{n-1}}} \phi( (\iota(v_{k,b})^{-1}u_{k,a},1_n) \,\textbf{{Frob}}^k \cdot [1]) \quad\quad (\Delta(v_{k,b}) \in H^{\der})\\
&= \sum_{{ (a,b) \in  S_k^n\times S_k^{n-1}}} \phi( (u_{k,a-b},1_n) \,\textbf{{Frob}}^k \cdot [1]), 
\end{align*}
where, $a-b:= (a_1-b_1, \dots ,a_{n-1}-b_{n-1}, a_n)\in S_k^{n}$ and ${1_n=\diag(1,\dots,1)\in \GL_n(F)}$. By substituting $c=(a-b)\in S_k^n$\footnote{We sum over $c=(a-b)$'s. For every fixed $c$ there is $q^{k(n-1)}=|(\cO_F/\varpi^k\cO_F)^{n-1}|$ choice of pairs $(a,b)\in(\cO_F/\varpi^k\cO_F)^{n}\times (\cO_F/\varpi^k\cO_F)^{n-1}$ such that $c=a-b= (a_1-b_1, \dots ,a_{n-1}-b_{n-1}, a_n)$.}, we get
\begin{align*}\phi(\cU_\mu^k([1]))=q^{k(n-1)} \textbf{{Frob}}^k \cdot \phi([1])+ q^{k(n-1)}\sum_{{ c \in   S_k^n \setminus \{0_n\}}} \phi( (u_{k,c},1_n) \,\textbf{{Frob}}^k \cdot [1]),
\end{align*}
where the first term in the right-hand-side is the contribution for $c=a-b=0$ i.e. ($a_i=b_i$ for $1\le i \le n-1$ and $a_n=0$).
Therefore, 
$$\maltese:= \phi( (\cU_\mu^k-q^{k(n-1)} \textbf{{Frob}})\cdot [1])=q^{k(n-1)}\sum_{{ c \in  S_k^n \setminus \{0_n\}}} \phi( (u_{k,c},1_n) \,\textbf{{Frob}}^k \cdot [1])$$
\item[B.]
Define the following map $$\varepsilon\colon S_k^n \to (\varpi^\N)^{n}$$ sending $c=(c_1,\dots,c_n)\in S_k^n$ to  $\underline{\varepsilon}_c:=(\varepsilon(c_1),\dots,\varepsilon(c_n))$ such that
$\varepsilon(c_i):= \varpi^{{\text{ord}}_F(c_i)}$ with $\varepsilon(0)=0$  (note that ${\text{ord}}_F(c_i) \le {k-1}$ and that $\varepsilon(c)=0$ if and only if $\underline{\varepsilon}_c=0_n:=(0,\dots,0)$). 

Put $\mathcal{E}:= \varepsilon(S_k^n \setminus \{0_n\})$. Note that we can view any $n$-tuple $\underline{\varepsilon}\in\mathcal{E}$ as a $n$-tuple $\underline{\varepsilon}\in S_k^n$. Set 
$$\tilde{x}:=\begin{cases} x/\varpi^{{\text{ord}}_F(x)}\in  \cO_F^\times \text{ if } x\in \cO_F \setminus\{ 0\}, \\ 1 \text{ if } x=0.\end{cases}$$ 
For every $c=(c_i)\in  S_k^n \setminus \{0_n\}$ consider the following two matrices
$$\overline{c}:=\diag\big(\tilde{c}_n^{n-1}\prod_{i=1}^{n-1}\tilde c_i^{-1},\tilde c_n^{-1}\tilde c_1,\dots,\tilde c_n^{-1} \tilde c_i,\dots,\tilde c_n^{-1} \tilde c_{n-1}\big)\in K_2\cap \SL_n(F),$$
and
$$\underline{c}:=\diag\left(\tilde c_n,\tilde c_n\tilde c_1^{-1},\dots ,\tilde c_n\tilde c_i^{-1},\dots,\tilde c_n\tilde c_{n-1}^{-1},1\right)\in K_1.$$
Therefore, we have
\begin{align*}\phi\left((u_{k,c},1_n) \,\textbf{{Frob}}^k \cdot [1]\right)&= \phi\left(\Delta(\overline{c})(u_{k,c},1_n) \,\textbf{{Frob}}^k \cdot [1]\right)&(\text{since }\Delta(\overline{c})\in H^{\der})\\
&=\phi\left((\iota(\overline{c})u_{k,c}\underline{c},\overline{c}) \,\textbf{{Frob}}^k \cdot [1]\right)&(\text{since $\underline{c}\in K_1\cap T_1$})\\
&=\phi\left((\iota(\overline{c})u_{k,c}\underline{c},1_n) \,\textbf{{Frob}}^k \cdot [1]\right)&(\text{since $\overline{c}\in K_2\cap T_2$})
\end{align*}
A direct calculation yields the equality
$$\iota(\overline{c}) u_{k,c}\underline{c}=\diag\big(\tilde c_n^{n}\prod_{i=1}^{n-1}\tilde c_i^{-1},1,\dots,1\big)u_{k,\underline{\varepsilon}_c}\in \GL_{n+1}(F).$$
For every $c\in S_k^n \setminus\{0_n\}$, set $$\alpha(c):=(\tilde c_n^{n}\prod_{i=1}^{n-1}\tilde c_i^{-1} \mod \varpi^k \cO_F) \in (\cO_F/\varpi^k\cO_F)^\times.$$The above equality shows that for every $c=(c_i),c'=(c_i')\in S_k^n\setminus\{0_n\}$:
$$\alpha(c)=\alpha(c')\text{ and } {\varepsilon}(c)={\varepsilon}(c') \Rightarrow \phi( (u_{k,c},1_n) \,\textbf{{Frob}}^k \cdot [1])=\phi( (u_{k,c'},1_n) \,\textbf{{Frob}}^k \cdot [1]).$$
\item[C.] Using this we continue the computation of $\maltese$ by regrouping terms over $c\in S_k^n \setminus\{0_n\}$ giving the same values by $\alpha$ and $\varepsilon$.
\begin{align*}\maltese&=q^{k(n-1)}\sum_{{ c \in  S_k^n \setminus \{0_n\}}} \phi( (u_{k,c},1_n) \,\textbf{{Frob}}^k \cdot [1])\\
&=q^{k(n-1)}\sum_{\underline{\varepsilon}\in \mathcal{E}} \sum_{\beta\in S_k^\times}\sum_{\{c\in S_k^n \colon\alpha(c)=\beta,\,  \varepsilon(c)=\underline{\varepsilon}\}} \phi\left( (\diag\left(\beta,1,\dots,1\right)u_{k,\underline{\varepsilon}},1_n) \,\textbf{{Frob}}^k \cdot [1]\right)\\
&=q^{k(n-1)}\sum_{\underline{\varepsilon}\in \mathcal{E}} \sum_{\beta\in S_k^\times}|J(\underline{\varepsilon},\beta)|\,\phi\left( (\diag\left(\beta,1,\dots,1)u_{k,\underline{\varepsilon}},1_n\right) \,\textbf{{Frob}}^k \cdot [1]\right)
\end{align*}
where, $J(\underline{\varepsilon},\beta):=\{c\in S_k^n\setminus \{0_n\}\colon\alpha(c)=\beta,\,  \varepsilon(c)=\underline{\varepsilon}\}$. 
\item[D.]Now we project $\maltese\in \Z[H^{\der}\backslash G/K]$ into $ \Z[H_0\backslash G/K]$. This gives us
\begin{align*}\phi_0(\maltese)&=\phi_0( (\cU_\mu^k-q^{k(n-1)} \textbf{{Frob}})\cdot [1])\\
&=q^{k(n-1)}\sum_{\underline{\varepsilon}\in \mathcal{E}} \sum_{\beta\in S_k^\times}|J(\underline{\varepsilon},\beta)|\,\phi_0\left((\diag(\beta,1,\dots,1)u_{k,\underline{\varepsilon}},1_n) \textbf{{Frob}}^k \cdot [1]\right)\\
&=q^{k(n-1)}\sum_{\underline{\varepsilon}\in \mathcal{E}} \sum_{\beta\in S_k^\times}|J(\underline{\varepsilon},\beta)|\,\phi_0\left((u_{k,\underline{\varepsilon}},\diag(\beta^{-1},1,\dots,1)) \textbf{{Frob}}^k \cdot [1]\right)\\
&=q^{k(n-1)}\sum_{\underline{\varepsilon}\in \mathcal{E}} \big(\sum_{\beta\in S_k^\times}|J(\underline{\varepsilon},\beta)|\big)\,\phi_0\left((u_{k,\underline{\varepsilon}},1_n) \textbf{{Frob}}^k \cdot [1]\right)& 
\end{align*}
We observe that the sum $\sum_{\beta\in S_k^\times}|J(\underline{\varepsilon},\beta)|$ is equal to $$|\{c\in S_k^n \setminus \{0_n\} \colon {\varepsilon}(c)=\underline{\varepsilon}|.$$
Let $c=(c_i)$ such that ${\varepsilon}(c)=\underline{\varepsilon}$, which means that for every $i$, we must have ${\text{ord}}_F(c_i)={\text{ord}}_F(\varepsilon_i)$. It implies that for each $i$ such that $\varepsilon_i\neq 0$, the $\tilde c_i=c_i/\varpi^{{\text{ord}}_F \varepsilon_i}$ are such that ${\text{ord}}_F(\tilde c_i)\le k-{\text{ord}}_F \varepsilon_i$. Hence, by definition of the set $S_k$, the set defined by these elements is described as follows
$$\left\{\sum_{j\ge 0}^{j=k-{\text{ord}}_F \varepsilon_j} a_j \varpi^{j}\in \cO_F^\times \colon a_j=0 \text{ for all } j\ge k-{\text{ord}}_F \varepsilon_j \right\}.$$
This shows that there exist $q^{k-{\text{ord}}_F \varepsilon_i}-q^{k-{\text{ord}}_F \varepsilon_i-1}$ possible choices for $c_i$ for each such $i$ (with $\varepsilon_i\neq 0$). Therefore,
$$\sum_{\beta\in(\cO_F/\varpi^k\cO_F)^\times}|J(\underline{\varepsilon},\beta)|=\prod_{i}(q^{k-{\text{ord}}_F \varepsilon_i}-q^{k-1-{\text{ord}}_F \varepsilon_i}),$$
where the product is taken over the indices $1\le i \le n$ such that $\varepsilon_i\neq 0$. Since this set is nonempty for any $\underline{\varepsilon}\neq 0_n$, we deduce that 
$$\sum_{\beta\in(\cO_F/\varpi^k\cO_F)^\times}|J(\underline{\varepsilon},\beta)|\equiv  0 \mod (q -1),$$ and accordingly, for all $k\ge 1$,
\begin{align*}\phi_0\left( (\cU_\mu^k-q^{k(n-1)} \textbf{{Frob}})\cdot [1]\right)  &\equiv 0 \mod q^{k(n-1)}(q-1) R[H_0\backslash G/K].\qedhere \end{align*}
\end{enumerate}
\subsubsection{Local horizontal relation}
Recall that $H_c=\det^{-1}(1 + \varpi^{c} \cO_{F_v})$, for $c\in \N$.
\begin{corollary}\label{relhor2}
Set $x_0:=\phi([1])$. There exists $x\in R[H^{\der}\backslash G/K]^{H_1}$ such that
$$
{H}_{w}(\textbf{{Frob}})\cdot x_0=\Tr_{1,0} x \in R[H^{\der}\backslash G/K]^{H_0},
$$
where, $\Tr_{1,0} \phi(v):=\sum_{h\in H_0/H_1} h \cdot \phi(v)$.
\end{corollary}
\begin{proof}
The action of $H_0$ on $\Z[H^{\der}\backslash G/K]$ factors through $H_0/H^{\der}$ which we identify with $ \cO_F^\times$ through the determinant map.
We note that $${H}_{w}(\textbf{{Frob}})\cdot x_0\in R[H^{\der}\backslash G/K]^{H_0},$$
since, the induced action of $\cO_F^\times$ commutes with the induced action of ${H}_{w}(\textbf{{Frob}})$ and fixes $\phi([1])$, as 
$\det(K\cap H)=\det K_2=\cO_F^\times$. 

Write$${H}_{w}(\textbf{{Frob}})\cdot x_0 =\sum_{y\in H^{\der}\backslash G/K}a_y  y,$$
with only finitely many nonzero integral coefficients $a_y\in \Z$.
The stabilizer of any $y\in H^{\der}\backslash G/K$ in $H/H^{\der}\simeq F^\times$ is of the form 
$1 + \varpi^{c(y)} \cO_{F_v}$ for some integer $c(y)> 0$ that we call the conductor of $y$. We have,
\begin{align*}
{H}_{w}(\textbf{{Frob}})\cdot x_0 &= \sum_{y\in H^{\der}\backslash G/K}a_y  y\\
&=\sum_{c\ge 0} \sum_{\substack{y\in H^{\der}\backslash G/K\\
c(y)=c}} a_y y \\
&=\sum_{\substack{y\in H^{\der}\backslash G/K\\
c(y)=0}}{a_y} y+\sum_{c\ge 1} \sum_{\substack{H_0y\in H_0\backslash G/K\\
c(y)=c}} a_y \sum_{h\in H_0/H_c}h \cdot y
\end{align*}
In the third equality, we may sum up over classes $H_0y$ since ${H}_{w}(\textbf{{Frob}})\cdot x_0$ is $H_0$-invariant. In the last sum above, choose for each
$H_0$-orbit $H_0y$ with $c_(y)\ge1$ some $H_1$-orbit $H_1\tilde y \subset  H_0y$. Set
$$x:= \sum_{\substack{y\in H^{\der}\backslash G/K\\
c(y)=0}}\frac{a_y}{q-1} y+ \sum_{c\ge 1} \sum_{\substack{H_0y\in H_0\backslash G/K\\
c(y)=c}} a_y \sum_{h\in H_1/H_c}h \cdot \tilde y.
$$
By Theorem \ref{Hordist}, $(q-1)\mid a_y$ if $c(y)=0$, which gives $x\in R[H^{\der}\backslash G/K]$ with $\Tr_{1,0} x={H}_{w}(\textbf{{Frob}})\cdot x_0$.
\end{proof}
\begin{remark}
Denote the element $x$ constructed above by $x_v\nomenclature[G]{$x_v$}{The local cycle given in Corollary \ref{relhor2}}$ to keep track of its associated place $v\in \mathcal{P}_{sp}$. 
\end{remark}
\subsection{Global split distribution relations}
\subsubsection{Definition of the norm-compatible system}\label{constructioncycles}
For any $\mathfrak{f} \in \mathcal{N}_{sp}^r$, set $\mathcal{P}_\ff\subset \mathcal{P}_{sc}$ for the places of $F$ defined by the prime ideals dividing $\ff$. 
Define
$$\xi_{\mathfrak{f}}:=u(r)^{-1} \cdot  \pi_{\text{cyc}}\big([g_{0,S}]\otimes [1]^S(\mathfrak{f})\big)\in \Q[\cZ_{\G,K}(\H)]\nomenclature[G]{$\xi_{\mathfrak{f}}$}{The cycle $u(r)^{-1} \cdot  \pi_{\text{cyc}}\big([g_{0,S}]\otimes [1]^S(\mathfrak{f})\big)$}$$
where, $[g_{0,S}]:=\H^{\der}(F_S)g_{0,S}K_S$, $ [1]^S(\mathfrak{f}):= \left(\otimes_{v\in \mathcal{P}_\ff} x_v\right)\otimes \left(\otimes_{v \not\in S \cup \mathcal{P}_\ff}[1]_v\right)$ and
$$u(r)=\begin{cases}  [E^\times \cap \A_{F,f}^\times \mathfrak{O}_{1}^\times :F^\times] &\text{If } r=0,\\
1 & \text{If } r\ge 1.\end{cases}$$
\begin{proposition}\label{fieldofdefsplitcycles}
For each $\ff \in \mathcal{N}_{sp}$, the field of definition $E_{{\ff}}$ of $\xi_\ff$ is contained in $\cK(\ff)$ the $\cK$-transfer field of conductor $\ff$.
\end{proposition}
\proof

Write $\xi_\ff=\sum a_i \mathfrak{z}_{g_i}$ ($a_i\in \Q$) with 
$$g_{i,S}=g_{0,S},\, \forall v\not\in S\cup \mathcal{P}_\ff \quad  g_{i,v}=g_{0,v}$$
and $x_v=\sum_i a_i [g_i]_v$ for all $v \in \mathcal{P}_\ff$. The stabilizer of $\mathfrak z_{g_i}$ in $\H(\A_f)$ contains (as in Lemma \ref{fielddefinitiong0})
$$\H(\Q)\cdot({K_{\H,g_0,S}^Z}\times K_\H^{S\cup \mathcal{P}_\ff}\times \prod_{\p \in \mathcal{P}_\ff} K_{i,\p} ),$$
for some open compact subgroups $K_{i,\p}\subset K_{\p}$. Therefore, the stabilizer of $\xi_\ff$ in $\T^1(\A_{f})$ contains\footnote{Since by Corollary \ref{relhor2}, we have $[x_v]\in (\Z[q_v^{\pm 1}][H_v^{\der}\backslash G_v/K_v])^{U_v^1(1)}$.}
$$\T^1(\Q)\cdot(U_{g_0,S}\times U_\ff^{S\cup \mathcal{P}_\ff}\times \prod_{v \in \mathcal{P}_\ff} U_v^1(1))=\T^1(\Q)\cdot U_\ff,$$
and accordingly, $E_{{g_{\ff}}}$ is contained in $\cK(\ff)$ for which we have (see \S \ref{kappatransferfields})
\begin{align*}
\Gal(\cK(\ff)/E)&\simeq \frac{\T^1(\A_f)}{\T^1(\Q)U_\ff}\simeq \frac{ \A_{E,f}^\times}{E^\times \A_{F,f}^\times \mathfrak{O}_{\ff}^\times}.\qedhere
\end{align*}
\begin{remark}
Proposition \ref{fieldofdefsplitcycles} and Lemma \ref{transferkappa} imply that $E_{\ff}$ is also contained in $E(\fc_1 \ff)$. Therefore, by Corollary \ref{unramifiedprimesintransfer} any prime ideal of $E$ above an ideal $\p\in \mathcal{P}$ that is prime to $\ff$ is unramified in the extension $E_{\ff}/E$, i.e.
$$E_{\ff}\subset \cK(\ff) \subset E(\infty)^{un, w_\p}.\qedhere$$
\end{remark}

\subsubsection{Proof of Theorem \ref{Horizontal}}
For any $\mathfrak{f} \in \mathcal{N}_{sp}^r$ and any place ${v_\circ} \in \mathcal{P}_{sc}\setminus \mathcal{P}_\ff$ with prime ideal $\p_{v_\circ}\in \mathcal{P}_{sp}$, we show here that:
$$H_{{w_\circ}}(\Fr_{{w_\circ}}) \cdot \xi_{\mathfrak{f}}=\Tr_{\cK(\p_{v_\circ} \mathfrak{f})/\cK(\mathfrak{f})}\xi_{\p_{v_\circ} \mathfrak{f}},$$
where, $H_{{w_\circ}}\in \cH_{K_{\p_{v_\circ}}}(\Z[q_{v_\circ}^{\pm 1}])[X]$ is the Hecke polynomial attached to $\Sh_K(\G,\cX)$ at the place ${w_\circ}$ of the reflex field $E=E(\G,\cX)$ defined by $\iota_{v_\circ}$. Indeed, $H_{{w_\circ}}(\Fr_{w_{\circ}})\xi_{\ff}$ is equal to:
\begin{align*}\hspace{-2cm}
&=u(r)^{-1}  \pi_{\text{cyc}}\big([g_{0,s}]\otimes (H_{{w_\circ}}(\Fr_{w_{\circ}})[1]_{v_\circ})\otimes  (\otimes_{v \in \mathcal{P}_\ff} x_v) \otimes \left(\otimes_{ v \not \in S  \cup \mathcal{P}_{\ff}\cup \{v_\circ\}}[1]_v\right)\big)\\
\overset{\text{(Cor. \ref{relhor2})}}&{=} u(r)^{-1}   \pi_{\text{cyc}}\big( [g_{0,s}]\otimes (\sum_{\lambda \in {\cO_{{v_\circ},0}^\times}/{\cO_{{v_\circ},1}^\times}} \lambda\cdot x_{v_\circ})\otimes  (\otimes_{v \in \mathcal{P}_{\ff}} x_v) \otimes \left(\otimes_{ v \not \in S  \cup \mathcal{P}_{\ff}\cup \{v_\circ\}}[1]_v\right)\big)\\
&=u(r)^{-1} \sum_{\lambda \in {\cO_{{v_\circ},0}^\times}/{\cO_{{v_\circ},1}^\times}}  \pi_{\text{cyc}}\big( [g_{0,s}]\otimes (\lambda\cdot x_{v_\circ})\otimes  (\otimes_{v \in \mathcal{P}_\ff} x_v) \otimes  \left(\otimes_{ v \not \in S  \cup \mathcal{P}_{\ff}\cup \{v_\circ\}}[1]_v\right)\big)\\
&=u(r)^{-1}\sum_{\lambda \in {\cO_{{v_\circ},0}^\times}/{\cO_{{v_\circ},1}^\times}}     \pi_{\text{cyc}}\big(\lambda \cdot  ([g_{0,s}]\otimes  (\otimes_{v \in \mathcal{P}_\ff\cup {v_\circ}} x_v) \otimes  \left(\otimes_{ v \not \in S  \cup \mathcal{P}_{\ff}\cup \{v_\circ\}}[1]_v\right))\big)\\
\overset{\text{(Prop. \ref{galoisusingH})}}&{=}u(r)^{-1} \sum_{\lambda \in {\cO_{{v_\circ},0}^\times}/{\cO_{{v_\circ},1}^\times}}    \pi_{\text{cyc}}\big( [g_{0,s}]\otimes  (\otimes_{v \in \mathcal{P}_\ff\cup {v_\circ}} x_v) \otimes  \left(\otimes_{ v \not \in S  \cup \mathcal{P}_{\ff}\cup \{v_\circ\}}[1]_v\right)\big)^{\Art_{w_\circ}(\lambda)}\\
&=u(r)^{-1} \sum_{\lambda \in {\cO_{{v_\circ},0}^\times}/{\cO_{{v_\circ},1}^\times}}\xi_{\p_{v_\circ}\ff}^{\Art_{w_\circ}(\lambda)}\\
&=u(r)^{-1}\Tr_{\cK(\p_{v_\circ}\ff)_{w_\circ}/\cK(\ff)_{w_\circ}}\xi_{\p_{v_\circ}\ff}\\
\overset{\text{(Prop. \ref{VII11})}}&{=}\sum_{\sigma \in \Gal(\cK(\p_{v_\circ}\ff)/\cK(\ff))}\xi_{\p_{v_\circ}\ff}^{\sigma}\\
&=\Tr_{\cK(\p_{v_\circ}\ff)/\cK(\ff)}\xi_{\p_{v_\circ}\ff},
\end{align*}
where, \begin{tikzcd}\Art_{w_\circ}\colon E_{w_\circ} \arrow[r, hook] & \A_{E,f}^\times \arrow[twoheadrightarrow]{r}{\Art_E} & \Gal(E^{ab}/E).\end{tikzcd} \qed

\bibliographystyle{amsalpha}
\bibliography{Reda_library2020}

\providecommand{\bysame}{\leavevmode\hbox to3em{\hrulefill}\thinspace}
\providecommand{\MR}{\relax\ifhmode\unskip\space\fi MR }
\providecommand{\MRhref}[2]{%
  \href{http://www.ams.org/mathscinet-getitem?mr=#1}{#2}
}
\providecommand{\href}[2]{#2}
\begin{thebibliography}{Moo98b}

\bibitem[AT90]{artin-tate:cft}
E~Artin and J~Tate, \emph{{Class field theory}}, second ed., Advanced Book
  Classics, Addison-Wesley Publishing Company Advanced Book Program, Redwood
  City, CA, 1990.

\bibitem[BB66]{BB66}
W~L Baily and A~Borel, \emph{{Compactification of arithmetic quotients of
  bounded symmetric domains}}, Ann. of Math \textbf{84} (1966), 442--528.

\bibitem[BBJ16]{BBJ16}
R~Boumasmoud, E~Brooks, and D~Jetchev, \emph{{Horizontal distribution relations
  for special cycles on unitary Shimura varieties: Split case}}, {\tt
  https://arxiv.org/pdf/1611.09425.pdf} (2016).

\bibitem[BBJ18]{BBJ18}
\bysame, \emph{{Vertical Distribution Relations for Special Cycles on Unitary
  Shimura Varieties}}, International Mathematics Research Notices \textbf{2020}
  (2018), 3902--3926.

\bibitem[BFH90]{bump-friedberg-hoffstein}
D~Bump, S~Friedberg, and J~Hoffstein, \emph{{Eisenstein series on the
  metaplectic group and nonvanishing theorems for automorphic L-functions and
  their derivatives}}, Annals of Mathematics \textbf{131} (1990), no.~1,
  53--127.

\bibitem[Bor84]{Bor83}
M~Borovoi, \emph{{Langlands's conjecture concerning conjugation of connected
  Shimura varieties}}, Selecta Math. Soviet. \textbf{3} (1983/84), no.~1,
  3--39.

\bibitem[Bou19]{boumasmoud19}
R~Boumasmoud, \emph{{Generalized Norm-compatible Systems on Unitary Shimura
  Varieties}}, Ph.D. thesis, EPFL, Lausanne, August 2019.

\bibitem[Bou21a]{Vunitarynormrelations2020}
\bysame, \emph{{General Vertical Norm Compatible Systems}}, arXiv e-prints
  (2021).

\bibitem[Bou21b]{Seedrelations2020}
\bysame, \emph{{Seed Relations and Hecke polynomials}}, arXiv e-prints (2021),
  1--17.

\bibitem[Bou21c]{UoperatorsII2021}
\bysame, \emph{{The ring of $\mathbb{U}$-operators: Definitions and
  Integrality}}, arXiv e-prints (2021), 1--18.

\bibitem[BR94]{BR94}
D~Blasius and J~D Rogawski, \emph{{Zeta functions of Shimura varieties}},
  Motives (Providence, RI), vol.~55, Proceedings of Symposia in Pure
  Mathematics, no.~II, Amer. Math. Soc., 1994, pp.~525--571.

\bibitem[CGP15]{CGP10}
B~Conrad, O~Gabber, and G~Prasad, \emph{{Pseudo-reductive Groups}}, second ed.,
  New mathematical monographs, no.~26, Cambridge University Press, 2015.

\bibitem[Con14]{Conrad2014}
B~Conrad, \emph{{Reductive group schemes},}, {Autour des sch{\'e}mas en
  groupes, {\'E}cole d'{\'e}t{\'e} "Sch{\'e}mas en groupes'", Group Schemes, A
  celebration of SGA3} (Conrad~B Brochard~S and Oesterl{\'e} J, eds.), vol.~I,
  Panoramas et synth{\`e}ses, no. 42-43, Soci\'et\'e Math\'ematique de France,
  2014.

\bibitem[Cor18]{Cornut2018}
C~Cornut, \emph{{An Euler system of Heegner type}}, preprint (2018).

\bibitem[CV05]{cornut-vatsal}
C~Cornut and V~Vatsal, \emph{{CM points and quaternion algebras}}, Documenta
  Mathematica \textbf{10} (2005), 263--309 (electronic).

\bibitem[CV07]{cornut-vatsal:durham}
\bysame, \emph{{Nontriviality of Rankin-Selberg L-functions and CM points}},
  L-functions and Galois representations, Cambridge Univ. Press, Cambridge,
  2007, pp.~121--186.

\bibitem[Dar04]{darmon:ratpoints}
Henri Darmon, \emph{{Rational points on modular elliptic curves}}, CBMS
  Regional Conference Series in Mathematics, vol. 101, Published for the
  Conference Board of the Mathematical Sciences, Washington, DC, 2004.

\bibitem[Del71]{deligne:travaux}
P~Deligne, \emph{{Travaux de Shimura}}, S\'eminaire Bourbaki, 23\`eme ann\'ee
  (1970/71), Expos\'{e} 389, Springer, Berlin, F\'{e}vrier 1971, pp.~123--165.
  Lecture Notes in Math., Vol. 244.

\bibitem[Del79]{deligne:shimura}
\bysame, \emph{{Vari\'et\'es de Shimura: interpr\'etation modulaire, et
  techniques de construction de mod\`eles canoniques}}, Automorphic forms,
  representations and L-functions (Proc. Sympos. Pure Math., Oregon State
  Univ., Corvallis, Ore., 1977), Part 2, Amer. Math. Soc., Providence, R.I.,
  1979, pp.~247--290.

\bibitem[GGP09]{gan-gross-prasad}
W~T Gan, B~Gross, and D~Prasad, \emph{{Symplectic local root numbers, central
  critical L-values, and restriction problems in the representation theory of
  classical groups}}, preprint (2009).

\bibitem[GH19]{GetHah2019}
R~J Getz and H~Hahn, \emph{{An Introduction to Automorphic Representations with
  a view toward Trace Formulae}}, GTM Series, Springer, To be published in
  2019.

\bibitem[Gol99]{Goldman99}
W~Goldman, \emph{{Complex Hyperbolic Geometry}}, Oxford Science Publications,
  1999.

\bibitem[Gro91]{gross:kolyvagin}
B~H Gross, \emph{{Kolyvagin's work on modular elliptic curves}}, L-functions
  and arithmetic (Durham, 1989), Cambridge Univ. Press, Cambridge, 1991,
  pp.~235--256.

\bibitem[GZ86]{gross-zagier}
B~Gross and D~Zagier, \emph{{Heegner points and derivatives of L-series}},
  Inventiones mathematicae \textbf{84} (1986), no.~2, 225--320.

\bibitem[How04]{howard:heeg}
B~Howard, \emph{{The Heegner point Kolyvagin system}}, Compositio Mathematica
  \textbf{140} (2004), no.~6, 1439--1472.

\bibitem[Jet16]{jetchev:unitary}
D~Jetchev, \emph{{Hecke and Galois properties of special cycles on unitary
  Shimura varieties}}, {\tt http://arxiv.org/pdf/1410.6692.pdf} (2016).

\bibitem[JNS18]{jetchev-nekovar-skinner}
D~Jetchev, J~Nekov\'a\v{r}, and C~Skinner, \emph{{Split Kolyvagin Systems for
  Conjugate-Dual Galois Representations}}, preprint available upon request
  (2018).

\bibitem[Kat04]{kato04}
K~Kato, \emph{{p-adic Hodge theory and values of zeta functions of modular
  forms}}, Ast\'erisque (2004), no.~295, ix, 117--290.

\bibitem[Kim18]{Wan2018}
W~Kim, \emph{{Rapoport--Zink Uniformization of Hodge-type Shimura varieties}},
  Forum of Mathematics, Sigma \textbf{6} (2018), no.~e16, 36.

\bibitem[KLZ15]{KLZ15}
G~Kings, D~Loeffler, and S~L Zerbes, \emph{{Rankin--Eisenstein classes for
  modular forms}}, to appear in American J. Math. (2015).

\bibitem[KLZ17]{KLZ17}
\bysame, \emph{{Rankin-Eisenstein classes and explicit reciprocity laws}},
  Camb. J. Math. \textbf{5} (2017), no.~2, 1--122.

\bibitem[Kol90]{kolyvagin:euler_systems}
V~A Kolyvagin, \emph{{Euler systems}}, The Grothendieck Festschrift, Vol.\ II,
  Birkh\"auser Boston, Boston, MA, 1990, pp.~435--483.

\bibitem[Kol91a]{kolyvagin:mordellweil}
\bysame, \emph{{On the Mordell-Weil group and the Shafarevich-Tate group of
  modular elliptic curves}}, Proceedings of the International Congress of
  Mathematicians, Vol.\ I, II (Kyoto, 1990) (Tokyo), Math. Soc. Japan, 1991,
  pp.~429--436.

\bibitem[Kol91b]{kolyvagin:structure_of_selmer}
\bysame, \emph{{On the structure of Selmer groups}}, Mathematische Annalen
  \textbf{291} (1991), no.~2, 253--259.

\bibitem[Kol91c]{kolyvagin:structureofsha}
\bysame, \emph{{On the structure of Shafarevich-Tate groups}}, Algebraic
  geometry (Chicago, IL, 1989), Springer, Berlin, 1991, pp.~94--121.

\bibitem[Kot84]{Ko1}
R~E Kottwitz, \emph{{Shimura Varieties and Twisted Orbital Integrals}},
  Mathematische Annalen \textbf{269} (1984), 287--300.

\bibitem[Lan56]{lang:finitefields}
S~Lang, \emph{{Algebraic groups over finite fields}}, Amer. J. Math.
  \textbf{78} (1956), 555--563.

\bibitem[LLZ14]{LLZ14}
A~Lei, D~Loeffler, and S~L Zerbes, \emph{{Euler systems for Rankin--Selberg
  convolutions of modular forms}}, Ann. of Math \textbf{180} (2014), no.~2,
  653--771.

\bibitem[LLZ17]{LLZ17}
\bysame, \emph{{Euler systems for Hilbert modular surfaces}}, preprint (2017).

\bibitem[LSZ17]{LSZ17}
D~Loeffler, C~Skinner, and S~L Zerbes, \emph{{Euler systems for
  $\mathbf{Gsp}(4)$}}, preprint (2017).

\bibitem[Mil80]{milne:etale}
J~S Milne, \emph{{\'Etale cohomology}}, Princeton University Press, Princeton,
  N.J., 1980.

\bibitem[Mil83]{Milne1983}
\bysame, \emph{{The Action of an Automorphism of $\C$ on a Shimura Variety and
  its Special Points}}, pp.~239--265, Birkh{\"a}user Boston, Boston, MA, 1983.

\bibitem[Mil90]{Milne90}
\bysame, \emph{{Canonical models of (mixed) Shimura varieties and automorphic
  vector bundles}}, Automorphic forms, Shimura varieties, and Lfunctions
  (L.~Clozel and J.~S. Milne, eds.), Persp. in Math., vol. Vol. 10 (I),
  Academic Press Inc., 1990, pp.~283--414.

\bibitem[Mil99]{Milne:1999}
\bysame, \emph{{Descent for Shimura varieties}}, Michigan Math. J. \textbf{46}
  (1999), no.~1, 203--208.

\bibitem[Mil17]{milne:shimura}
\bysame, \emph{{Introduction to Shimura varieties}}, www.jmilne.org (2017),
  1--172.

\bibitem[MM97]{murty-murty}
M~R Murty and V~K Murty, \emph{{Non-vanishing of L-functions and
  applications}}, Birkh\"auser Verlag, Basel, 1997.

\bibitem[Moo98a]{linearityI}
B~Moonen, \emph{{Linearity of Shimura varieties. I}}, Journal of Algebraic
  Geometry (1998), no.~7, 539--567.

\bibitem[Moo98b]{Moonen98}
\bysame, \emph{{Models of Shimura varieties in mixed characteristics}}, {Galois
  Representations in Arithmetic Algebraic Geometry}, vol. 256, London
  Mathematical Society Lecture Note Series, no. 254, Cambridge Univ. Press,
  1998, pp.~267--350.

\bibitem[Nek07]{Nekovar2007}
J~Nekov\'a\v{r}, \emph{{The Euler system method for CM points on Shimura
  curves}}, London Mathematical Society Lecture Note Series, pp.~471--547,
  Cambridge University Press, Cambridge, 2007.

\bibitem[Neu86]{Neukirch1986}
J~Neukirch, \emph{{Class field theory}}, {Grundlehren der Mathematischen
  Wissenschaften}, vol. 280, Springer-Verlag, Berlin, Heidelberg, 1986.

\bibitem[NSW08]{NSW2008}
J~Neukirch, A~Schmidt, and K~Wingberg, \emph{{Cohomology of number fields}},
  second ed., Grundlehren der Mathematischen Wissenschaften [Fundamental
  Principles of Mathematical Sciences], vol. 323, Springer-Verlag, 2008.

\bibitem[Pin89]{Pink89}
R~Pink, \emph{Arithmetical compactification of mixed shimura varieties}, no.
  209, Bonner Mathematische Schriften, 1989.

\bibitem[Shi08]{Shimura08}
G~Shimura, \emph{{Arithmetic of Hermitian forms}}, Documenta Mathematica
  \textbf{13} (2008), 739--774.

\bibitem[Zha01]{zhang:gross-zagier}
S~W Zhang, \emph{{Gross-Zagier formula for GL$_{2}$}}, Asian J. Math.
  \textbf{5} (2001), no.~2, 183--290.

\end{thebibliography}

\end{document}